\RequirePackage[l2tabu, orthodox]{nag}

\documentclass[
11pt,                          
english                        
]{article}

\usepackage[english]{babel}    
\usepackage{amsmath}           
\usepackage[utf8]{inputenc}    
\usepackage[T1]{fontenc}       
\usepackage{longtable}         
\usepackage{exscale}           
\usepackage[final]{graphicx}   
\usepackage[sort]{cite}        
\usepackage{array}             
\usepackage{wasysym}           
\usepackage[a4paper]{geometry} 
\usepackage[multiuser]{fixme}  
\usepackage{xspace}            
\usepackage{tikz}              
\usepackage{paralist}
\usepackage{amssymb}           
\usepackage{bbm}               
\usepackage{mathtools}         
\usepackage{aliascnt}          
\usepackage[amsmath,thmmarks,hyperref]{ntheorem} 
\usepackage{mathrsfs}          
\usepackage{stmaryrd}          
\usepackage{ifdraft}           
\usepackage[expansion=false    
           ]{microtype}        
\usepackage[nottoc]{tocbibind} 
\usepackage[backref=page,      
           final=true,         
           hypertexnames=false,
           plainpages=false,              
           pdfpagelabels=true,            
           pdfencoding=auto,              
           unicode=true                   
                      ]{hyperref}         

\ifdraft{\synctex=1}{}

\fxusetheme{color}

\FXRegisterAuthor{ms}{anms}{matthias}
\FXRegisterAuthor{sw}{answ}{stefan}

\author{
  \textbf{Matthias Schötz}\thanks{\texttt{matthias.schoetz@mathematik.uni-wuerzburg.de}},
  \addtocounter{footnote}{2}
  \textbf{Stefan Waldmann}\thanks{\texttt{stefan.waldmann@mathematik.uni-wuerzburg.de}}
  \\[0.5cm]
  Institut für Mathematik \\
  Lehrstuhl für Mathematik X \\
  Universität Würzburg \\
  Campus Hubland Nord \\
  Emil-Fischer-Straße 31 \\
  97074 Würzburg \\
  Germany
}

\makeatletter

\renewcommand{\mathbb}[1]{\mathbbm{#1}}

\numberwithin{equation}{section}

\allowdisplaybreaks

\renewcommand{\arraystretch}{1.2}

\let\originalleft\left
\let\originalright\right
\renewcommand{\left}{\mathopen{}\mathclose\bgroup\originalleft}
\renewcommand{\right}{\aftergroup\egroup\originalright}

\renewcommand{\cleardoublepage}{\clearpage\ifodd\c@page\else\vspace*{\fill}\thispagestyle{empty}\newpage\fi}

\newcommand{\lemmachairxname}{Lemma}
\newcommand{\propositionchairxname}{Proposition}
\newcommand{\theoremchairxname}{Theorem}
\newcommand{\corollarychairxname}{Corollary}
\newcommand{\definitionchairxname}{Definition}
\newcommand{\claimchairxname}{Claim}
\newcommand{\examplechairxname}{Example}
\newcommand{\remarkchairxname}{Remark}
\newcommand{\questionchairxname}{Question}
\newcommand{\conjecturechairxname}{Conjecture}
\newcommand{\exercisechairxname}{Exercise}
\newcommand{\maintheoremchairxname}{Main Theorem}
\newcommand{\notationchairxname}{Notation}
\newcommand{\proofchairxname}{Proof}
\newcommand{\subproofchairxname}{Proof}
\newcommand{\hintchairxname}{Hint}

\StartBabelCommands{english}{extras}
\SetString{\lemmachairxname}{Lemma}
\SetString{\propositionchairxname}{Proposition}
\SetString{\theoremchairxname}{Theorem}
\SetString{\corollarychairxname}{Corollary}
\SetString{\definitionchairxname}{Definition}
\SetString{\claimchairxname}{Claim}
\SetString{\examplechairxname}{Example}
\SetString{\remarkchairxname}{Remark}
\SetString{\questionchairxname}{Question}
\SetString{\conjecturechairxname}{Conjecture}
\SetString{\exercisechairxname}{Exercise}
\SetString{\maintheoremchairxname}{Main Theorem}
\SetString{\notationchairxname}{Notation}
\SetString{\proofchairxname}{Proof}
\SetString{\subproofchairxname}{Proof}
\SetString{\hintchairxname}{Hint}
\EndBabelCommands

\StartBabelCommands{german}{extras}
\SetString{\lemmachairxname}{Lemma}
\SetString{\propositionchairxname}{Proposition}
\SetString{\theoremchairxname}{Satz}
\SetString{\corollarychairxname}{Korollar}
\SetString{\definitionchairxname}{Definition}
\SetString{\claimchairxname}{Behauptung}
\SetString{\examplechairxname}{Beispiel}
\SetString{\remarkchairxname}{Bemerkung}
\SetString{\questionchairxname}{Frage}
\SetString{\conjecturechairxname}{Vermutung}
\SetString{\exercisechairxname}{Übung}
\SetString{\maintheoremchairxname}{Theorem}
\SetString{\notationchairxname}{Notation}
\SetString{\proofchairxname}{Beweis}
\SetString{\subproofchairxname}{Beweis}
\SetString{\hintchairxname}{Hinweis}
\EndBabelCommands

\theoremheaderfont{\normalfont\bfseries}
\theorembodyfont{\itshape}
\newtheorem{lemma}{\lemmachairxname}[section]

\newtheorem{proposition}[lemma]{\propositionchairxname}
\newtheorem{theorem}[lemma]{\theoremchairxname}
\newtheorem{corollary}[lemma]{\corollarychairxname}
\newtheorem{definition}[lemma]{\definitionchairxname}

\theorembodyfont{\rmfamily}

\newtheorem{remark}[lemma]{\remarkchairxname}

\theorembodyfont{\itshape}
\theoremnumbering{Roman}

\def\theorem@checkbold{}

\theoremheaderfont{\scshape}
\theorembodyfont{\normalfont}
\theoremstyle{nonumberplain}
\theoremseparator{:}
\theoremsymbol{\hbox{$\boxempty$}}
\newtheorem{proof}{\proofchairxname}
\theoremsymbol{\hbox{$\triangledown$}}

\newenvironment{lemmalist}{\begin{compactenum}[\itshape i.)]}{\end{compactenum}}

\newenvironment{propositionlist}{\begin{compactenum}[\itshape i.)]}{\end{compactenum}}
\newenvironment{definitionlist}{\begin{compactenum}[\itshape i.)]}{\end{compactenum}}

\newcommand{\I}              {\mathrm{i}}
\newcommand{\E}              {\mathrm{e}}
\newcommand{\D}              {\mathop{}\!\mathrm{d}}
\newcommand{\cc}[1]          {\overline{{#1}}}

\newcommand{\at}[1]          {\big|_{#1}}
\newcommand{\At}[1]          {\Big|_{#1}}

\newcommand{\argument}       {\,\cdot\,}

\newcommand{\id}             {\mathsf{id}}
\newcommand{\pr}             {\mathrm{pr}}

\newcommand{\algebra}[1]      {\mathscr{#1}}

\DeclarePairedDelimiter{\abs}{\lvert}{\rvert}
\DeclarePairedDelimiter{\norm}{\lVert}{\rVert}
\newcommand{\@supnormstar}[1]{\norm*{#1}_\infty}
\newcommand{\@supnormnostar}[2][]{\norm[#1]{#2}_\infty}
\newcommand{\supnorm}{\@ifstar\@supnormstar\@supnormnostar}

\newcommand{\Stetig}         {\mathscr{C}}

\DeclareMathOperator*{\esssup}  {\mathrm{ess\,\sup}}
\newcommand{\@esssupnormstar}[1]{\norm*{#1}_{\esssup}}
\newcommand{\@esssupnormnostar}[2][]{\norm[#1]{#2}_{\esssup}}
\newcommand{\esssupnorm}{\@ifstar\@esssupnormstar\@esssupnormnostar}

\DeclareMathSymbol\dAlembert     {\mathord}{AMSa}{"03}

\DeclareMathOperator{\spann} {\mathrm{span}}

\DeclareFontFamily{U}{FdSymbolF}{}
\DeclareFontShape{U}{FdSymbolF}{m}{n}{
    <-7.1> FdSymbolF-Book
    <7.1-> FdSymbolF-Book
}{}

\DeclareSymbolFont{delimiters}{U}{FdSymbolF}{m}{n}
\DeclareMathDelimiter{\llangle}{\mathopen}{delimiters}{"92}{delimiters}{"92}
\DeclareMathDelimiter{\rrangle}{\mathclose}{delimiters}{"98}{delimiters}{"98}

\DeclarePairedDelimiter{\coolIP}{\llangle}{\rrangle}

\DeclarePairedDelimiter{\ordinaryIP}{\langle}{\rangle}

\newcommand{\skal}[3][]{\ordinaryIP[#1]{#2 \,#1|\, #3}}

\newcommand{\ptwskal}[3][]{\coolIP[#1]{#2 \,#1|\, #3}}

\newcommand{\seminorm}[3][]{\norm[#1]{#3}_{#2}}

\newcommand{\Ten}{\mathcal{T}}
\newcommand{\alg}{\textup{alg}}
\newcommand{\CC}{\mathbb{C}} 
\newcommand{\RR}{\mathbb{R}}
\newcommand{\NN}{\mathbb{N}}
\renewcommand{\P}{\mathrm{P}}
\newcommand{\cpl}{{\textup{cpl}}}
\newcommand{\SymTen}{\mathcal{S}}
\newcommand{\SymFkt}{\mathscr{S}}
\newcommand{\SymGrp}{\mathfrak{S}}
\newcommand{\stpro}[1][]{\,{\star_{\!#1}}\,}
\newcommand{\per}{\textup{per}}

\newcommand{\cs}{{\textup{\tiny{CS}}}}
\newcommand{\multotimes}{\mu_\otimes}
\newcommand{\multvee}{\mu_\vee}
\newcommand{\multstpro}[1][]{\mu_{\stpro[#1]}}
\newcommand{\Bil}{\mathfrak{Bil}}

\newcommand{\analytic}{\Stetig^{\omega_{HS}}}
\newcommand{\smooth}{\Stetig^{\infty}}
\newcommand{\Kern}{\textup{ker}}
\newcommand{\hilb}{\mathscr{H}}

\makeatother

\title{Convergent Star Products for Projective Limits of Hilbert Spaces}

\date{April 2017}

\begin{document}

%
%

\maketitle

%
%

\begin{abstract}
    Given a locally convex vector space with a topology induced by
    Hilbert seminorms and a continuous bilinear form on it we
    construct a topology on its symmetric algebra such that the usual
    star product of exponential type becomes continuous. Many
    properties of the resulting locally convex algebra are
    explained. We compare this approach to various other discussions
    of convergent star products in finite and infinite dimensions. We
    pay special attention to the case of a Hilbert space and to
    nuclear spaces.
\end{abstract}

\newpage

%
%

\tableofcontents

%
%

\section{Introduction}
\label{sec:Introduction}

The canonical commutation relations
\begin{equation*}
    QP - PQ = \I \hbar
\end{equation*}
are the paradigm of quantum physics. They indicate the transition from
formerly commutative algebras of observables in classical mechanics to
now non-commutative algebras, those generated by the fundamental
variables of position $Q$ and momentum $P$. While this basic form of
the commutation relations is entirely algebraic, the need of physics is
to have some more analytic framework. Traditionally, one views $Q$ and
$P$ as (necessarily unbounded) self-adjoint operators on a Hilbert
space. Then the commutation relation becomes immediately much more
touchy as one has to take care of domains. Ultimately, the reasonable
way to handle these difficulties is to use the Schrödinger
representation which leads to a strongly continuous representation of
the Heisenberg group. This way, the commutation relations encode an
integration problem, namely from the infinitesimal picture of a Lie
algebra representation by unbounded operators to the global picture of
a group representation by unitary operators.

While this is all well-understood, things become more interesting in
infinite dimensions: here one still has canonical commutation
relations now based on a symplectic (or better: Poisson) structure on
a vector space $V$. Physically, infinite dimensions correspond to a
field theory with infinitely many degrees of freedom instead of a
mechanical system. Then, algebraically, the commutation relations can
be realized as a star product for the symmetric algebra over this
vector space, see the seminal paper \cite{bayen.et.al:1978a} where the
basic notions of deformation quantization have been introduced as well
as e.g.\cite{waldmann:2007a, gutt:2000a, weinstein:1995a} for
introductions. However, beyond the algebraic questions one is again
interested in an analytic context: it turns out that now things are
much more involved. First, there is no longer an essentially unique
way to represent the canonical commutation relations by operators, a
classical result which can be stated in many ways. One way to approach
this non-uniqueness is now to focus first on the algebraic part and
discuss the whole representation theory of this quantum algebra. To
make this possible one has to go beyond the symmetric algebra and
incorporate suitable completions instead.

Based on a $C^*$-algebraic formulation there are (at least) two
approaches available. The classical one is to take formal exponentials
of the unbounded quantities and implement a $C^*$-norm for the algebra
they generate, see \cite{manuceau:1968a}. More recently, an
alternative was proposed by taking formal resolvents and the
$C^*$-algebra they generate\cite{buchholz.grundling:2008a}. These two
approaches can be formulated in arbitrary dimensions and are used
extensively in quantum field theory.

Only for finite dimensions there is a third $C^*$-algebraic way based
on (strict) deformation quantization in the framework of Rieffel
\cite{rieffel:1993a}, see also
\cite{dubois-violette.kriegl.maeda.michor:2001a, forger.paulino:2016a}
for some more recent development: here one constructs a rather large
$C^*$-algebra by deforming the bounded continuous functions on the
underlying symplectic vector space. The deformation is based on
certain oscillatory integrals which is the reason that this approach,
though extremely appealing and powerful, will be restricted to finite
dimensions. Nevertheless, in such finite-dimensional situations one
has even ways to go beyond the flat situation and include also much
more non-trivial geometries of the underlying geometric system, see
e.g. \cite{bieliavsky.gayral:2015a}.  Unfortunately, none of those
techniques carry over to infinite dimensions.

While the $C^*$-algebraic approaches are very successful in many
aspects, some questions seem to be hard to answer within this
framework: from a deformation quantization point of view it is not
completely obvious in which sense these algebras provide deformations
of their classical counterparts, see, however,
\cite{binz.honegger.rieckers:2004b}. Closely related is the question
of how one can get back the analogs of the classically unbounded
quantities like polynomials on the symplectic vector space: in the
quantum case they can not be elements of any $C^*$-algebra and thus
they have to be recovered in certain well-behaved representations as
unbounded operators on the representation space. This raises the
question whether they can acquire some intrinsic meaning, independent
of a chosen representation. In particular, all the $C^*$-algebraic
constructions completely ignore possible additional structures on the
underlying vector space $V$, like e.g. a given topology. This seems
both from the purely mathematical but also from the physical point of
view rather unpleasant.

In \cite{waldmann:2014a} a first step was taken to overcome some of
these difficulties: instead of considering a $C^*$-algebraic
construction, the polynomials, modeled as the symmetric algebra, were
kept and quantized by means of a star product directly. Now the
additional feature is that a given locally convex topology on the
underlying vector space $V$ induces a specific locally convex topology
on the symmetric algebra $\SymTen^\bullet(V)$ in such a way that the
star product becomes \emph{continuous}. Necessarily, there will be no
non-trivial sub-multiplicative seminorms, making the whole locally
convex algebra quite non-trivial. It was then shown that in the
completion the star product is a convergent series in the deformation
parameter $\hbar$. This construction has good functorial properties
and works for every locally convex space $V$ with continuous constant
Poisson structure. The basic feature was that on a fixed symmetric
power $\SymTen^k(V)$ the \emph{projective} locally convex topology was
chosen. In finite dimensions this construction reproduces earlier
versions \cite{omori.maeda.miyazaki.yoshioka:2000a,
  omori.maeda.miyazaki.yoshioka:2007a} of convergence results for the
particular case of the Weyl-Moyal star product.

In the present paper we want to adapt the construction of
\cite{waldmann:2014a} to the more particular case of a projective
limit of (pre-) Hilbert spaces, i.e. a locally convex space where the
topology is determined by Hilbert seminorms coming from (not
necessarily non-degenerate) positive inner products. The major
difference is now that for each fixed symmetric power $\SymTen^k(V)$
we have another choice of the topology, namely the one by extending
the inner products first and taking the corresponding Hilbert
seminorms afterwards. In general, this is coarser than the projective
one and thus yields a larger and hence more interesting completion. We
then use a star product coming from an arbitrary continuous bilinear
form on $V$, thereby allowing for various other orderings beside the
usual Weyl symmetrization. We are able to determine many features of
this new algebra hosting the canonical commutation relations in
arbitrary dimensions, including the convergence of the star product
and an explicit description of the completion as certain analytic
functions on the topological dual.

The paper is organized as follows: in
Section~\ref{sec:ConstructionAlgebra} we outline the construction of
the star product and the relevant topology. Since the star product is
the usual one of exponential type on a vector space we can be brief
here. The topological properties are discussed in some detail, in
particular as they differ at certain points significantly from the
previous work \cite{waldmann:2014a}. After the necessary but technical
estimates this results in the construction of the locally convex
algebra in Theorem~\ref{theorem:stprocont}.

Section~\ref{sec:PropertiesStarProduct} contains various properties of
the star-product algebra. First we show that a continuous antilinear
involution on $V$ extends to a continuous $^*$-involution on the
algebra. Then we are able to characterize the topology by some very
simple conditions in Theorem~\ref{theo:char}, a feature which is
absent in the case of \cite{waldmann:2014a}. The discussion of
equivalences between different star products becomes now more involved
as not all continuous symmetric bilinear forms give rise to
equivalences as that was the case in \cite{waldmann:2014a}. Now in
Theorem~\ref{theo:equivalence} we have to add a Hilbert-Schmidt
condition similar to the one of Dito in \cite{dito:2005a}. In
Theorem~\ref{theo:gelfand} we are able to characterize the completed
star-product algebra as certain analytic functions on the topological
dual. This will later be used to make contact to the more particular
situation considered in \cite{dito:2005a}. In
Theorem~\ref{theorem:posex} we show the existence of many positive
linear functionals provided the Poisson tensor allows for a compatible
positive bilinear form of Hilbert-Schmidt type.  Since the algebra is
(necessarily) not locally multiplicatively convex, we have no general
entire calculus. However, we can show that for elements of degree one,
i.e. vectors in $V$, the star-exponential series converges absolutely.
This is no longer true for quadratic elements, i.e. elements in
$\SymTen^2(V)$. However, we are able to show that in all GNS
representations with respect to continuous positive linear functionals
all elements up to quadratic ones yield essentially self-adjoint
operators in Theorem~\ref{theo:essselfadj}. Here our topology is used
in an essential way. The statement can be seen as a
representation-independent version of Nelson's theorem, as it holds
for arbitrary such GNS representations.

Finally, Section~\ref{sec:SpecialCasesExamples} contains a discussion
of two particular cases of interest: First, we consider the case that
$V$ is not just a projective limit of Hilbert spaces but a Hilbert
space directly. In this case, Dito discussed formal star products of
exponential type and their formal equivalence in \cite{dito:2005a}. We
can show that his algebra of functions contains our algebra, where the
star product converges nicely, as a subalgebra. In this sense, we
extend Dito's results form the formal power series context to a
convergent one. In fact, we show a rather strong continuity with
respect to the deformation parameter in
Theorem~\ref{theorem:ContinuityOnLambda}.

The second case is a nuclear space $V$. It is well-known that any
(complete) nuclear space can be seen as a projective limit of Hilbert
spaces, see e.g. \cite[Cor.~21.2.2]{jarchow:1981a}. Not very
surprisingly, we prove that in this case our construction coincides
with the previous one of \cite{waldmann:2014a} as for nuclear spaces
the two competing notions of topological tensor products we use
coincide. This way we can transfer the abstract characterization of
the topology to the case of nuclear spaces in \cite{waldmann:2014a}, a
result which was missing in that approach.  The important benefit from
the projective Hilbert space point of view is now that we can show the
existence of sufficiently many continuous positive linear functionals:
an element in the completed $^*$-algebra is zero iff all continuous
positive functionals on it vanish. It follows that the resulting
$^*$-algebra has a faithful $^*$-representation on a pre-Hilbert
space, i.e. it is $^*$-semisimple in the sense of
\cite{schmuedgen:1990a}.

\noindent
\textbf{Notation:} For a set $X$ and $k \in \NN_0$ we define $X^k$ as
the set of all functions from $\{1, \ldots, k\}$ (or the empty set if
$k=0$) to $X$ and usually put the parameter in the index, i.e.
$\{1, \ldots, k\} \ni i \mapsto f_i \in X$ for $f \in X^k$. Let $V$ be
a vector space and $k \in \NN_0$, then we write $\Ten^k_\alg(V)$ for
the space of degree $k$-tensors over $V$ and
$\Ten^\bullet_\alg(V) := \bigoplus_{k\in \NN_0} \Ten^k_\alg(V)$ for the
vector space underlying the tensor algebra. For $x \in V^k$ we define
the projections on the tensors of degree $k$ by
$\langle \argument\rangle_k\colon \Ten^\bullet_\alg(V) \rightarrow
\Ten^k_\alg(V)$.
Let $\SymGrp_k \subseteq \{1, \ldots, k\}^k$ be the symmetric group of
degree $k$ (in the case $k = 0$ this is
$\SymGrp_0 = \{\id_\emptyset\}$), then $\SymGrp_k$ acts linearly on
$\Ten^k_\alg(V)$ from the right via
$(x_1 \otimes \cdots \otimes x_k)^\sigma := x_{\sigma(1)} \otimes
\cdots \otimes x_{\sigma(k)}$.
This allows us to define the symmetrisation operators
$\SymFkt^k\colon \Ten^k_\alg(V) \rightarrow \Ten^k_\alg(V)$ by
$X \mapsto \SymFkt^k(X) := \frac{1}{k!}\sum_{\sigma\in\SymGrp_k}
X^\sigma$
and
$\SymFkt^\bullet\colon \Ten^\bullet_\alg(V) \rightarrow
\Ten^\bullet_\alg(V)$
by
$X \mapsto \SymFkt^\bullet(X) := \sum_{k\in \NN_0}
\SymFkt^k\big(\langle X\rangle_k\big)$.
These are projectors on subspaces of $\Ten^k_\alg(V)$ and
$\Ten^\bullet_\alg(V)$ which we will denote by $\SymTen^k_\alg(V)$ and
$\SymTen^\bullet_\alg(V)$.

We will always denote an algebra as a pair $(V, \circ)$ of a vector
space $V$ and a multiplication $\circ$, because we will discuss
different products on the same vector space.

%
%

\noindent
\textbf{Acknowledgements:} We would like to thank Philipp Schmitt for
useful comments on an early version of the manuscript.

%
%

\section{Construction of the Algebra}
\label{sec:ConstructionAlgebra}

As we want to construct a similar algebra like in
\cite{waldmann:2014a}, but by using Hilbert tensor products instead of
projective tensor products, we have to restrict our attention to
locally convex spaces whose topology is given by Hilbert seminorms.

Let $V$ be a locally convex space, then a positive Hermitian form on
$V$ is a sesquilinear Hermitian and positive semi-definite form
$\skal{\argument}{\argument}_\alpha\colon V\times V \rightarrow \CC$
(antilinear in the first, linear in the second argument). By
$\mathcal{I}_V$ we denote the set of all continuous positive Hermitian
forms on $V$ and we will distinguish different positive Hermitian
forms by a lowercase greek subscript. Out of $p,q\ge 0$ and
$\skal{\argument}{\argument}_\alpha,\skal{\argument}{\argument}_\beta
\in \mathcal{I}_V$ we get a new continuous positive Hermitian form
$\skal{\argument}{\argument}_{p\alpha + q\beta} := p
\skal{\argument}{\argument}_\alpha + q
\skal{\argument}{\argument}_\beta$.

Every $\skal{\argument}{\argument}_\alpha \in \mathcal{I}_V$ yields a
continuous Hilbert seminorm on $V$, defined as $\seminorm{\alpha}{v}
:= \sqrt{\skal{v}{v}_\alpha}$ for all $v\in V$.  The set of all
continuous Hilbert seminorms on $V$ will be denoted by
$\mathcal{P}_V$. Note that $\seminorm{p\alpha+q\beta}{\argument} =
\big(q \seminorm{\alpha}{\argument}^2 + p
\seminorm{\beta}{\argument}^2 \big)^{1/2}$ and that
$\mathcal{P}_V$ with the usual partial ordering of seminorms
(i.e. by pointwise comparison) is an upwards directed poset and that
there is a one-to-one corres\-pondence between $\mathcal{I}_V$ and
$\mathcal{P}_V$ due to the polarisation identity.

In the following we will always assume that $V$ is a Hausdorff locally
convex space whose topology is given by its continuous Hilbert
seminorms (``hilbertisable'' in the language of \cite{jarchow:1981a}),
i.e. we assume that $\mathcal{P}_V$ is cofinal in the upwards directed
set of all continuous seminorms on $V$.  Important examples of such
spaces are (pre-) Hilbert spaces and nuclear spaces (see
\cite[Corollary 21.2.2]{jarchow:1981a}) and, in general, all
projective limits of pre-Hilbert spaces in the category of locally
convex spaces.

%
%

\subsection{Extension of Hilbert Seminorms to the Tensor Algebra}

Analogous to \cite{waldmann:2014a}, we extend all Hilbert seminorms
from $V$ to $\Ten^\bullet_\alg(V)$ with the difference that we first
extend the $\skal{\argument}{\argument}_\alpha$ and reconstruct the
seminorms out of their extensions:
\begin{definition}
    \label{defi:extension}%
    For every continuous positive Hermitian form
    $\skal{\argument}{\argument}_\alpha \in \mathcal{I}_V$ we define the
    sesquilinear extension
    $\skal{\argument}{\argument}^{\bullet}_\alpha\colon
    \Ten^\bullet_\alg(V)\times \Ten^\bullet_\alg(V) \rightarrow \CC$
    \begin{equation}
        \label{eq:skalBulletDef}
        (X, Y)
        \mapsto
        \skal{X}{Y}^{\bullet}_\alpha
        :=
        \sum_{k=0}^\infty
        \skal[\big]
        {\langle X \rangle_k}{\langle Y \rangle_k}^{\bullet}_\alpha,
    \end{equation}
    where
    \begin{equation}
        \label{eq:semiinnerproductextension1}
        \skal[\big]
        {x_1 \otimes \cdots \otimes x_k}
        {y_1 \otimes \cdots \otimes y_k}^{\bullet}_\alpha
        :=
        k! \prod_{m=1}^k
        \skal{x_m}{y_m}_\alpha
    \end{equation}
    for all $k \in \NN_0$ and all $x, y \in V^k$.
\end{definition}

It is well-known that this is a positive Hermitian form on all
homogeneous tensor spaces and then it is clear that
$\skal{\argument}{\argument}^\bullet_\alpha$ is a positive Hermitian
form on $\Ten^\bullet_\alg(V)$. We write
$\seminorm{\alpha}{\argument}^\bullet$ for the resulting seminorm on
$\Ten^\bullet_\alg(V)$ and $\Ten^\bullet(V)$ for the locally convex
space of $\Ten^\bullet_\alg(V)$ with the topology defined by the
extensions of all
$\seminorm{\alpha}{\argument}\in\mathcal{P}_V$. Analogously, we write
$\Ten^k(V)$, $\SymTen^k(V)$ and $\SymTen^\bullet(V)$ for the subspaces
$\Ten^k_\alg(V)$, $\SymTen^k_\alg(V)$ and $\SymTen^\bullet_\alg(V)$
with the subspace topology.  Note that
$\seminorm{\alpha}{\argument}^\bullet \le
\seminorm{\beta}{\argument}^\bullet$ holds if and only if
$\seminorm{\alpha}{\argument} \le \seminorm{\beta}{\argument}$.  Note
that, in general, for a fixed tensor degree the resulting topology on
$\Ten^k(V)$ is \emph{not} the projective topology used in
\cite{waldmann:2014a}.


The factor $k!$ in \eqref{eq:semiinnerproductextension1} for the
extensions of positive Hermitian forms corresponds to the factor
$(n!)^R$ for $R=1/2$ in \cite[Eq.~(3.7)]{waldmann:2014a} for the
extensions of seminorms (where $R=1/2$ yields the coarsest topology
for which the continuity of the star-product could be shown in
\cite{waldmann:2014a}). We are only interested in this special case
because of the characterization in Section~\ref{subsec:char}.

The following is an easy consequence of the definition of the topology
on $\Ten^\bullet(V)$:
\begin{proposition}
    $\Ten^\bullet(V)$ is Hausdorff and is metrizable if and only
    if $V$ is metrizable.
\end{proposition}

For working with these extensions of not necessarily positive definite
positive Hermitian forms, the following technical lemma will be
helpful:
\begin{lemma}
    \label{lemma:helpfull}%
    Let $\skal{\argument}{\argument}_\alpha \in \mathcal{I}_V$,
    $k\in \NN$ and $X\in \Ten^k(V)$ be given. Then $X$ can be
    expressed as $X = X_0 + \tilde{X}$ with tensors
    $X_0, \tilde{X} \in \Ten^k(V)$ that have the following properties:
    \begin{lemmalist}
    \item One has $\seminorm{\alpha}{X_0}^\bullet = 0$ and there
        exists a finite (possibly empty) set $A$ and tuples $x_a \in
        V^k$ for all $a \in A$ that fulfill $\prod_{n=1}^k
        \seminorm{\alpha}{x_{a,n}} = 0$ and $X_0 = \sum_{a\in A}
        x_{a,1} \otimes \cdots \otimes x_{a,k}$.
    \item There exist a $d\in \NN_0$ and a
        $\skal{\argument}{\argument}_\alpha$-orthonormal tuple
        $e \in V^d$ as well as complex coefficients $X^{a'}$, such
        that
        \begin{equation}
            \label{eq:ZerlegungX}
            \tilde{X}
            =
            \sum_{a'\in\{1, \ldots, d\}^k} X^{a'}
            e_{a'_1} \otimes \cdots \otimes e_{a'_k}
            \quad
            \textrm{and}
            \quad
            \seminorm{\alpha}{X}^{\bullet,2}
            =
            \seminorm{\alpha}{\tilde{X}}^{\bullet,2}
            =
            k! \sum_{a'\in\{1, \ldots, d\}^k}
            \abs[\big]{X^{a'}}^2.
        \end{equation}
    \end{lemmalist}
\end{lemma}
\begin{proof}
    We can express $X$ as a finite sum of simple tensors,
    $X = \sum_{b\in B} x_{b,1} \otimes \cdots \otimes x_{b,k}$
    with a finite set $B$ and vectors $x_{b,i} \in V$. Let
    \[
    V_X
    :=
    \spann
    \big\{
    x_{b,i}
    \; \big| \;
    b\in B, i\in\{1, \ldots, k\}
    \big\}
    \textrm{ and }
    V_{X_0}
    :=
    \big\{v\in V_X
    \; \big| \;
    \seminorm{\alpha}{v} = 0
    \big\}.
    \]
    Construct a complementary linear subspace $V_{\tilde{X}}$ of
    $V_{X_0}$ in $V_X$, then we can also assume without loss of
    generality that $x_{b,i} \in V_{X_0} \cup V_{\tilde{X}}$ for all
    $b\in B$ and $i\in\{1, \ldots, k\}$. Note that $V_X,V_{X_0}$ and
    $V_{\tilde{X}}$ are all finite-dimensional.  Now define
    $A := \big\{a\in B \; \big| \; \exists_{n \in \{1, \ldots,
      k\}}\colon x_{a,n} \in V_{X_0}\big\}$
    and $X_0 := \sum_{a\in A} x_{a,1} \otimes \cdots \otimes x_{a,k}$,
    then $\prod_{n=1}^k \seminorm{\alpha}{x_{a,n}} = 0$ by
    construction and so $\seminorm{\alpha}{X_0}^{\bullet} = 0$ and
    $\seminorm{\alpha}{X - X_0}^\bullet =
    \seminorm{\alpha}{X}^\bullet$.
    Restricted to $V_{\tilde{X}}$, the positive Hermitian form
    $\skal{\argument}{\argument}_\alpha$ is even positive definite,
    i.e. an inner product. Let $d := \dim(V_{\tilde{X}})$ and
    $e \in V^d$ be an $\skal{\argument}{\argument}_\alpha$-orthonormal
    base of $V_{\tilde{X}}$. Define $\tilde{X} := X-X_0$, then
    $\tilde{X} = \sum_{a'\in\{1, \ldots, d\}^k} X^{a'} e_{a'_1}
    \otimes \cdots \otimes e_{a'_k}$
    with complex coefficients $X^{a'}$ and
    \begin{align*}
        \seminorm{\alpha}{X}^{\bullet,2}
        =
        \seminorm{\alpha}{\tilde{X}}^{\bullet, 2}
        =
        \sum_{a'\in\{1, \ldots, d\}^k} \abs[\big]{X^{a'}}^2
        \seminorm{\alpha}
        {e_{a'_1} \otimes \cdots \otimes e_{a'_k}}^{\bullet,2}
        =
        \sum_{a'\in\{1, \ldots, d\}^k} \abs[\big]{X^{a'}}^2 k!.
    \end{align*}
\end{proof}

On the locally convex space $\Ten^\bullet(V)$, the tensor product is
indeed continuous and $\big(\Ten^\bullet(V),\otimes\big)$ is a locally
convex algebra. In order to see this, we are going to prove the
continuity of the following function:
\begin{definition}
    We define the map
    $\multotimes\colon \Ten^\bullet(V) \otimes_\pi \Ten^\bullet(V)
    \rightarrow \Ten^\bullet(V)$ by
    \begin{equation}
        X \otimes_\pi Y
        \; \mapsto \;
        \multotimes(X\otimes_\pi Y)
        :=
        X \otimes Y.
    \end{equation}
\end{definition}
Algebraically, $\multotimes$ is of course just the product of the
tensor algebra. The emphasize lies here on the topologies involved:
$\otimes_\pi$ denotes the projective tensor product. We recall that
the topology on $\Ten^\bullet(V) \otimes_\pi \Ten^\bullet(V)$ is
described by the seminorms
$\seminorm{\alpha \otimes_\pi \beta}{\argument}^\bullet\colon
\Ten^\bullet(V) \otimes_\pi \Ten^\bullet(V) \rightarrow [0,\infty[$
\begin{equation}
    \label{eq:ProjectiveSeminorm}
    Z
    \; \mapsto \;
    \seminorm{\alpha \otimes_\pi \beta}{Z}^\bullet
    :=
    \inf \sum_{i\in I}
    \seminorm{\alpha}{X_i}^\bullet \seminorm{\beta}{Y_i}^\bullet,
\end{equation}
where the infimum runs over all possibilities to express $Z$ as a sum
$Z = \sum_{i\in I} X_i \otimes_\pi Y_i$ indexed by a finite set $I$
and
$\seminorm{\alpha}{\argument}^\bullet,
\seminorm{\beta}{\argument}^\bullet$
run over all extensions of continuous Hilbert seminorms on $V$. The
only property of the projective tensor product relevant for our
purposes is the following lemma, which is a direct result of the
definition of the seminorms
$\seminorm{\alpha \otimes_\pi \beta}{\argument}^\bullet$:
\begin{lemma}
    \label{lemma:otimespi}%
    Let $W$ be a locally convex space, $p$ a continuous seminorm on
    $W$ and
    $\seminorm{\alpha}{\argument},\seminorm{\beta}{\argument} \in
    \mathcal{P}_V$.
    Let
    $\Phi\colon \Ten^\bullet(V) \otimes_\pi \Ten^\bullet(V)\rightarrow
    W$ be a linear map. Then the two statements
    \begin{lemmalist}
    \item
        $p \big(\Phi(X\otimes_\pi Y)\big) \le
        \seminorm{\alpha}{X}^\bullet \seminorm{\beta}{Y}^\bullet$
        for all $X, Y \in \Ten^\bullet(V)$
    \item
        $p \big(\Phi(Z)\big) \le \seminorm{\alpha \otimes_\pi
          \beta}{Z}^\bullet$
        for all $Z \in \Ten^\bullet(V)\otimes_\pi \Ten^\bullet(V)$
    \end{lemmalist}
    are equivalent. Continuity of the bilinear map
    $\Ten^\bullet(V)\times \Ten^\bullet(V) \ni (X, Y) \mapsto
    \Phi(X\otimes_\pi Y) \in W$
    is therefore equivalent to continuity of $\Phi$.
\end{lemma}

\begin{proposition}
    \label{prop:contotimes}%
    The linear map $\multotimes$ is continuous and the estimate
    \begin{equation}
        \label{eq:ContinuityMuotimes}
        \seminorm{\gamma}{\multotimes(Z)}^\bullet
        \le
        \seminorm{2\gamma\otimes_\pi 2\gamma}{Z}^\bullet
    \end{equation}
    holds for all $Z\in \Ten^\bullet(V)\otimes_\pi \Ten^\bullet(V)$
    and all $\seminorm{\gamma}{\argument}\in \mathcal{P}_V$. Moreover,
    all $X\in \Ten^k(V)$ and $Y\in \Ten^\ell(V)$ with
    $k,\ell \in \NN_0$ fulfill for all
    $\seminorm{\gamma}{\argument}\in \mathcal{P}_V$ the estimate
    \begin{equation}
        \label{eq:PreciseContTensorProduct}
        \seminorm{\gamma}{\mu_\otimes(X\otimes_\pi Y)}^\bullet
        \le
        \binom{k+\ell}{k}^{1/2}
        \seminorm{\gamma}{X}^\bullet
        \seminorm{\gamma}{Y}^\bullet.
    \end{equation}
\end{proposition}
\begin{proof}
    Let $X\in \Ten^k(V)$ and $Y\in \Ten^\ell(V)$ with
    $k,\ell \in \NN_0$ be given. Then
    \begin{align*}
        \seminorm{\gamma}{X\otimes Y}^\bullet
        =
        \sqrt{\skal{X\otimes Y}{X\otimes Y}_\gamma^\bullet}
        =
        \binom{k\!+\!\ell}{k}^{\frac{1}{2}}
        \seminorm{\gamma}{X}^\bullet
        \seminorm{\gamma}{Y}^\bullet
    \end{align*}
    holds. It now follows for all $X, Y \in \Ten^\bullet(V)$ that
    \begin{align*}
        \seminorm{\gamma}{X \otimes Y}^{\bullet, 2}
        &=
        \sum_{m=0}^\infty
        \seminorm{\gamma}{\langle X \otimes Y\rangle_m}^{\bullet, 2}\\
        &\le
        \sum_{m=0}^\infty
        \left(
            \sum_{n=0}^m
            \seminorm{\gamma}
            {
              \langle X\rangle_{m-n} \otimes \langle Y\rangle_{n}
            }^\bullet
        \right)^2\\
        &=
        \sum_{m=0}^\infty
        \left(
            \sum_{n=0}^m
            \binom{m}{n}^{\frac{1}{2}}
            \seminorm{\gamma}{\langle X\rangle_{m-n}}^\bullet
            \seminorm{\gamma}{\langle Y\rangle_{n}}^\bullet
        \right)^2\\
        &=
        \sum_{m=0}^\infty
        \left(
            \sum_{n=0}^m
            \left(
                \binom{m}{n}
                \frac{1}{2^m}
            \right)^{\frac{1}{2}}
            \seminorm{2\gamma}{\langle X\rangle_{m-n}}^\bullet
            \seminorm{2\gamma}{\langle Y\rangle_{n}}^\bullet
        \right)^2\\
        &\stackrel{\cs}{\le}
        \sum_{m=0}^\infty
        \left(
            \sum_{n=0}^m
            \binom{m}{n}
            \frac{1}{2^{m}}
        \right)
        \left(
            \sum_{n=0}^m
            \seminorm{2\gamma}{\langle X\rangle_{m-n}}^{\bullet, 2}
            \seminorm{2\gamma}{\langle Y\rangle_{n}}^{\bullet,2}
        \right)\\
        &=
        \seminorm{2\gamma}{X}^{\bullet,2}
        \seminorm{2\gamma}{Y}^{\bullet,2},
    \end{align*}
    by the Cauchy-Schwarz (CS) inequality.
\end{proof}

%
%

\subsection{Symmetrisation}

The star product will be defined on the symmetric tensor algebra with
undeformed product $X\vee Y := \SymFkt^\bullet (X\otimes Y)$ for
$X,Y\in\SymTen^\bullet(V)$, which is indeed continuous:
\begin{proposition}
    \label{prop:symcont}%
    The symmetrisation operator is continuous and fulfills
    $\seminorm{\gamma}{\SymFkt^\bullet X}^\bullet \le
    \seminorm{\gamma}{X}^\bullet$
    for all $X\in \Ten^\bullet(V)$ and
    $\seminorm{\gamma}{\argument} \in \mathcal{P}_V$.
\end{proposition}
\begin{proof}
    From Definition~\ref{defi:extension} it is clear that
    $\skal{X^\sigma}{Y^\sigma}^{\bullet}_\gamma =
    \skal{X}{Y}^{\bullet}_\gamma$
    for all $k\in \NN_0$, $X,Y\in \Ten^k(V)$ and
    $\sigma \in \SymGrp_k$, because this holds for all simple tensors
    and because both sides are (anti-)linear in $X$ and $Y$. Therefore
    $\seminorm{\gamma}{X^\sigma}^\bullet =
    \seminorm{\gamma}{X}^\bullet$
    and
    $\seminorm{\gamma}{\SymFkt^k X}^\bullet \le
    \seminorm{\gamma}{X}^\bullet$
    and we get the desired estimate
    \begin{align*}
        \seminorm{\gamma}{\SymFkt^\bullet X}^{\bullet, 2}
        =
        \sum_{k=0}^\infty
        \seminorm{\gamma}{\SymFkt^k \langle X\rangle_k}^{\bullet,2}
        \le
        \sum_{k=0}^\infty
        \seminorm{\gamma}{\langle X\rangle_k}^{\bullet,2}
        =
        \seminorm{\gamma}{X}^{\bullet, 2}
    \end{align*}
    on $\Ten^\bullet(V)$.
\end{proof}

Analogously to $\multotimes$ we define the linear map $\multvee:=
\SymFkt^\bullet \circ \multotimes\colon \Ten^\bullet(V)\otimes_\pi
\Ten^\bullet(V)\rightarrow \Ten^\bullet(V)$.  Then the restriction of
$\multvee$ to $\SymTen^\bullet(V)$ describes the symmetric tensor
product $\vee$ and Propositions~\ref{prop:contotimes} and
\ref{prop:symcont} yield:
\begin{corollary}
    \label{cor:contvee}%
    The linear map $\multvee$ is continuous and the estimate
    $\seminorm{\gamma}{\multvee(Z)}^\bullet\le
    \seminorm{2\gamma\otimes_\pi 2\gamma}{Z}^\bullet$ holds for all
    $Z\in \Ten^\bullet(V)\otimes_\pi \Ten^\bullet(V)$ and all
    $\seminorm{\gamma}{\argument}\in \mathcal{P}_V$.
\end{corollary}

%
%

\subsection{The Star Product}

The following star product is based on a bilinear form and generalizes
the usual exponential-type star product like the Weyl-Moyal or Wick
star product, see e.g. \cite[Chap.~5]{waldmann:2007a}, to arbitrary
dimensions:
\begin{definition}
    \label{defi:stpro}%
    For every continuous bilinear form $\Lambda$ on $V$ we define the product
    $\multstpro[\Lambda]\colon
    \Ten^\bullet(V) \otimes_\pi \Ten^\bullet(V) \rightarrow \Ten^\bullet(V)$
    by
    \begin{equation}
        \label{eq:defstpro}
        X \otimes_\pi Y
        \; \mapsto \;
        \multstpro[\Lambda](X\otimes_\pi Y)
        :=
        \sum_{t=0}^\infty \frac{1}{t!}
        \multvee\Big(\big(\P_{\Lambda}\big)^t(X\otimes_\pi Y)\Big),
    \end{equation}
    where the linear map
    $\P_{\Lambda}\colon \Ten^\bullet(V) \otimes_\pi \Ten^\bullet(V)
    \rightarrow \Ten^{\bullet-1}(V) \otimes_\pi \Ten^{\bullet-1}(V)$
    is given on factorizing tensors of degree $k, \ell \in \NN$ by
    \begin{equation}
        \P_{\Lambda} \big(
        (x_1\otimes \cdots \otimes x_k)
        \otimes_\pi
        (y_1 \otimes \cdots \otimes y_\ell)
        \big)
        :=
        k\ell \Lambda(x_k, y_1)
        (x_1\otimes \cdots \otimes x_{k-1})
        \otimes_\pi
        (y_2 \otimes \cdots \otimes y_{\ell})
    \end{equation}
    for all $x \in V^k$ and $y \in V^\ell$. Moreover, we define the
    product $\stpro[\Lambda]$ on $\SymTen^\bullet(V)$ as the bilinear
    map described by the restriction of $\multstpro[\Lambda]$ to
    $\SymTen^\bullet(V)$.
\end{definition}

Note that these definitions of $\P_{\Lambda}$ and $\stpro[\Lambda]$
coincide (algebraically) on $\SymTen^\bullet(V)$ with the ones in
\cite[Eq.~(2.13) and (2.19)]{waldmann:2014a}, evaluated at a fixed
value for $\nu$ in the truely (not graded) symmetric case $V = V_0$.
Note that with our convention the deformation parameter $\hbar$ is
already part of $\Lambda$.

We are now going to prove the continuity of
$\stpro[\Lambda]$. Therefore we note that continuity of $\Lambda$
means that there exist
$\seminorm{\alpha}{\argument}, \seminorm{\beta}{\argument}\in
\mathcal{P}_V$
such that
$\abs{\Lambda(v,w)} \le \seminorm{\alpha}{v} \seminorm{\beta}{w}$
holds for all $v, w \in V$. So the set
\begin{equation}
    \label{eq:PVLambdaDef}
    \mathcal{P}_{V,\Lambda}
    :=
    \left\{
        \seminorm{\gamma}{\argument} \in \mathcal{P}_V
        \; \Big| \;
        \abs{\Lambda(v,w)} \le \seminorm{\gamma}{v} \seminorm{\gamma}{w}
        \textrm{ for all }
        v, w \in V
    \right\}
\end{equation}
contains at least all continuous Hilbert seminorms on $V$ that
dominate $\seminorm{\alpha + \beta}{\argument}$. Thus this set is
cofinal in $\mathcal{P}_V$.

\begin{lemma}
    \label{lemma:plambdaklij}%
    Let $\Lambda$ be a continuous bilinear form on $V$, let
    $\seminorm{\alpha}{\argument}, \seminorm{\beta}{\argument} \in
    \mathcal{P}_{V,\Lambda}$
    as well as $k, \ell \in \NN_0$ and $X\in \Ten^k(V)$,
    $Y\in \Ten^\ell(V)$ be given. Then
    \begin{equation}
        \label{eq:PLambdaEstiamte}
        \seminorm{\alpha \otimes_\pi \beta}
        {\P_{\Lambda} (X\otimes_\pi Y)}^\bullet
        \le
        \sqrt{k\ell}
        \seminorm{\alpha}{X}^\bullet
        \seminorm{\beta}{Y}^\bullet.
    \end{equation}
\end{lemma}
\begin{proof}
    If $k = 0$ or $\ell = 0$ this is clearly true, so assume
    $k, \ell \in \NN$. We use Lemma~\ref{lemma:helpfull} to construct
    $X_0 = \sum_{a\in A} x_{a,1} \otimes \cdots \otimes x_{a,k}$ and
    $\tilde{X} = \sum_{a' \in \{1, \ldots, c\}^k} X^{a'} e_{a'_1}
    \otimes \cdots \otimes e_{a'_k}$
    with respect to $\skal{\argument}{\argument}_\alpha$ as well as
    $Y_0 = \sum_{b\in B} y_{b,1} \otimes \cdots \otimes y_{b,\ell}$ and
    $\tilde{Y} = \sum_{b'\in\{1, \ldots, d\}^\ell} Y^{b'} f_{b'_1}
    \otimes \cdots \otimes f_{b'_\ell}$
    with respect to $\skal{\argument}{\argument}_\beta$. Then
    \[
    \seminorm[\big]{\alpha \otimes_\pi \beta}
    {\P_{\Lambda}
      \big((X_0 + \tilde{X}) \otimes_\pi (Y_0 + \tilde{Y})\big)
    }^\bullet
    \le
    \seminorm[\big]{\alpha \otimes_\pi \beta}
    {\P_{\Lambda}(\tilde{X}\otimes_\pi \tilde{Y})}^\bullet,
    \]
    because
    \begin{align*}
        \seminorm[\big]{\alpha \otimes_\pi \beta}
        {\P_{\Lambda}
          \big(
          (\xi_1 \otimes \cdots \otimes \xi_k)
          \otimes_\pi
          (\eta_1 \otimes \cdots \otimes \eta_\ell)
          \big)
        }^\bullet
        &=
        k\ell \abs{\Lambda(\xi_k, \eta_1)}
        \seminorm{\alpha}
        {\xi_1 \otimes \cdots \otimes \xi_{k-1}}^\bullet
        \seminorm{\beta}
        {\eta_2 \otimes \cdots \otimes \eta_\ell}^\bullet\\
        &=
        0
    \end{align*}
    for all $\xi\in V^k$, $\eta\in V^\ell$ for which there is at least
    one $m \in \{1, \ldots, k\}$ with
    $\seminorm{\alpha}{\xi_{m}} = 0$ or one
    $n\in\{1, \ldots, \ell\}$ with $\seminorm{\beta}{\eta_{n}} = 0$.
    On the subspaces
    $V_{\tilde{X}} = \spann\big\{e_1, \ldots, e_c\big\}$ and
    $V_{\tilde{Y}}=\spann\{f_1, \ldots, f_d\}$ of $V$, the bilinear
    form $\Lambda$ is described by a matrix
    $\Omega \in \CC^{c\times d}$ with entries
    $\Omega_{gh} = \Lambda(e_g,f_h)$. By using a singular value
    decomposition we can even assume without loss of generality that
    all off-diagonal entries of $\Omega$ vanish. We also note that
    $\abs{\Omega_{gg}} = \abs{\Lambda(e_g, f_g)} \le
    \seminorm{\alpha}{e_g} \seminorm{\beta}{f_g} \le 1$.
    This gives the desired estimate
    \begin{align*}
        &\seminorm{\alpha \otimes_\pi \beta}
        {\P_{\Lambda}(X \otimes_\pi Y)}^\bullet \\
        &\quad\le
        \seminorm[\big]{\alpha \otimes_\pi \beta}
        {\P_{\Lambda}(\tilde{X}\otimes_\pi \tilde{Y})}^\bullet \\
        &\quad=
        \seminorm[\bigg]{\alpha \otimes_\pi \beta}
        {
          \sum\nolimits_{a'\in \{1, \ldots, c\}^k}
          \sum\nolimits_{b'\in \{1, \ldots, d\}^\ell}
          X^{a'} Y^{b'} \P_{\Lambda}
          \big(
          (e_{a'_1} \otimes \cdots \otimes e_{a'_k})
          \otimes_\pi
          (f_{b'_1} \otimes \cdots \otimes f_{b'_\ell})
          \big)
        }^\bullet \\
        &\quad=
        k\ell
        \seminorm[\bigg]{\alpha \otimes_\pi \beta}
        {
          \sum_{r=1}^{\min\{c, d\}}
          \sum\nolimits_{
            \substack{
              \tilde{a}' \in \{1, \ldots, c\}^{k-1}
              \\
              \tilde{b}' \in \{1, \ldots, d\}^{\ell-1}
            }
          }
          X^{(\tilde{a}', r)}
          Y^{(r, \tilde{b}')}
          \Omega_{rr}
          (e_{\tilde{a}'_1}
          \otimes \cdots \otimes
          e_{\tilde{a}'_{k-1}})
          \otimes_\pi
          (f_{\tilde{b}'_1}
          \otimes \cdots \otimes
          f_{\tilde{b}'_{\ell-1}})
        }^\bullet \\
        &\quad\le
        k\ell
        \sum_{r=1}^{\min\{c, d\}}
        \seminorm[\bigg]{\alpha}
        {
          \sum_{\tilde{a}' \in  \{1, \ldots, c\}^{k-1}}
          X^{(\tilde{a}',r)} e_{\tilde{a}'_1}
          \otimes \cdots \otimes
          e_{\tilde{a}'_{k-1}}
        }^\bullet
        \seminorm[\bigg]{\beta}
        {
          \sum_{\tilde{b}' \in  \{1, \ldots, d\}^{\ell-1}}
          Y^{(r, \tilde{b}')} f_{\tilde{b}'_1}
          \otimes \cdots \otimes
          f_{\tilde{b}'_{\ell-1}}
        }^\bullet \\
        &\quad\stackrel{\cs}{\le}
        \sqrt{k\ell}
        \seminorm{\alpha}{X}^{\bullet}
        \seminorm{\beta}{Y}^{\bullet},
    \end{align*}
    where we have used in the last line after applying the
    Cauchy-Schwarz inequality that
    \begin{align*}
        \sum_{r=1}^{\min\{c, d\}}
        \seminorm[\bigg]{\alpha}
        {
          \sum_{\tilde{a}' \in \{1, \ldots, c\}^{k-1}}
          X^{(\tilde{a}', r)} e_{\tilde{a}'_1}
          \otimes \cdots \otimes
          e_{\tilde{a}'_{k-1}}
        }^{\bullet, 2}
        &=
        \sum_{r=1}^{\min\{c, d\}}
        \sum_{\tilde{a}' \in \{1, \ldots, c\}^{k-1}}
        \abs[\big]{X^{(\tilde{a}', r)}}^2 (k-1)! \\
        &\le
        \frac{1}{k} \seminorm{\alpha}{X}^{\bullet, 2}
    \end{align*}
    and analogously for $Y$.
\end{proof}
\begin{proposition}
    \label{prop:contplambda}%
    Let $\Lambda$ be a continuous bilinear form on $V$, then the
    function $\P_{\Lambda}$ is continuous and fulfills the estimate
    \begin{equation}
        \label{eq:EstimatePLambdaHocht}
        \seminorm[\big]{\alpha\otimes_\pi\beta}
        {\big(\P_{\Lambda}\big)^t(Z)}^\bullet
        \le
        \frac{c}{c-1}
        \frac{t!}{c^t}
        \seminorm{2c\alpha \otimes_\pi 2c\beta}{Z}^\bullet
    \end{equation}
    for all $c > 1$, all $t \in \NN_0$, all seminorms
    $\seminorm{\alpha}{\argument}, \seminorm{\beta}{\argument} \in
    \mathcal{P}_{V, \Lambda}$,
    and all $Z\in \Ten^\bullet(V) \otimes_\pi \Ten^\bullet(V)$.
\end{proposition}
\begin{proof}
    Let $X, Y \in \Ten^\bullet(V)$ be given, then the previous
    Lemma~\ref{lemma:plambdaklij}
    together with Lemma~\ref{lemma:otimespi} yields
    \begin{align*}
        \seminorm[\big]{\alpha \otimes_\pi\beta}
        {\big(\P_{\Lambda}\big)^t(X \otimes_\pi Y)}^\bullet
        &\le
        \sum_{k, \ell = 0}^\infty
        \seminorm[\big]{\alpha \otimes_\pi \beta}
        {
          \big(\P_{\Lambda}\big)^t
          \big(\langle X \rangle_{k+t}
          \otimes_\pi
          \langle Y \rangle_{\ell+t}\big)
        }^\bullet \\
        &\le
        t! \sum_{k, \ell = 0}^\infty
        \binom{k+t}{t}^{\frac{1}{2}}
        \binom{\ell+t}{t}^{\frac{1}{2}}
        \seminorm{\alpha}{\langle X \rangle_{k+t}}^\bullet
        \seminorm{\beta}{\langle Y \rangle_{\ell+t}}^\bullet \\
        &\le
        t! \sum_{k, \ell = 0}^\infty
        \seminorm{2\alpha}{\langle X \rangle_{k+t}}^\bullet
        \seminorm{2\beta}{\langle Y \rangle_{\ell+t}}^\bullet \\
        &=
        \frac{t!}{c^t}
        \sum_{k, \ell = 0}^\infty
        \frac{1}{\sqrt{c}^{k+\ell}}
        \seminorm{2c\alpha}{\langle X \rangle_{k+t}}^\bullet
        \seminorm{2c\beta}{\langle Y \rangle_{\ell+t}}^\bullet \\
        &\stackrel{\cs}{\le}
        \frac{t!}{c^t}
        \bigg(
        \sum_{k, \ell=0}^\infty
        \frac{1}{c^{k+\ell}}
        \bigg)^{\frac{1}{2}}
        \bigg(
        \sum_{k, \ell=0}^\infty
        \seminorm{2c\alpha}{\langle X \rangle_{k+t}}^{\bullet, 2}
        \seminorm{2c\beta}{\langle Y \rangle_{\ell+t}}^{\bullet, 2}
        \bigg)^{\frac{1}{2}} \\
        &\le
        \frac{c}{c-1} \frac{t!}{c^t}
        \seminorm{2c\alpha}{X}^\bullet
        \seminorm{2c\beta}{Y}^\bullet.
    \end{align*}
\end{proof}
\begin{lemma}
    \label{lemma:stprocont}%
    Let $\Lambda$ be a continuous bilinear form on $V$, then
    $\multstpro[\Lambda]$ is continuous and, given $R>1/2$, the estimate
    \begin{align}
        \seminorm{\gamma}{\multstpro[z\Lambda](Z)}^\bullet
        \le
        \sum_{t=0}^\infty
        \frac{1}{t!}
        \seminorm[\Big]{\gamma}
        {\multvee\Big(\big(\P_{z\Lambda}\big)^t(Z)\Big)}^\bullet
        \le
        \frac{4R}{2R-1} \seminorm{8R\gamma\otimes_\pi 8R\gamma}{Z}^\bullet
    \end{align}
    holds for all
    $\seminorm{\gamma}{\argument} \in \mathcal{P}_{V, \Lambda}$,
    all $Z \in \Ten^\bullet(V) \otimes_\pi \Ten^\bullet(V)$ and all $z\in \CC$ with $|z|\le R$.
\end{lemma}
\begin{proof}
    The first estimate is just the triangle-inequality. By combining
    Corollary~\ref{cor:contvee} and Proposition~\ref{prop:contplambda}
    with $c=2R$ we get the second estimate
    \begin{align*}
        \sum_{t=0}^\infty
        \frac{1}{t!}
        \seminorm[\Big]{\gamma}
        {\multvee\Big(\big(\P_{z\Lambda}\big)^t(Z)\Big)}^\bullet
        &\le
        \sum_{t=0}^\infty
        \frac{|z|^t}{t!}
        \seminorm[\Big]{2\gamma \otimes_\pi 2\gamma}
        {\big(\P_{\Lambda}\big)^t(Z)}^\bullet \\
        &\le
        \frac{2R}{2R-1}\sum_{t=0}^\infty
        \frac{1}{2^t}
        \seminorm{8R\gamma \otimes_\pi 8R\gamma}{Z} \\
        &=
        \frac{4R}{2R-1} \seminorm{8R\gamma \otimes_\pi 8R\gamma}{Z}.
    \end{align*}
\end{proof}
This estimate immediately leads to:
\begin{theorem}
    \label{theorem:stprocont}%
    Let $\Lambda$ be a continuous bilinear form on $V$, then the
    product $\stpro[\Lambda]$ is continuous and
    $\big(\SymTen^\bullet(V), \stpro[\Lambda]\big)$ is a locally
    convex algebra. Moreover, for fixed tensors $X,Y$ from the completion
    $\SymTen^\bullet(V)^\cpl$, the product $X \stpro[z\Lambda] Y$ converges
    absolutely and locally uniformly in $z\in \CC$ and thus depends
    holomorphically on $z$.
\end{theorem}

Note that the above estimate also shows that $\big(\SymTen^\bullet(V),
\stpro[z\Lambda]\big)$ describes a holomorphic deformation (as defined
in \cite{pflaum.schottenloher:1998a}) of the locally convex algebra
$\big(\SymTen^\bullet(V), \vee\big)$. However, in the following we
will examine the star product for fixed values of both $\Lambda$ and
$z$ and therefore can absorb the deformation parameter $z$ in the
bilinear form $\Lambda$.

%
%

\section{Properties of the Star Product}
\label{sec:PropertiesStarProduct}

In this section we want to examine some properties of the products
$\stpro[\Lambda]$, namely how the topology on $\SymTen^\bullet(V)$ can
be characterized by demanding that certain algebraic operations are
continuous, which products are equivalent, how to transform
$\SymTen^\bullet(V)$ to a space of complex functions, the existence
of continuous positive linear functionals and whether or
not some exponentials of elements in $\SymTen^\bullet(V)$ exist and
which elements are represented by essentially self-adjoint operators
via GNS-construction. At some points we will also work with the
completion $\SymTen^\bullet(V)^\cpl$ of $\SymTen^\bullet(V)$ and
therefore note that the previous
constructions and results extend to $\SymTen^\bullet(V)^\cpl$ by
continuity.

%
%

\subsection{Characterization of the Topology}
\label{subsec:char}

We are going to show that the topology on $\SymTen^\bullet(V)$ that
was defined in the last section in a rather unmotivated way is --
under some additional assumptions -- the coarsest possible one. More
precisely, we want to express the extensions of positive Hermitian
forms with the help of suitable star products. Due to the
sesquilinearity of positive Hermitian forms, this is only possible if
we also have an antilinear structure on $\SymTen^\bullet(V)$, so we
construct a $^*$-involution.

There is clearly one and only one possibility to extend an antilinear
involution $\cc{\argument}$ on $V$ to a $^*$-involution $^*\colon
\Ten^\bullet(V)\rightarrow \Ten^\bullet(V)$ on the tensor algebra over
$V$, namely by $(x_1 \otimes \cdots \otimes x_k)^* := \cc{x}_k \otimes
\cdots \otimes \cc{x}_1$ for all $k\in \NN$ and $x \in V^k$ and
antilinear extension. Its restriction to $\SymTen^\bullet(V)$ gives a
$^*$-involution on $\big(\SymTen^\bullet(V),\vee\big)$.
\begin{proposition}
    Let $\cc{\argument}$ be a continuous antilinear involution on $V$,
    then the induced $^*$-involution on $\Ten^\bullet(V)$ is also
    continuous.
\end{proposition}
\begin{proof}
    For $\skal{\argument}{\argument}_\alpha \in \mathcal{I}_V$ define
    the continuous positive Hermitian form
    $V^2\ni (v, w) \mapsto \skal{v}{w}_{\alpha^*} :=
    \cc{\skal{\cc{v}}{\cc{w}}_\alpha}$.
    Then
    $\skal{X^*}{Y^*}_\alpha^\bullet =
    \cc{\skal{X}{Y}_{\alpha^*}^\bullet}$
    and in particular
    $\seminorm{\alpha}{X^*}^\bullet = \seminorm{\alpha^*}{X}^\bullet$
    for all $X, Y \in \Ten^\bullet(V)$ because this is clearly true
    for simple tensors and because both sides are (anti-)linear in $X$
    and $Y$.
\end{proof}

For certain bilinear forms $\Lambda$ on $V$ we can also show that
${ }^*$ is a $^*$-involution of $\stpro[\Lambda]$, which is of
course not a new result:
\begin{definition}
    Let $\cc{\argument}\colon V \rightarrow V$ be a continuous
    antilinear involution on $V$. For every continuous bilinear form
    $\Lambda\colon V \times V\rightarrow \CC$ we define its conjugate
    $\Lambda^*$ by $\Lambda^*(v, w) := \cc{\Lambda(\cc{w}, \cc{v})}$,
    which is again a continuous bilinear form on $V$. We say that
    $\Lambda$ is Hermitian if $\Lambda=\Lambda^*$ holds.
\end{definition}
Note that the bilinear form $(v, w) \mapsto \Lambda(v,w)$ is Hermitian
if and only if the sesquilinear form $(v, w)\mapsto \Lambda(\cc{v},w)$
is Hermitian.  The typical example of a complex vector space $V$ with
antilinear involution $\cc{\argument}$ is that $V = W \otimes \CC$ is
the complexification of a real vector space $W$ with the canonical
involution $\cc{w \otimes \lambda} := w \otimes \cc{\lambda}$. In this
case, every bilinear form $\Lambda$ on $V$ is fixed by two bilinear
forms $\Lambda_r, \Lambda_i\colon W\times W\rightarrow \RR$, the
restriction of the real- and imaginary part of $\Lambda$ to the real
subspace $W \cong W \otimes 1$ of $V$, and $\Lambda$ is Hermitian if
and only if $\Lambda_r$ is symmetric and $\Lambda_i$ antisymmetric.
Similarly to \cite[Prop.~3.25]{waldmann:2014a} we
get:
\begin{proposition}
    Let $\cc{\argument}\colon V \rightarrow V$ be a continuous
    antilinear involution and $\Lambda$ a continuous bilinear form on
    $V$. Then $(X \stpro[\Lambda] Y)^* = Y^* \stpro[\Lambda^*] X^*$
    holds for all $X, Y \in \SymTen^\bullet(V)$. Consequently, if
    $\Lambda$ is Hermitian, then
    $\big(\SymTen^\bullet(V), \stpro[\Lambda], { }^*\big)$ is a
    locally convex $^*$-algebra.
\end{proposition}
\begin{proof}
    The identities
    ${ }^*\circ \SymFkt^\bullet = \SymFkt^\bullet\circ { }^*$ and
    ${ }^* \circ \multotimes = \multotimes \circ \tau \circ \big({ }^*
    \otimes_\pi { }^*\big)$,
    with
    $\tau\colon \Ten^\bullet(V)\otimes_\pi\Ten^\bullet(V)\rightarrow
    \Ten^\bullet(V)\otimes_\pi\Ten^\bullet(V)$
    defined as $\tau(X \otimes_\pi Y) := Y \otimes_\pi X$, can easily
    be checked on simple tensors, so
    ${ }^* \circ \multvee = \multvee\circ \tau \circ \big({ }^*
    \otimes_\pi { }^*\big)$.
    Combining this with
    $\tau \circ ({ }^* \otimes_\pi { }^* ) \circ \P_{\Lambda} =
    \P_{\Lambda^*} \circ \tau \circ ({ }^* \otimes_\pi { }^*)$
    on symmetric tensors, which again can easily be checked on simple
    symmetric tensors, yields the desired result.
\end{proof}
\begin{lemma}
    \label{lemma:char}%
    Let $\cc{\argument}\colon V\rightarrow V$ be a continuous
    antilinear involution. For every
    $\skal{\argument}{\argument}_\alpha \in \mathcal{I}_V$ we define a
    continuous bilinear form $\Lambda_\alpha$ on $V$ by
    $\Lambda_\alpha(v, w) := \skal{\cc{v}}{w}_\alpha$ for all
    $v, w \in V$, then $\Lambda_\alpha$ is Hermitian and the
    identities
    \begin{gather}
        \label{eq:starskal}
        \sum_{t=0}^\infty
        \frac{1}{t!}
        \multotimes\Big(
        \big(\P_{\Lambda_\alpha}\big)^t
        \big(
        \langle X^* \rangle_t \otimes_\pi \langle Y\rangle_t
        \big)
        \Big)
        =
        \skal{X}{Y}_\alpha^{\bullet} \\
        \shortintertext{and}
        \label{eq:starskalNull}
        \big\langle
        \multstpro[\Lambda_\alpha](X^*\otimes_\pi Y)
        \big\rangle_0
        =
        \skal{X}{Y}_\alpha^{\bullet}
    \end{gather}
    hold for all $X, Y\in \Ten^\bullet(V)$.
\end{lemma}
\begin{proof}
    Clearly, $\Lambda_\alpha$ is Hermitian because
    $\skal{\argument}{\argument}_\alpha$ is Hermitian. Then
    \eqref{eq:starskalNull} follows directly from \eqref{eq:starskal}
    because of the grading of $\multvee$ and $\P_{\Lambda_\alpha}$.
    For proving \eqref{eq:starskal} it is sufficient to check
    it for factorizing tensors of the same degree, because both sides
    are (anti-)linear in $X$ and $Y$ and vanish if $X$ and $Y$ are
    homogeneous of different degree. If $X$ and $Y$ are of degree $0$
    then \eqref{eq:starskal} is clearly fulfilled.
    Otherwise we get
    \begin{align*}
        &\frac{1}{k!}
        \multotimes
        \Big(
        \big(\P_{\Lambda_\alpha}\big)^k
        \big(
        (x_1 \otimes \cdots \otimes x_k)^*
        \otimes_\pi
        (y_1\otimes \cdots \otimes y_k)
        \big)
        \Big) \\
        &\quad=
        \frac{1}{k!}
        \multotimes
        \Big(
        \big(\P_{\Lambda_\alpha}\big)^k
        \big(
        (\cc{x}_k \otimes \cdots \otimes \cc{x}_1)
        \otimes_\pi
        (y_1 \otimes \cdots \otimes y_k)
        \big)
        \Big) \\
        &\quad=
        \frac{1}{k!}
        \multotimes
        \bigg(
        \big(1 \otimes_\pi 1 \big)
        (k!)^2
        \prod_{m=1}^k \Lambda_\alpha(\cc{x}_m, y_m)
        \bigg)\\
        &\quad=
        k!
        \prod_{m=1}^k \Lambda_\alpha(\cc{x}_m, y_m)\\
        &\quad=
        k!
        \prod_{m=1}^k \skal{x_m}{y_m}_\alpha \\
        &\quad=
        \skal{x_1 \otimes \cdots \otimes x_k}
        {y_1 \otimes\cdots \otimes y_k}^\bullet_\alpha.
    \end{align*}
\end{proof}

\begin{theorem}
    \label{theo:char}%
    The topology on $\SymTen^\bullet(V)$ is the coarsest locally
    convex one that makes all star products $\stpro[\Lambda]$ for all
    continuous and Hermitian bilinear forms $\Lambda$ on $V$ as well
    as the $^*$-involution and the projection
    $\langle\argument\rangle_0$ onto the scalars continuous. In
    addition we have for all $X, Y \in \SymTen^\bullet(V)$ and all
    $\skal{\argument}{\argument}_\alpha \in \mathcal{I}_V$
    \begin{equation}
        \label{eq:SkalarproduktSternprodukt}
        \langle X^* \stpro[\Lambda_\alpha] Y\rangle_0
        =
        \skal{X}{Y}_\alpha^\bullet,
    \end{equation}
    with $\Lambda_\alpha$ as in Lemma~\ref{lemma:char}.
\end{theorem}
\begin{proof}
    We have already shown the continuity of the star product and of
    the $^*$-involution, the continuity of $\langle\argument\rangle_0$
    is clear. Conversely, if these three functions are continuous,
    their compositions yield the extensions of all
    $\skal{\argument}{\argument}_\alpha\in \mathcal{I}_V$ which then
    have to be continuous. Then \eqref{eq:starskalNull} gives
    \eqref{eq:SkalarproduktSternprodukt} for symmetric tensors $X$ and
    $Y$.
\end{proof}

%
%

\subsection{Equivalence of Star Products}

Next we want to examine the usual equivalence transformations between
star products, given by exponentials of a Laplace operator (see
\cite{waldmann:2014a} for the algebraic background).
\begin{definition}
    Let $b\colon V\times V\rightarrow \CC$ be a symmetric bilinear
    form on $V$, i.e. $b(v, w) = b(w,v )$ for all $v, w \in V$. Then
    we define the Laplace operator $\Delta_{b}\colon \Ten^\bullet(V)
    \rightarrow \Ten^{\bullet-2}(V)$ as the linear map given on simple
    tensors of degree $k \in \NN \backslash\{1\}$ by
    \begin{equation}
        \label{eq:DeltaDef}
        \Delta_{b}
        (x_1 \otimes \cdots \otimes x_k)
        :=
        \frac{k(k-1)}{2} b(x_1, x_2) x_3 \otimes \cdots \otimes x_{k}.
    \end{equation}
\end{definition}
Note that $\Delta_{b}$ can be restricted to symmetric tensors on which
it coincides with the Laplace operator from
\cite[Eq.~(2.31)]{waldmann:2014a}.  However, there is no need for
$\Delta_b$ to be continuous even if $b$ is continuous, because the
Hilbert tensor product in general does not allow the extension of all
continuous multilinear forms. Note that this is very different from
the approach taken in \cite{waldmann:2014a} where the projective
tensor product was used: this guaranteed the continuity of the Laplace
operator directly for all continuous bilinear forms.

For the restriction of $\Delta_b$ to $\SymTen^2(V)$, continuity is
equivalent to the existence of a
$\seminorm{\alpha}{\argument} \in \mathcal{P}_V$ that fulfills
$\abs{\Delta_b X} \le \seminorm{\alpha}{X}^\bullet$ for all
$X\in \SymTen^2(V)$. This motivates the following:
\begin{definition}
    \label{defi:hilbertschmidt}%
    A bilinear form of Hilbert-Schmidt type on $V$ is a bilinear form
    $b\colon V\times V\rightarrow \CC$ for which there is a seminorm
    $\seminorm{\alpha}{\argument} \in \mathcal{P}_V$ such that the
    following two conditions are fulfilled:
    \begin{definitionlist}
    \item If $\seminorm{\alpha}{v} = 0$ or $\seminorm{\alpha}{w} = 0$
        for vectors $v, w\in V$, then $b(v, w) = 0$.
    \item For every tuple of
        $\skal{\argument}{\argument}_\alpha$-orthonormal vectors
        $e\in V^d$, $d\in \NN$, the estimate
        \begin{equation}
            \label{eq:HScondition}
            \sum_{i, j = 1}^d \abs{b(e_i, e_j)}^2 \le 1
        \end{equation}
        holds.
    \end{definitionlist}
    For such a bilinear form of Hilbert-Schmidt type $b$ we define
    $\mathcal{P}_{V, b, HS}$ as the set of all
    $\seminorm{\alpha}{\argument} \in \mathcal{P}_V$ that fulfill
    these two conditions.
\end{definition}
We can characterize the bilinear forms of Hilbert-Schmidt type in the
following way:
\begin{proposition}
    \label{prop:hschar}%
    Let $b$ be a symmetric bilinear form on $V$ and
    $\seminorm{\alpha}{\argument}\in \mathcal{P}_V$, then the
    following two statements are equivalent:
    \begin{propositionlist}
    \item The bilinear form $b$ is of Hilbert-Schmidt type and
        $\seminorm{\alpha}{\argument} \in \mathcal{P}_{V, b, HS}$.
    \item The estimate
        $\abs{\Delta_b X} \le 2^{-1/2} \seminorm{\alpha}{X}^\bullet$
        holds for all $X\in \SymTen^2(V)$.
    \end{propositionlist}
    Moreover, if this holds then
    $\seminorm{\alpha}{\argument} \in \mathcal{P}_{V, b}$ and $b$ is
    continuous.
\end{proposition}
\begin{proof}
    If the first point holds, let $X \in \Ten^2(V)$ be
    given. Construct $X_0 = \sum_{a \in A} x_{a, 1} \otimes x_{a, 2}$
    and
    $\tilde{X} = \sum_{a'_1, a'_2 = 1}^d X^{a'_1, a'_2} e_{a'_1}
    \otimes e_{a'_2} \in \Ten^2(V)$
    like in Lemma~\ref{lemma:helpfull}. Then $b(x_{a,1}, x_{a,2}) = 0$
    for all $a \in A$ because $\seminorm{\alpha}{x_{a, 1}} = 0$ or
    $\seminorm{\alpha}{x_{a,2}} = 0$. Moreover,
    \begin{align*}
        \abs{\Delta_b X}
        &\le
        \abs[\bigg]{
          \sum\nolimits_{a'_1, a'_2=1}^d
          X^{a'_1, a'_2}
          b(e_{a'_1}, e_{a'_2})
        } \\
        &\stackrel{\cs}{\le}
        \bigg(
        \sum\nolimits_{a'_1, a'_2=1}^d
        \abs[\big]{X^{a'_1,a'_2}}^2
        \bigg)^{\frac{1}{2}}
        \bigg(
        \sum\nolimits_{a'_1, a'_2=1}^d
        \abs[\big]{b(e_{a'_1}, e_{a'_2})}^2
        \bigg)^{\frac{1}{2}}
        \\
        &\le
        \frac{1}{\sqrt{2}} \seminorm{\alpha}{X}^\bullet
    \end{align*}
    shows that the second point holds.  Conversely, from the second
    point we get
    $\abs{b(v,w)} = \abs{\Delta_b (v\vee w)} \le 2^{-1/2}
    \seminorm{\alpha}{v \vee w}^\bullet \le \seminorm{\alpha}{v}
    \seminorm{\alpha}{w}$
    for all $v, w \in V$. Hence
    $\seminorm{\alpha}{\argument} \in \mathcal{P}_{V,b}$, the bilinear
    form $b$ is
    continuous, and $b(v,w) = 0$ if one of $v$ or $w$ is in the kernel
    of $\seminorm{\alpha}{\argument}$. Moreover, given an
    $\skal{\argument}{\argument}_\alpha$-orthonormal set of vectors
    $e \in V^d$, $d\in \NN$, we define
    $X := \sum_{i,j=1}^d \cc{b(e_i,e_j)} e_i \otimes e_j \in
    \SymTen^2(V)$ and get
    \begin{equation*}
        0
        \le
        \sum_{i,j=1}^d
        \abs{b(e_i,e_j)}^2
        =
        \abs{\Delta_bX}
        \le
        \frac{1}{\sqrt{2}}
        \seminorm{\alpha}{X}^\bullet
        =
        \bigg(\sum\nolimits_{i,j=1}^d \abs{b(e_i,e_j)}^2\bigg)^{\frac{1}{2}},
    \end{equation*}
    which implies $\sum_{i,j=1}^d \abs{b(e_i,e_j)}^2 \le 1$.
\end{proof}

Note that this also implies that for a bilinear form of
Hilbert-Schmidt type $b$, the set $\mathcal{P}_{V, b, HS}$ is cofinal
in $\mathcal{P}_V$, because if
$\seminorm{\alpha}{\argument} \in \mathcal{P}_{V, b, HS}$ and
$\seminorm{\beta}{\argument} \ge \seminorm{\alpha}{\argument}$, then
$\abs{\Delta_b X} \le 2^{-1/2} \seminorm{\alpha}{X}^\bullet \le
2^{-1/2}\seminorm{\beta}{X}^\bullet$
and so $\seminorm{\beta}{\argument} \in \mathcal{P}_{V, b, HS}$.

As a consequence of the above characterization we see that a
symmetric bilinear form $b$ on $V$ has to be of Hilbert-Schmidt type
if we want $\Delta_b$ to be continuous. We are going to show now that
this is also sufficient:
\begin{proposition}
    \label{proposition:Laplaceestimate}%
    Let $b$ be a symmetric bilinear form of Hilbert-Schmidt type on
    $V$, then the Laplace operator $\Delta_b$ is continuous and
    fulfills the estimate
    \begin{equation}
        \label{eq:Laplacepowers}
        \seminorm[\big]{\alpha}{(\Delta_b)^t X}^\bullet
        \le
        \frac{\sqrt{(2t)!}}{(2r)^t} \seminorm{2r\alpha}{X}^\bullet
    \end{equation}
    for all $X \in \Ten^\bullet(V)$, $t\in \NN_0$, $r\ge 1$, and all
    $\seminorm{\alpha}{\argument} \in \mathcal{P}_{V, b, HS}$.
\end{proposition}
\begin{proof}
    First, let $X\in \Ten^k(V)$, $k \ge 2$, and
    $\seminorm{\alpha}{\argument} \in \mathcal{P}_{V,b,HS}$ be
    given. Construct
    $X_0 = \sum_{a\in A} x_{a,1} \otimes \cdots \otimes x_{a,k}$ and
    $\tilde{X} = \sum_{a' \in \{1, \ldots, d\}^k} X^{a'} e_{a'_1}
    \otimes \cdots \otimes e_{a'_k}$
    like in Lemma~\ref{lemma:helpfull}. Then again
    \begin{equation*}
        \seminorm{\alpha}{\Delta_{b} X_0}^\bullet
        \le
        \frac{k(k-1) \sqrt{(k-2)!}}{2}
        \sum_{a\in A} \abs{b(x_{a_1}, x_{a_2})}
        \prod_{m=3}^{k} \seminorm{\alpha}{x_{a_{m}}}
        =
        0
    \end{equation*}
    shows that
    $\seminorm{\alpha}{\Delta_{b} X}^\bullet \le
    \seminorm{\alpha}{\Delta_{b} \tilde{X}}^\bullet$.
    For $\tilde{X}$ we get:
    \begin{align*}
        \seminorm[\big]{\alpha}{\Delta_{b}\tilde{X}}^{\bullet, 2}
        &=
        \seminorm[\Bigg]{\alpha}
        {
          \frac{k(k-1)}{2}
          \sum_{a' \in \{1, \ldots, d\}^k}
          X^{a'} b\big(e_{a'_1}, e_{a'_2}\big)
          e_{a'_3} \otimes \cdots \otimes e_{a'_{k}}
        }^{\bullet, 2} \\
        &=
        \frac{k^2(k-1)^2}{4}
        \sum_{\tilde{a}' \in \{1, \ldots, d\}^{k-2}}
        \seminorm[\Bigg]{\alpha}
        {
          \sum_{g,h=1}^{d}
          X^{(g,h,\tilde{a}')} b(e_g,e_h)
          e_{\tilde{a}'_1} \otimes \cdots \otimes e_{\tilde{a}'_{k-2}}
        }^{\bullet, 2} \\
        &=
        \frac{k^2(k-1)^2}{4}
        \sum_{\tilde{a}' \in \{1, \ldots, d\}^{k-2}}
        \abs[\Bigg]{
          \sum_{g,h=1}^{d}
          X^{(g,h,\tilde{a}')}
          b(e_g,e_h)
        }^2 (k-2)! \\
        &\le
        \frac{k(k-1)k!}{4}
        \sum_{\tilde{a}' \in \{1, \ldots,d\}^{k-2}}
        \Bigg(
        \sum_{g,h=1}^{d}
        \abs[\big]{X^{(g,h,\tilde{a}')}}
        \abs{b(e_g,e_h)}
        \Bigg)^2 \\
        &\stackrel{\cs}{\le}
        \frac{k(k-1)k!}{4}
        \sum_{\tilde{a}' \in \{1, \ldots, d\}^{k-2}}
        \Bigg(
        \sum_{g,h=1}^{d}
        \abs[\big]{X^{(g,h,\tilde{a}')}}^2
        \Bigg)
        \Bigg(
        \sum_{g,h=1}^{d}
        \abs{b(e_g,e_h)}^2
        \Bigg) \\
        &\le
        \frac{k(k-1)k!}{4}
        \sum_{a' \in \{1, \ldots, d\}^k}
        \abs{X^{a'}}^2 \\
        &=
        \frac{k(k-1)}{4}
        \seminorm{\alpha}{X}^{\bullet, 2}.
    \end{align*}
    Using this we get
    \begin{align*}
        \seminorm[\big]{\alpha}{(\Delta_b)^t X}^{\bullet, 2}
        &=
        \sum_{k=2t}^\infty
        \seminorm[\big]{\alpha}
        {(\Delta_b)^t \langle X\rangle_k}^{\bullet,2}
        \\
        &\le
        \sum_{k=2t}^\infty
        \binom{k}{2t}
        \frac{(2t)!}{4^t}
        \seminorm{\alpha}{\langle X\rangle_k}^{\bullet, 2} \\
        &\le
        \frac{(2t)!}{4^t}
        \sum_{k=2t}^\infty
        \frac{1}{r^k}
        \seminorm{2r\alpha}{\langle X\rangle_k}^{\bullet, 2} \\
        &\le
        \frac{(2t)!}{(2r)^{2t}}
        \seminorm{2r\alpha}{X}^{\bullet,2}
    \end{align*}
    for arbitrary $X\in \Ten^\bullet(V)$ and $t\in \NN$.  Finally, the
    estimate \eqref{eq:Laplacepowers} also holds in the case $t=0$.
\end{proof}
\begin{theorem}
    \label{theo:equivalence}%
    Let $b$ be a symmetric bilinear form on $V$, then the linear operator
    $\E^{\Delta_b} = \sum_{t=0}^\infty \frac{1}{t!} (\Delta_b)^t$ as
    well as its restriction to $\SymTen^\bullet(V)$ are continuous if
    and only if $b$ is of Hilbert-Schmidt type. In this case
    \begin{equation}
        \label{eq:equiv}
        \E^{\Delta_b}\big(X\stpro[\Lambda] Y\big)
        =
        \big(\E^{\Delta_b} X\big) \stpro[\Lambda+b] \big(\E^{\Delta_b} Y\big)
    \end{equation}
    holds for all $X, Y \in \SymTen^\bullet(V)$ and all continuous
    bilinear forms $\Lambda$ on $V$. Hence $\E^{\Delta_b}$ describes
    an isomorphism of the locally convex algebras
    $\big(\SymTen^\bullet(V),\stpro[\Lambda]\big)$ and
    $\big(\SymTen^\bullet(V),\stpro[\Lambda+b]\big)$.
    Moreover, for fixed $X\in \SymTen^\bullet(V)^\cpl$,
    the series $\E^{z\Delta_b} X$ converges
    absolutely and locally uniformly in $z\in \CC$ and thus depends
    holomorphically on $z$.
\end{theorem}
\begin{proof}
    As
    $\abs{\Delta_b X} \le \seminorm{\alpha}{\E^{\Delta_b}
      X}^{\bullet}$
    holds for all $\seminorm{\alpha}{\argument} \in \mathcal{P}_V$ and
    all $X\in \SymTen^2(V)$, it follows from
    Proposition~\ref{prop:hschar} that continuity of the restriction
    of $\E^{\Delta_b}$ to $\SymTen^\bullet(V)$ implies that $b$
    is of Hilbert-Schmidt type. Conversely, for all $X \in \Ten^\bullet(V)$,
    all $\alpha \in \mathcal{P}_{V,b,HS}$, and $r>1$, the estimate
    \begin{equation*}
        \seminorm[\big]{\alpha}{\E^{z\Delta_b} X}
        \le
        \sum_{t=0}^\infty
        \frac{1}{t!}
        \seminorm[\big]{\alpha}{\big(z\Delta_b\big)^t(X)}
        \le
        \sum_{t=0}^\infty
        \frac{|z|^t}{(4r)^t}
        \binom{2t}{t}^{\frac{1}{2}}
        \seminorm{4r\alpha}{X}^\bullet
        \le
        \sum_{t=0}^\infty
        \frac{1}{2^t}
        \seminorm{4r\alpha}{X}^\bullet
        =
        2 \seminorm{4r\alpha}{X}^\bullet
    \end{equation*}
    holds for all $z\in \CC$ with $|z|\le r$
    due to the previous Proposition~\ref{proposition:Laplaceestimate}
    if $b$ is of Hilbert-Schmidt
    type, which proves the continuity of $\E^{z\Delta_b}$ for all
    $z\in \CC$ as well as the absolute and locally uniform convergence
    of the series $\E^{z\Delta_b}X$. The
    algebraic relation \eqref{eq:equiv} is well-known, see
    e.g. \cite[Prop.~2.18]{waldmann:2014a}. Finally, as
    $\E^{\Delta_b}$ is invertible with inverse $\E^{-\Delta_b}$,
    and because
    $\Delta_b$ and thus $\E^{\Delta_b}$ map symmetric tensors to
    symmetric ones, we conclude that the restriction of
    $\E^{\Delta_b}$ to $\SymTen^\bullet(V)$ is an isomorphism of the
    locally convex algebras $\big(\SymTen^\bullet(V),
    \stpro[\Lambda]\big)$ and $\big(\SymTen^\bullet(V),
    \stpro[\Lambda+b]\big)$.
\end{proof}

%
%

\subsection{Gel'fand Transformation}

We are now going to construct an isomorphism of the undeformed
$^*$-algebra $\big(\SymTen^\bullet(V), \vee, { }^*\big)$ to a
$^*$-algebra of smooth functions by a construction similar to the
Gel'fand transformation of commutative $C^*$-algebras.

Let $\cc{\argument}$ be a continuous antilinear involution on $V$. We
write $V_h$ for the real linear subspace of $V$ consisting of
Hermitian elements, i.e.
\begin{equation}
    \label{eq:HermitianElements}
    V_h
    :=
    \big\{v \in V \; \big| \; \cc{v} = v \big\}.
\end{equation}
The inner products compatible with the involution are denoted by
\begin{equation}
    \label{eq:IVhDef}
    \mathcal{I}_{V,h}
    :=
    \big\{
    \skal{\argument}{\argument}_\alpha \in \mathcal{I}_V
    \; \big| \;
    \cc{\skal{v}{w}_\alpha}
    =
    \skal{\cc{v}}{\cc{w}}_\alpha
    \textrm{ for all }
    v, w \in V
    \big\}.
\end{equation}
Moreover, we write $V'$ for the topological dual space of $V$ and
$V'_h$ again for the real linear subspace of $V'$ consisting of
Hermitian elements, i.e.
\begin{equation}
    \label{eq:RealDualDef}
    V'_h
    :=
    \big\{
    \rho\in V'
    \; \big| \;
    \cc{\rho(v)} = \rho(\cc{v})
    \textrm{ for all }
    v\in V
    \big\}.
\end{equation}
Finally, recall that a subset $B\subseteq V'_h$ is \emph{bounded}
(with respect to the equicontinuous bornology)
if there exists a $\skal{\argument}{\argument}_\alpha \in
\mathcal{I}_{V,h}$ such that $\abs{\rho(v)} \le \seminorm{\alpha}{v}$
holds for all $v \in V$ and all $\rho\in B$. This also gives a notion
of boundedness of functions from or to $V'_h$: A (multi-)linear
function is bounded if it maps bounded sets to bounded ones.

Note that one can identify $V'_h$ with the topological dual of $V_h$
and $\mathcal{I}_{V, h}$ with the set of continuous positive bilinear
forms on $V_h$. Moreover, $\mathcal{I}_{V, h}$ is cofinal in
$\mathcal{I}_{V}$: every
$\skal{\argument}{\argument}_\alpha \in \mathcal{I}_V$ is dominated by
$V^2 \ni (v, w) \mapsto \skal{v}{w}_\alpha +
\cc{\skal{\cc{v}}{\cc{w}}}_\alpha \in \CC$.
\begin{definition}
    \label{defi:gelfandstuff}%
    Let $\cc{\argument}$ be a continuous antilinear involution on $V$
    and $\rho \in V'_h$, then we define the derivative in direction of
    $\rho$ as the linear map
    $D_\rho\colon \Ten^\bullet(V) \rightarrow \Ten^{\bullet-1}(V)$ by
    \begin{equation}
        \label{eq:DrhoDef}
        x_1 \otimes \cdots \otimes x_k
        \; \mapsto \;
        D_\rho\big(x_1\otimes \cdots \otimes x_k\big)
        :=
        k \rho(x_k) x_1 \otimes \cdots \otimes x_{k-1}
    \end{equation}
    for all $k\in \NN$ and all $x \in V^k$. Next, we
    define the translation by $\rho$ as the linear map
    \begin{equation}
        \label{eq:TranslationUmRho}
        \tau^*_\rho
        :=
        \sum_{t=0}^\infty \frac{1}{t!}
        \big(D_\rho\big)^t\colon
        \Ten^\bullet(V) \rightarrow \Ten^\bullet(V),
    \end{equation}
    and the evaluation at $\rho$ by
    \begin{equation}
        \label{eq:DeltaFunctional}
        \delta_\rho
        :=
        \langle \argument \rangle_0 \circ \tau^*_\rho \colon
        \Ten^\bullet(V)\rightarrow \CC.
    \end{equation}
    Finally, for $k\in \NN$ and $\rho_1, \ldots, \rho_k \in V'_h$ we
    set
    $D^{(k)}_{\rho_1, \ldots, \rho_k} := D_{\rho_1} \cdots
    D_{\rho_k}\colon \Ten^\bullet(V)\rightarrow \Ten^{\bullet-k}(V)$.
\end{definition}
Note that $\tau^*_\rho$ is well-defined because for every
$X\in \Ten^\bullet(V)$ only finitely many terms contribute to the
infinite series
$\tau^*_\rho X = \sum_{t=0}^\infty \frac{1}{t!}
\big(D_\rho\big)^t(X)$.
Note also that $D_\rho$ and consequently also $\tau^*_\rho$ can be
restricted to endomorphisms of $\SymTen^\bullet(V)$. Moreover, this
restriction of $D_\rho$ is a $^*$-derivation of all the $^*$-algebras
$\big(\SymTen^\bullet(V), \stpro[\Lambda], { }^*\big)$ for all
continuous Hermitian bilinear forms $\Lambda$ on $V$ (see
\cite[Lem.~2.13, \emph{iii}]{waldmann:2014a}, the compatibility with
the $^*$-involution is clear), so that $\tau^*_\rho$ turns out to be a
unital $^*$-automorphism of these $^*$-algebras.


\begin{lemma}
    \label{lemma:comderiv}%
    Let $\cc{\argument}$ be a continuous antilinear involution on $V$
    and $\rho, \sigma\in V'_h$. Then
    \begin{equation}
        \label{eq:DrhoDsigmaTaurhoTausigma}
        \big(D_\rho D_\sigma - D_\sigma D_\rho\big)(X)
        =
        \big(\tau^*_\rho D_\sigma - D_\sigma \tau^*_\rho\big)(X)
        =
        \big(\tau^*_\rho \tau^*_\sigma - \tau^*_\sigma
        \tau^*_\rho\big)(X)
        =
        0
    \end{equation}
    holds for all $X\in \SymTen^\bullet(V)$.
\end{lemma}
\begin{proof}
    It is sufficient to show that
    $\big(D_\rho D_\sigma - D_\sigma D_\rho\big)(X) = 0$ for all
    $X\in \SymTen^\bullet(V)$, which clearly holds if $X$ is a
    homogeneous factorizing symmetric tensor and so holds for all
    $X \in \SymTen^\bullet(V)$ by linearity.
\end{proof}
\begin{lemma}
    \label{lemma:contcharacter}%
    Let $\cc{\argument}$ be a continuous antilinear involution on $V$
    and $\rho \in V'_h$. Then $D_\rho$, $\tau^*_\rho$ and
    $\delta_\rho$ are all continuous. Moreover, if
    $\seminorm{\alpha}{\argument} \in \mathcal{P}_V$ fulfills
    $\abs{\rho(v)} \le \seminorm{\alpha}{v}$, then the estimates
    \begin{gather}
        \label{eq:derivation}
        \seminorm[\big]{\alpha}{(D_\rho)^t X}^\bullet
        \le
        \sqrt{t!} \seminorm{2\alpha}{X}^\bullet
        \\
        \shortintertext{and}
        \seminorm{\alpha}{\tau^*_\rho(X)}^\bullet
        \le
        \sum_{t'=0}^\infty \frac{1}{t'!}
        \seminorm[\big]{\alpha}{(D_\rho)^{t'} X}^\bullet
        \le
        \frac{2}{\sqrt{2}-1}
        \seminorm{2\alpha}{X}^\bullet
        \label{eq:translation}
    \end{gather}
    hold for all $X\in\Ten^\bullet(V)$ and all $t\in \NN_0$.
\end{lemma}
\begin{proof}
    Let $\seminorm{\alpha}{\argument} \in \mathcal{P}_V$ be given such
    that $\abs{\rho(v)} \le \seminorm{\alpha}{v}$ holds for all
    $v \in V$. For all $d \in \NN_0$ and all
    $\skal{\argument}{\argument}_\alpha$-orthonormal $e \in V^d$ we
    then get
    \begin{equation*}
        \sum_{i=1}^d \abs{\rho(e_i)}^2
        =
        \rho\bigg(
        \sum_{i=1}^d e_i
        \cc{\rho(e_i)}
        \bigg)
        \le
        \seminorm[\bigg]{\alpha}
        {\sum_{i=1}^d e_i \cc{\rho(e_i)}}
        =
        \bigg(\sum_{i=1}^d \abs{\rho(e_i)}^2\bigg)^\frac{1}{2},
    \end{equation*}
    hence $\sum_{i=1}^d \abs{\rho(e_i)}^2 \le 1$.  Given $k\in \NN$
    and a tensor $X\in\Ten^k(V)$, then we construct
    $X_0 = \sum_{a\in A} x_{a,1} \otimes \cdots \otimes x_{a,k}$ and
    $\tilde{X} = \sum_{a'\in\{1, \ldots, d\}^k} X^{a'} e_{a'_1}
    \otimes \cdots \otimes e_{a'_k}$
    like in Lemma~\ref{lemma:helpfull}. Then we have
    $\seminorm{\alpha}{D_\rho X_0}^\bullet = 0$ because
    \begin{equation*}
        \seminorm{\alpha}
        {D_\rho (x_{a,1} \otimes \cdots \otimes x_{a,k})}^\bullet
        =
        k \abs{\rho(x_{a,k})}
        \seminorm{\alpha}
        {x_{a,1} \otimes \cdots \otimes x_{a,k-1}}^\bullet
        \le
        k \sqrt{(k-1)!}
        \prod_{m=1}^k\seminorm{\alpha}{x_{a,m}}
        =
        0
    \end{equation*}
    holds for all $a\in A$. Consequently
    $\seminorm{\alpha}{D_\rho X}^{\bullet} \le
    \seminorm{\alpha}{D_\rho \tilde{X}}^{\bullet}$ and we get
    \begin{align*}
        \seminorm{\alpha}{D_\rho X}^{\bullet,2}
        \le
        \seminorm{\alpha}{D_\rho \tilde{X}}^{\bullet, 2}
        &=
        \seminorm[\bigg]{\alpha}
        {
          \sum_{a'\in\{1, \ldots, d\}^k}
          X^{a'} D_\rho(e_{a'_1} \otimes \cdots \otimes e_{a'_k})
        }^{\bullet, 2} \\
        &=
        k^2 \sum_{\tilde{a}'\in\{1, \ldots, d\}^{k-1}}
        \seminorm[\bigg]{\alpha}
        {
          \sum_{g=1}^d
          X^{(\tilde{a}',g)}
          \rho(e_g)
          e_{\tilde{a}'_1}
          \otimes \cdots \otimes
          e_{\tilde{a}'_{k-1}}
        }^{\bullet, 2} \\
        &\le
        k^2 (k-1)!
        \sum_{\tilde{a}' \in \{1, \ldots, d\}^{k-1}}
        \bigg(
        \sum_{g=1}^d
        \abs[\big]{X^{(\tilde{a}',g)}}
        \abs{\rho(e_g)}
        \bigg)^2 \\
        &\stackrel{\cs}{\le}
        k^2 (k-1)!
        \sum_{\tilde{a}' \in \{1, \ldots, d\}^{k-1}}
        \bigg(
        \sum_{g=1}^d
        \abs[\big]{X^{(\tilde{a}',g)}}^2
        \bigg)
        \bigg(
        \sum_{g=1}^d
        \abs{\rho(e_g)}^2
        \bigg) \\
        &\le
        k^2 (k-1)!
        \sum_{a'\in\{1, \ldots, d\}^{k}}
        \abs[\big]{X^{a'}}^2 \\
        &=
        k \seminorm{\alpha}{X}^{\bullet, 2}.
    \end{align*}
    Using this we can derive the estimate \eqref{eq:derivation}, which
    also proves the continuity of $D_\rho$: If $t=0$, then this is
    clearly fulfilled. Otherwise, let $X\in \Ten^\bullet(V)$ be given,
    then
    \begin{equation*}
        \seminorm[\big]{\alpha}{(D_\rho)^t X}^{\bullet, 2}
        =
        \sum_{k=t}^\infty
        \seminorm[\big]{\alpha}
        {(D_\rho)^t\langle X \rangle_k}^{\bullet, 2}
        \le
        t!
        \sum_{k=t}^\infty
        \binom{k}{t}
        \seminorm{\alpha}{\langle X \rangle_k}^{\bullet, 2}
        \le
        t!
        \sum_{k=t}^\infty
        \seminorm{2\alpha}{\langle X \rangle_k}^{\bullet, 2}
        \le
        t!
        \seminorm{2\alpha}{X}^{\bullet, 2}.
    \end{equation*}
    From this we can now also deduce the estimate
    \eqref{eq:translation}, which then shows continuity of
    $\tau^*_\rho$ and of
    $\delta_\rho = \langle\argument\rangle_0 \circ \tau^*_\rho$: The
    first inequality is just the triangle inequality and for the
    second we use that $t! \ge 2^{t-1}$ for all $t \in \NN_0$, so
    \begin{equation*}
        \sum_{t=0}^\infty
        \frac{1}{t!}
        \seminorm[\big]{\alpha}{(D_\rho)^t X}^\bullet
        \le
        \sum_{t=0}^\infty
        \frac{1}{\sqrt{t!}}
        \seminorm{2\alpha}{X}^\bullet
        \le
        \sqrt{2}
        \sum_{t=0}^\infty
        \frac{1}{\sqrt{2}^{t}}
        \seminorm{2\alpha}{X}^\bullet
        \le
        \frac{2}{\sqrt{2}-1}
        \seminorm{2\alpha}{X}^\bullet.
    \end{equation*}
\end{proof}
\begin{proposition}
    \label{proposition:contcharacterchar}
    Let $\cc{\argument}$ be a continuous antilinear
    involution on $V$, then the set of all continuous unital
    $^*$-homomorphisms from
    $\big(\SymTen^\bullet(V)^\cpl, \vee, { }^*\big)$ to $\CC$ is
    $\big\{\delta_\rho \; \big| \; \rho \in V'_h \big\}$ (strictly
    speaking, the continuous extensions to $\SymTen^\bullet(V)^\cpl$
    of the restrictions of $\delta_\rho$ to $\SymTen^\bullet(V)$).
\end{proposition}
\begin{proof}
    On the one hand, every such $\delta_\rho$ is a continuous unital
    $^*$-homomorphism, because $\langle\argument\rangle_0$ and
    $\tau^*_\rho$ are. On the other hand, if
    $\phi\colon \big(\SymTen^\bullet(V)^\cpl, \vee, { }^*\big)
    \rightarrow \CC$
    is a continuous unital $^*$-homomorphism, then
    $V\ni v \mapsto \rho(v) := \phi(v) \in \CC$ is an element of
    $V'_h$ and fulfills $\delta_\rho = \phi$ because the unital
    $^*$-algebra $\big(\SymTen^\bullet(V), \vee, { }^*\big)$ is
    generated by $V$ and because $\SymTen^\bullet(V)$ is dense in its
    completion.
\end{proof}

Let $\Phi := \big\{\delta_\rho \; \big| \; \rho\in V'_h\big\}$ be the
set of all continuous unital $^*$-homomorphisms from
$\big(\SymTen^\bullet(V)^\cpl, \vee, { }^*\big)$ to $\CC$ and
$\CC^\Phi$ the unital $^*$-algebra of all functions from $\Phi$ to
$\CC$ with the pointwise operations, then the Gel'fand-transformation
is usually defined as the unital $^*$-homomorphism
$\widetilde{\argument}\colon \big(\SymTen^\bullet(V)^\cpl, \vee, {
}^*\big) \rightarrow \CC^\Phi$,
$X \mapsto \widetilde{X}$ with $\widetilde{X}(\phi) := \phi(X)$ for all
$\phi \in \Phi$. This is a natural way to transform an abstract commutative
unital locally convex $^*$-algebras to a $^*$-algebra of complex-valued
functions. For our purposes, however, it will be more convenient to identify
$\Phi$ with $V'_h$ like in the previous
Proposition~\ref{proposition:contcharacterchar}:
\begin{definition}
    Let $\cc{\argument}$ be a continuous antilinear
    involution on $V$ and $X \in \SymTen^\bullet(V)^\cpl$, then we
    define the function $\widehat{X}\colon V'_h \rightarrow \CC$
    by
    \begin{equation}
        \label{eq:Gelfand}
        \rho
        \; \mapsto \;
        \widehat{X}(\rho) := \delta_\rho(X).
    \end{equation}
\end{definition}

In the following we will show that this construction yields an
isomorphism between $\big(\SymTen^\bullet(V)^\cpl, \vee, { }^*\big)$
and a unital $^*$-algebra of certain functions on $V'_h$:
\begin{definition}
    Let $f\colon V'_h\rightarrow \CC$ be a function. For
    $\rho, \sigma \in V'_h$ we denote by
    \begin{equation}
        \label{eq:hatDrhoDef}
        \big(\widehat{D}_\rho f\big)(\sigma)
        :=
        \frac{\D}{\D t}\At{t=0} f(\sigma + t\rho)
    \end{equation}
    (if it exists) the directional derivative of $f$ at $\sigma$ in
    direction $\rho$. If the directional derivative of $f$ in
    direction $\rho$ exists at all $\sigma \in V'_h$, then we denote
    by $\widehat{D}_\rho f\colon V'_h\rightarrow \CC$ the function
    $\sigma \mapsto \big(\widehat{D}_\rho f\big)(\sigma)$. In this
    case we can also examine directional derivatives of
    $\widehat{D}_\rho f$ and define the iterated directional
    derivative
    \begin{equation}
        \label{eq:IteratedDerivative}
        \widehat{D}_{\rho}^{(k)}f
        :=
        \widehat{D}_{\rho_1} \cdots \widehat{D}_{\rho_k}f
    \end{equation}
    (if it exists) for $k\in \NN$ and $\rho\in (V'_h)^k$. For $k=0$ we
    define $\widehat{D}^{(0)} f := f$. Moreover, we say that $f$ is
    smooth if all iterated directional derivatives
    $\widehat{D}_{\rho}^{(k)}f$ exist for all $k\in \NN_0$ and all
    $\rho \in (V'_h)^k$ and describe a bounded symmetric multilinear
    form
    $(V'_h)^k \ni \rho \mapsto
    \big(\widehat{D}_{\rho}^{(k)}f\big)(\sigma) \in \CC$
    for all $\sigma \in V'_h$. Finally, we write $\smooth(V'_h)$ for
    the unital $^*$-algebra of all smooth functions on $V'_h$.
\end{definition}
Note that this notion of smoothness is rather weak, we do not even
demand that a smooth function is continuous (we did not even endow
$V'_h$ with a topology). For example, every bounded linear functional
on $V'_h$ is smooth.

\begin{proposition}
    \label{prop:deriv}%
    Let $\cc{\argument}$ be a continuous antilinear
    involution on $V$ and $X\in \SymTen^\bullet(V)^\cpl$. Then
    $\widehat{X}\colon V'_h\rightarrow \CC$ is smooth and
    \begin{equation}
        \label{eq:deriv}
        \widehat{D}^{(k)}_{\rho}\widehat{X}
        =
        \widehat{D^{(k)}_{\rho} X}
    \end{equation}
    holds for all $k\in \NN_0$ and all $\rho \in (V'_h)^k$.
\end{proposition}
\begin{proof}
    Let $X\in\SymTen^\bullet(V)^\cpl$ be given. As the exponential
    series $\tau^*_{t\rho}(X)$ is absolutely convergent by
    Lemma~\ref{lemma:contcharacter}, it follows that
    $\frac{\D}{\D t}\at{t=0}\tau^*_{t\rho}(X) = D_\rho(X)$ for all
    $\rho\in V'_h$ and so we conclude that
    \begin{equation*}
        \big(\widehat{D}_\rho \widehat{X}\big)(\sigma)
        =
        \frac{\D}{\D t}\At{t=0} \delta_{\sigma+t\rho}(X)
        =
        \bigg\langle
        \tau^*_\sigma\bigg(
        \frac{\D}{\D t}\At{t=0}\tau^*_{t\rho}(X)
        \bigg)
        \bigg\rangle_{0}
        =
        \big\langle\tau^*_\sigma\big(D_\rho(X)\big)\big\rangle_0
        =
        \widehat{D_\rho(X)}(\sigma)
    \end{equation*}
    holds for all $\rho, \sigma \in V'_h$, which proves
    \eqref{eq:deriv} in the case $k = 1$.  We see that
    $\widehat{D}_\rho$ for all $\rho\in V'_h$ is an endomorphism of
    the vector space
    $\big\{\widehat{X} \;\big| \; X \in \SymTen^\bullet(V)^\cpl
    \big\}$,
    so all iterated directional derivatives of such an $\widehat{X}$
    exist. By induction it is now easy to see that \eqref{eq:deriv}
    holds for arbitrary $k\in \NN_0$. Moreover,
    $D_\rho D_{\rho'} X = D_{\rho'} D_\rho X$ holds for all
    $\rho, \rho' \in V'_h$ and all $X\in \SymTen^\bullet(V)^\cpl$ by
    Lemmas~\ref{lemma:comderiv} and
    \ref{lemma:contcharacter}. Together with \eqref{eq:deriv} this
    shows that directional derivatives on $\widehat{X}$ commute.
    Finally, the multilinear form
    $(V'_h)^k \ni \rho \mapsto
    \big(\widehat{D}_{\rho}^{(k)}\widehat{X}\big)(\sigma) \in \CC$
    is bounded for all $\sigma\in V'_h$: It is sufficient to show this
    for $\sigma = 0$, because $\tau^*_\sigma$ is a continuous
    automorphism of $\SymTen^\bullet(V)$ and commutes with
    $D_{\rho}^{(k)}$. If $\rho \in (V'_h)^k$ fulfills
    $\abs{\rho_i(v)} \le \seminorm{\alpha}{v}$ for all
    $i \in \{1, \ldots, k\}$, all $v\in V$ and one
    $\seminorm{\alpha}{\argument} \in \mathcal{P}_V$, then we have
    $\seminorm{\alpha}{D_{\rho_1} \cdots D_{\rho_k} X}^\bullet \le
    \seminorm{2^k\alpha}{X}$
    due to Lemma~\ref{lemma:contcharacter}, which is an upper bound of
    $\big(\widehat{D}^{(k)}_{\rho} \widehat{X}\big)(0)$.
\end{proof}

Let $\cc{\argument}$ be a continuous antilinear involution on $V$ and
let $\skal{\argument}{\argument}_\alpha \in \mathcal{I}_{V, h}$ be
given, then the degeneracy space of the inner product
$\skal{\argument}{\argument}_\alpha$ is
\begin{equation}
    \label{eq:Kernelh}
    \Kern_h \seminorm{\alpha}{\argument}
    :=
    \big\{v \in V_h \;\big| \; \seminorm{\alpha}{v} = 0\big\}.
\end{equation}
Thus we get a well-defined non-degenerate positive bilinear form on the
real vector space $V_h \big/ \Kern_h \seminorm{\alpha}{\argument}$. We
write $V^\cpl_{h,\alpha}$ for the completion of this space to a real
Hilbert space with inner product $\skal{\argument}{\argument}_\alpha$
and define the linear map $\argument^{\flat_\alpha}$ from
$V^\cpl_{h, \alpha}$ to $V'_h$ as
\begin{equation}
    \label{eq:FlatMap}
    v^{\flat_\alpha}(w)
    :=
    \skal{v}{w}_\alpha
\end{equation}
for all $v \in V^\cpl_{h,\alpha}$ and all $w\in V$. Note that
$\argument^{\flat_\alpha}\colon V^\cpl_{h, \alpha} \rightarrow V'_h$
is a bounded linear map due to the Cauchy-Schwarz inequality.
Analogously, we define
\begin{equation}
    \label{eq:KernelSeminorm}
    \Kern \seminorm{\alpha}{\argument}^\bullet
    :=
    \big\{
    X \in \Ten^\bullet(V)
    \; \big| \;
    \seminorm{\alpha}{X}^\bullet = 0
    \big\},
\end{equation}
and denote by $\Ten^\bullet(V)^\cpl_\alpha$ the completion of the
complex vector space
$\Ten^\bullet_\alg(V) \big/ \Kern
\seminorm{\alpha}{\argument}^\bullet$
to a complex Hilbert space with inner product
$\skal{\argument}{\argument}_\alpha^\bullet$. Then
$\SymTen^\bullet(V)^\cpl_\alpha$ becomes the linear subspace of
(equivalence classes of) symmetric tensors, which is closed because
$\SymFkt^\bullet$ extends to a continuous endomorphism of
$\Ten^\bullet(V)^\cpl_\alpha$ by Proposition~\ref{prop:symcont}.

Moreover, for all $\skal{\argument}{\argument}_\alpha,
\skal{\argument}{\argument}_\beta \in \mathcal{I}_{V,h}$ with
$\skal{\argument}{\argument}_\beta \le \skal{\argument}{\argument}_\alpha$,
the linear map
$\id_{\Ten^\bullet(V)}\colon \Ten^\bullet(V)\rightarrow \Ten^\bullet(V)$
extends to continuous linear maps
$\iota_{\infty\alpha}\colon \Ten^\bullet(V)^\cpl
\rightarrow \Ten^\bullet(V)^\cpl_\alpha$
and
$\iota_{\alpha\beta}\colon \Ten^\bullet(V)^\cpl_\alpha
\rightarrow \Ten^\bullet(V)^\cpl_\beta$, such that
$\iota_{\alpha\beta} \circ \iota_{\infty\alpha} =
\iota_{\infty\beta}$ and
$\iota_{\beta\gamma} \circ \iota_{\alpha\beta} =
\iota_{\alpha\gamma}$
hold for all $\skal{\argument}{\argument}_\alpha,
\skal{\argument}{\argument}_\beta,
\skal{\argument}{\argument}_\gamma \in \mathcal{I}_{V,h}$ with
$\skal{\argument}{\argument}_\gamma \le
\skal{\argument}{\argument}_\beta \le
\skal{\argument}{\argument}_\alpha$.
This way, $\Ten^\bullet(V)^\cpl$ is realized as the
projective limit of the Hilbert spaces
$\Ten^\bullet(V)^\cpl_\alpha$ and similarly,
$\SymTen^\bullet(V)^\cpl$ as the
projective limit of the closed linear subspaces
$\SymTen^\bullet(V)^\cpl_\alpha$.

\begin{lemma}
  \label{lemma:coordchange}
  Let $\cc{\argument}$ be a continuous antilinear involution on
  $V$ and $f\in \smooth(V'_h)$. Given $\rho \in V'_h$ and
  $\skal{\argument}{\argument}_\alpha \in \mathcal{I}_{V,h}$ such
  that $|\rho(v)|\le \seminorm{\alpha}{v}$ holds for all $v\in V$,
  then
  \begin{equation}
    \widehat{D}_\rho f = \sum_{i\in I} \rho(e_i) \widehat{D}_{e_i^{\flat_\alpha}} f
  \end{equation}
  holds for every Hilbert basis $e\in(V^\cpl_{h,\alpha})^I$ of
  $V^\cpl_{h,\alpha}$ indexed by a set $I$.
\end{lemma}
\begin{proof}
  As $f$ is smooth, the function
  $V'_h \ni \sigma \mapsto \widehat{D}_\sigma f \in \CC$ is bounded,
  which implies that its restriction to the dual space of
  $V^\cpl_{h,\alpha}$ is continuous with respect to the Hilbert space
  topology on (the dual of) $V^\cpl_{h,\alpha}$. As
  $\rho = \sum_{i\in I} e_i^{\flat_\alpha} \rho(e_i)$ with respect to
  this topology, it follows that
  $\widehat{D}_\rho f = \sum_{i\in I} \rho(e_i)
  \widehat{D}_{e_i^{\flat_\alpha}} f$.
\end{proof}

\begin{definition}[Hilbert-Schmidt type functions]
    \label{definition:HSFunctions}%
    Let $\cc{\argument}$ be a continuous antilinear involution on
    $V$. We say that a function $f\colon V'_h\rightarrow \CC$ is
    analytic of Hilbert-Schmidt type, if it is smooth and additionally
    fulfills the condition that for all $\sigma,\sigma'\in V'_h$ and
    all $\skal{\argument}{\argument}_\alpha \in \mathcal{I}_{V,h}$
    there exists a $C_{\sigma, \sigma', \alpha} \in \RR$ such that
    \begin{equation}
        \label{eq:hsanalytic}
        \sum_{k=0}^\infty
        \frac{1}{k!}
        \sum_{i \in I^k}
        \abs[\Big]{
          \big(
          \widehat{D}^{(k)}_{
            (e_{i_1}^{\flat_\alpha},
            \ldots,
            e_{i_k}^{\flat_\alpha})
          }
          f\big)
          (\xi)
        }^2
        \le
        C_{\sigma, \sigma', \alpha}
    \end{equation}
    holds for one Hilbert base $e \in (V^\cpl_{h,\alpha})^I$ of
    $V^\cpl_{h,\alpha}$ indexed by a set $I$ and every $\xi$ from the
    line-segment between $\sigma$ and $\sigma'$, i.e. every
    $\xi = \lambda \sigma + (1-\lambda)\sigma'$ with
    $\lambda \in [0,1]$. We write $\analytic(V'_h)$ for the set of all
    complex functions on $V'_h$ that are analytic of Hilbert-Schmidt
    type.
\end{definition}
Here and elsewhere a sum over an uncountable Hilbert basis is
understood in the usual sense: only countably many terms in the sum
are non-zero.

This definition is independent of the choice of the Hilbert basis due
to Lemma~\ref{lemma:coordchange} and $\analytic(V'_h)$ is a complex
vector space. It is not too hard to check that $\analytic(V'_h)$ is
even a unital $^*$-subalgebra of $\smooth(V'_h)$. However, we will
indirectly prove this later on. Calling the functions in
$\analytic(V'_h)$ \emph{analytic} is justified thanks to the following
statement:
\begin{proposition}
    \label{prop:analytic}%
    Let $\cc{\argument}$ be a continuous antilinear involution on $V$
    and $f\colon V'_h \rightarrow \CC$ analytic of Hilbert-Schmidt
    type with $\big(\widehat{D}^{(k)}_\rho f\big)(0) = 0$ for all
    $k\in \NN_0$ and all $\rho \in (V'_h)^k$. Then $f = 0$.
\end{proposition}
\begin{proof}
    Given $\sigma \in V'_h$, then define the smooth function
    $g\colon \RR \rightarrow \CC$ by $t\mapsto g(t) := f(t\sigma)$. We
    write $g^{(k)}(t)$ for the $k$-th derivative of $g$ at $t$. Then
    there exists a
    $\skal{\argument}{\argument}_\alpha \in \mathcal{I}_{V,h}$ that
    fulfills $\abs{\sigma(v)} \le \seminorm{\alpha}{v}$ for all
    $v\in V$, and consequently $\sigma = \nu e^{\flat_\alpha}$ with a
    normalized $e\in V^\cpl_{h, \alpha}$ and $\nu \in [0,1]$ by the
    Fréchet-Riesz theorem. Therefore,
    \begin{equation*}
        \bigg(
        \sum_{k=0}^\infty
        \frac{1}{k!}
        \abs[\big]{g^{(k)}(t)}
        \bigg)^2
        \stackrel{\cs}{\le}
        \sum_{k=0}^\infty
        \frac{1}{k!}
        \sum_{\ell=0}^\infty
        \frac{1}{\ell!}
        \abs[\big]{g^{(\ell)}(t)}^2
        \le
        \E \sum_{\ell=0}^\infty
        \frac{\nu^{2\ell}}{\ell!}
        \abs[\Big]{
          \big(\widehat{D}^{(\ell)}_{
            (e^{\flat_\alpha},
            \ldots,
            e^{\flat_\alpha})
          }
          f\big)
          (t\sigma)
        }^2
        \le
        \E C_{-2\sigma, 2\sigma, \alpha}
    \end{equation*}
    holds for all $t\in [-2,2]$ with a constant
    $C_{-2\sigma, 2\sigma, \alpha}\in \RR$, which shows that $g$ is an
    analytic function on $]-2,2[$. As $g^{(k)}(0) = 0$ for all
    $k\in \NN_0$ this implies $f(\sigma) = g(1) = 0$.
\end{proof}

Note that one can derive even better estimates for the derivatives of
$g$. This shows that condition \eqref{eq:hsanalytic} is even stronger
than just analyticity.

\begin{definition}
Let $\cc{\argument}$ be a continuous antilinear involution on $V$ and
let $f, g\colon V'_h\rightarrow \CC$ be analytic of Hilbert-Schmidt
type as well as
$\skal{\argument}{\argument}_\alpha\in\mathcal{I}_{V,h}$. Because of
the estimate \eqref{eq:hsanalytic} we can define a function
$\ptwskal{f}{g}_\alpha^\bullet\colon V'_h \rightarrow \CC$
by
\begin{equation}
    \label{eq:FunnyptwskalfgFunction}
    \rho
    \; \mapsto \;
    \ptwskal{f}{g}_\alpha^\bullet(\rho)
    :=
    \sum_{k=0}^\infty
    \frac{1}{k!}\sum_{i \in I^k}
    \cc{\Big(\widehat{D}^{(k)}_{e_{i}^{\flat_\alpha}}f\Big)(\rho)}
    \Big(\widehat{D}^{(k)}_{e_{i}^{\flat_\alpha}} g\Big)(\rho),
\end{equation}
where $e \in (V^\cpl_{h,\alpha})^I$ is an arbitrary Hilbert base of
$V^\cpl_{h,\alpha}$ indexed by a set $I$.
\end{definition}

Note that
$\ptwskal{f}{g}_\alpha^\bullet$ does not depend on the choice of this
Hilbert base due to Lemma~\ref{lemma:coordchange}.
Essentially, $\ptwskal{f}{g}_\alpha^\bullet(\rho)$ is a weighted
$\ell^2$-inner product (yet not necessarily positive-definite) of all
partial derivatives of $f$ and $g$ at $\rho$ in directions described
by (the dual of) a $\skal{\argument}{\argument}_\alpha$-Hilbert
base. Note that the analyticity condition \eqref{eq:hsanalytic} for a
function $f$ is equivalent to demanding that
$\ptwskal{f}{f}_\alpha^\bullet(\xi)$ exists for all $\xi \in V'_h$ and
all $\skal{\argument}{\argument}_\alpha \in \mathcal{I}_{V,h}$ and is
uniformly bounded on line segments in $V'_h$.

\begin{lemma}
    \label{lemma:partial}%
    Let $\cc{\argument}$ be a continuous antilinear
    involution on $V$. Let $k \in \NN$ and $x \in (V_h)^k$ as well as
    $\skal{\argument}{\argument}_\alpha \in \mathcal{I}_{V,h}$ be
    given. Then
    \begin{equation}
        \label{eq:partial}
        \big(
        \widehat{D}_{x^{\flat_\alpha}}^{(k)}\widehat{Y}
        \big)(0)
        =
        \big\langle D_{x^{\flat_\alpha}}^{(k)} Y \big\rangle_0
        =
        \skal{x_1 \otimes \cdots \otimes x_k}{Y}_\alpha^\bullet
    \end{equation}
    holds for all $Y \in \SymTen^\bullet(V)^\cpl$.
\end{lemma}
\begin{proof}
    The first identity is just Proposition~\ref{prop:deriv}, and for the
    second one it is sufficient to show that
    $\big\langle D_{x^{\flat_\alpha}}^{(k)} Y \big\rangle_0 =
    \skal{x_1\otimes \cdots \otimes x_k}{Y}_\alpha^\bullet$
    holds for all factorizing tensors $Y$ of degree $k$, because both sides of this
    equation vanish on homogeneous tensors of different degree and are
    linear and continuous in $Y$ by
    Lemma~\ref{lemma:contcharacter}. However, it is an immediate
    consequence of the definitions of $D$,
    $\argument^{\flat_\alpha}$, and
    $\skal{\argument}{\argument}_\alpha^\bullet$ that
    \begin{equation*}
        \Big\langle
        D_{
          (x^{\flat_\alpha}_1,
          \ldots,
          x^{\flat_\alpha}_k)
        }^{(k)}
        y_1 \otimes \cdots \otimes y_k
        \Big\rangle_0
        =
        k!\prod_{m=1}^k \skal{x_m}{y_m}_\alpha
        =
        \skal{x_1 \otimes \cdots \otimes x_k}
        {y_1 \otimes \cdots \otimes y_k}_\alpha^\bullet
    \end{equation*}
    holds for all $y_1, \ldots, y_k \in V$.
\end{proof}

\begin{proposition}
    \label{prop:ptwskalinnerprod}%
    Let $\cc{\argument}$ be a continuous antilinear
    involution on $V$, then
    \begin{align}
        \ptwskal[\big]{\widehat{X}}{\widehat{Y}}_\alpha^\bullet(\rho)
        =
        \skal{\tau^*_\rho X}{\tau^*_\rho Y}_\alpha^\bullet
        =
        \widehat{X^* \stpro[\Lambda_\alpha] Y}(\rho)
    \end{align}
    holds for all $X, Y \in \SymTen^\bullet(V)^\cpl$, all
    $\rho \in V'_h$, and all
    $\skal{\argument}{\argument}_\alpha \in \mathcal{I}_{V,h}$, where
    $\Lambda_\alpha\colon V \times V \rightarrow \CC$ is the
    continuous bilinear form defined by
    $\Lambda_\alpha(v, w) := \skal{\cc{v}}{w}_\alpha$.
\end{proposition}
\begin{proof}
    Let $X, Y \in \SymTen^\bullet(V)^\cpl$, $\rho \in V'_h$ and
    $\skal{\argument}{\argument}_\alpha \in \mathcal{I}_{V,h}$ be
    given. Let $e \in (V_{h,\alpha}^\cpl)^I$ be a Hilbert base of
    $V_{h,\alpha}^\cpl$ indexed by a set $I$. Then
    \begin{align*}
        \ptwskal[\big]{\widehat{X}}{\widehat{Y}}_\alpha^\bullet(\rho)
        &=
        \sum_{k=0}^\infty
        \frac{1}{k!}
        \sum_{i \in I^k}
        \cc{
          \Big(
          \widehat{D}^{(k)}_{e_{i}^{\flat_\alpha}}
          \widehat{X}
          \Big)(\rho)
        }
        \Big(
        \widehat{D}^{(k)}_{e_{i}^{\flat_\alpha}}
        \widehat{Y}
        \Big)(\rho) \\
        &=
        \sum_{k=0}^\infty
        \frac{1}{k!}
        \sum_{i \in I^k}
        \cc{
          \Big\langle
          D^{(k)}_{e_{i}^{\flat_\alpha}} \tau^*_\rho X
          \Big\rangle_0
        }
        \Big\langle
        D^{(k)}_{e_{i}^{\flat_\alpha}} \tau^*_\rho  Y
        \Big\rangle_0 \\
        &=
        \sum_{k=0}^\infty
        \sum_{i \in I^k}
        \frac{1}{k!}
        \skal[\big]{\tau^*_\rho X}
        {e_{i_1} \otimes \cdots \otimes e_{i_k}}_\alpha^\bullet
        \skal[\big]{e_{i_1} \otimes \cdots \otimes e_{i_k}}
        {\tau^*_\rho Y}_\alpha^\bullet \\
        &=
        \skal[\big]{\tau^*_\rho X}{\tau^*_\rho Y}^\bullet_\alpha
    \end{align*}
    holds by Proposition~\ref{prop:deriv} and
    Lemma~\ref{lemma:comderiv} as well as the previous
    Lemma~\ref{lemma:partial} and the fact that the tensors
    $(k!)^{1/2} e_{i_1} \otimes \cdots \otimes e_{i_k}$ for all
    $k\in \NN_0$ and $i \in I^k$ form a Hilbert base of
    $\Ten^\bullet(V)^\cpl_\alpha$. The second identity is a direct
    consequence of Theorem~\ref{theo:char} because $\tau^*_\rho$ is a
    unital $^*$-automorphism of $\stpro[\Lambda_\alpha]$. Indeed, we have
    \begin{equation*}
        \skal[\big]{\tau^*_\rho X}{\tau^*_\rho Y}_\alpha^\bullet
        =
        \big\langle
        (\tau^*_\rho X)^* \stpro[\Lambda_\alpha] (\tau^*_\rho Y)
        \big\rangle_0
        =
        \big\langle
        \tau^*_\rho \big(X^* \stpro[\Lambda_\alpha] Y\big)
        \big\rangle_0
        =
        \widehat{X^* \stpro[\Lambda_\alpha] Y}(\rho).
    \end{equation*}
\end{proof}
\begin{corollary}
    \label{cor:analytic}%
    Let $\cc{\argument}$ be a continuous antilinear
    involution on $V$ and $X \in \SymTen^\bullet(V)^\cpl$, then
    $\widehat{X} \in \analytic(V'_h)$.
\end{corollary}
\begin{proof}
    The function $\widehat{X}$ is smooth by
    Proposition~\ref{prop:deriv}. By the previous
    Proposition~\ref{prop:ptwskalinnerprod}, we have
    \begin{align*}
        \sum_{k=0}^\infty
        \frac{1}{k!}
        \sum_{i \in I^k}
        \abs[\Big]{
          \Big(
          \widehat{D}^{(k)}_{e_{i}^{\flat_\alpha}}
          \widehat{X}
          \Big)(\xi)
        }^2
        =
        \ptwskal[\big]{\widehat{X}}{\widehat{X}}^\bullet_\alpha(\xi)
        =
        \widehat{X^* \stpro[\Lambda_\alpha] X}(\xi)
    \end{align*}
    for all $\skal{\argument}{\argument}_\alpha\in\mathcal{I}_{V,h}$,
    which is finite and depends smoothly on $\xi\in V'_h$ by
    Proposition~\ref{prop:deriv} again. Therefore it is uniformly
    bounded on line segments.
\end{proof}

\begin{lemma}
    \label{lemma:nocheins}%
    Let $\cc{\argument}$ be a continuous antilinear
    involution on $V$ and
    $\skal{\argument}{\argument}_\alpha \in \mathcal{I}_{V,h}$. For
    every $f\in\analytic(V'_h)$ there exists an
    $X_f \in \SymTen^\bullet(V)^\cpl$
    that fulfills
    $\ptwskal{f}{f}^\bullet_\alpha(0) =
    \ptwskal{\widehat{X}_f}{\widehat{X}_f}^\bullet_\alpha(0)$ and
    $\ptwskal{f}{\widehat{Y}}^\bullet_\alpha(0) =
    \ptwskal{\widehat{X}_f}{\widehat{Y}}^\bullet_\alpha(0)$
    for all $Y \in \SymTen^\bullet(V)^\cpl$ and all
    $\skal{\argument}{\argument}_\alpha \in \mathcal{I}_{V,h}$.
\end{lemma}
\begin{proof}
    For every $\alpha\in\mathcal{I}_{V,h}$
    construct $X_{f, \alpha} \in \SymTen^\bullet(V)^\cpl_\alpha$ as
    \begin{equation*}
        X_{f,\alpha}
        :=
        \sum_{k=0}^\infty
        \frac{1}{k!}
        \sum_{i\in I^k}
        e_{i_1} \otimes \cdots \otimes e_{i_k}
        \Big(\widehat{D}^{(k)}_{e_{i}^{\flat_\alpha}} f \Big)(0)
        \in
        \SymTen^\bullet(V)^\cpl_\alpha,
    \end{equation*}
    where $e \in (V^\cpl_{h,\alpha})^I$ is a Hilbert base of
    $V^\cpl_{h,\alpha}$ indexed by a set $I$. This infinite sum
    $X_{f,\alpha}$ indeed lies in $\SymTen^\bullet(V)^\cpl_\alpha$
    and fulfills
    $\skal{X_{f,\alpha}}{X_{f,\alpha}}^\bullet_\alpha =
    \ptwskal{f}{f}^\bullet_\alpha(0)$,
    because $\big(\widehat{D}^{(k)}_{e_{i}^{\flat_\alpha}} f \big)(0)$
    is invariant under permutations of the $e_{i_1}, \ldots, e_{i_k}$
    due to the smoothness of $f$ and because
    \begin{align*}
        &\sum_{k,\ell=0}^\infty
        \sum_{i \in I^k, i' \in I^\ell}
        \frac{1}{k! \ell!}
        \skal[\Big]{
          e_{i_1} \otimes \cdots \otimes e_{i_k}
          \Big(\widehat{D}^{(k)}_{e_{i}^{\flat_\alpha}} f \Big)(0)
        }
        {
          e_{i'_1} \otimes \cdots \otimes e_{i'_\ell}
          \Big(\widehat{D}^{(\ell)}_{e_{i}^{\flat_\alpha}} f \Big)(0)
        }_\alpha^\bullet
        \\
        &\quad=
        \sum_{k=0}^\infty
        \sum_{i \in I^k}
        \frac{1}{k!}
        \abs[\Big]{
          \Big(\widehat{D}_{e_{i}^{\flat_\alpha}}^{(k)} f \Big)(0)
        }^2 \\
        &\quad=
        \ptwskal{f}{f}^\bullet_\alpha(0).
    \end{align*}
    Moreover, for all $Y\in \SymTen^\bullet(V)^\cpl$ the identity
    \begin{align*}
        \ptwskal[\big]{f}{\widehat{Y}}^\bullet_\alpha(0)
        &=
        \sum_{k=0}^\infty
        \frac{1}{k!}
        \sum_{i \in I^k}
        \cc{\Big(\widehat{D}^{(k)}_{e_{i}^{\flat_\alpha}} f\Big)(0)}
        \Big(
	  \widehat{D}^{(k)}_{e_{i}^{\flat_\alpha}}\widehat{Y}
        \Big)(0) \\
        &=
        \sum_{k=0}^\infty
        \frac{1}{k!}
        \sum_{i \in I^k}
        \skal[\big]{
	  X_{f, \alpha}
	}{
	  e_{i_1} \otimes \cdots \otimes e_{i_k}
	}^\bullet_\alpha
        \skal[\big]{
	  e_{i_1} \otimes \cdots \otimes e_{i_k}
	}{
	  Y
	}^\bullet_\alpha \\
        &=
        \skal[\big]{X_{f, \alpha}}{Y}^\bullet_\alpha
    \end{align*}
    holds due to the construction of $X_{f,\alpha}$ and
    Lemma~\ref{lemma:partial} and because the tensors
    $(k!)^{1/2} e_{i_1}\otimes \cdots \otimes e_{i_k}$ for all
    $k\in \NN_0$ and all $i\in I^k$ are a Hilbert base of
    $\Ten^\bullet(V)^\cpl_\alpha$.

    Next, let $\skal{\argument}{\argument}_\beta \in \mathcal{I}_{V,h}$
    with $\skal{\argument}{\argument}_\beta \le
    \skal{\argument}{\argument}_\alpha$ and a Hilbert basis
    $d \in (V^\cpl_{h,\beta})^J$ of $V^\cpl_{h,\beta}$
    indexed by a set $J$ be given. Using the
    explicit formulas and the identity
    \begin{equation*}
      \Big(\widehat{D}^{(k)}_{d_{j}^{\flat_\beta}} f\Big)(0)
      =
      \frac{1}{k!}
      \sum_{i\in I^k}
      \Big(\widehat{D}^{(k)}_{e_{i}^{\flat_\alpha}} f\Big)(0)
      \skal{
	d_{j_1}\otimes \dots \otimes d_{j_k}
      }{
	\iota_{\alpha\beta}(
	  e_{i_1}\otimes \dots \otimes e_{i_k}
	)
      }_\beta^\bullet
    \end{equation*}
    from Lemma~\ref{lemma:coordchange}
    one can now calculate that
    \begin{align*}
      \iota_{\alpha\beta}(X_{f,\alpha})
      &=
      \sum_{k=0}^\infty
      \frac{1}{k!}
      \sum_{i\in I^k}
      \iota_{\alpha\beta}(
	e_{i_1}\otimes \dots \otimes e_{i_k}
      )
      \Big(
	\widehat{D}^{(k)}_{e_{i}^{\flat_\alpha}} f
      \Big)(0) \\
      &=
      \sum_{k=0}^\infty
      \frac{1}{(k!)^2}
      \sum_{i\in I^k}
      \sum_{j\in J^k}
      d_{j_1}\otimes \dots \otimes d_{j_k}
      \skal{
	d_{j_1}\otimes \dots \otimes d_{j_k}
      }{
	\iota_{\alpha\beta}(
	  e_{i_1}\otimes \dots \otimes e_{i_k}
	 )
      }^\bullet_\beta
      \Big(
	\widehat{D}^{(k)}_{e_{i}^{\flat_\alpha}} f
      \Big)(0) \\
      &=
      \sum_{k=0}^\infty
      \frac{1}{k!}
      \sum_{j\in J^k}
      d_{j_1}\otimes \dots \otimes d_{j_k}
      \Big(
	\widehat{D}^{(k)}_{d_{j}^{\flat_\beta}} f
      \Big)(0) \\
      &=
      X_{f,\beta}\,.
    \end{align*}
    As $\SymTen^\bullet(V)^\cpl$ is the projective limit
    of the Hilbert spaces $\SymTen^\bullet(V)^\cpl_\alpha$,
    this implies that there exists a unique
    $X_f \in \SymTen^\bullet(V)^\cpl$
    that fulfills
    $\iota_{\infty\alpha}(X_{f}) = X_{f,\alpha}$
    for all
    $\skal{\argument}{\argument}_\alpha \in \mathcal{I}_{V,h}$.
    Consequently and with the help of
    Proposition~\ref{prop:ptwskalinnerprod},
    \begin{equation*}
      \ptwskal{\widehat{X}_{f}}{\widehat{Y}}_\alpha^\bullet(0)
      =
      \skal{X_{f}}{Y}_\alpha^\bullet
      =
      \skal{\iota_{\infty\alpha}(X_{f})}{Y}_\alpha^\bullet
      =
      \skal{X_{f,\alpha}}{Y}_\alpha^\bullet
      =
      \ptwskal{f}{\widehat{Y}}_\alpha^\bullet(0)
    \end{equation*}
    holds for all $Y\in\SymTen^\bullet(V)^\cpl$ and all
    $\skal{\argument}{\argument}_\alpha \in \mathcal{I}_{V,h}$, and
    similarly,
    \begin{equation*}
      \ptwskal{\widehat{X}_{f}}{\widehat{X}_{f}}_\alpha^\bullet(0)
      =
      \skal{X_{f}}{X_f}_\alpha^\bullet
      =
      \skal{
        \iota_{\infty\alpha}(X_{f})
      }{
        \iota_{\infty\alpha}(X_{f})
      }_\alpha^\bullet
      =
      \skal{X_{f,\alpha}}{X_{f,\alpha}}_\alpha^\bullet
      =
      \ptwskal{f}{f}_\alpha^\bullet(0).
    \end{equation*}
\end{proof}
After this preparation we are now able to identify the image of the
Gel'fand transform explicitly:
\begin{theorem}
    \label{theo:gelfand}%
    Let $\cc{\argument}$ be a continuous antilinear
    involution on $V$, then the Gel'fand transformation
    $\widehat{\argument}\colon \big(\SymTen^\bullet(V)^\cpl, \vee,
    { }^*\big) \rightarrow \analytic(V'_h)$
    is an isomorphism of unital $^*$-algebras.
\end{theorem}
\begin{proof}
    Let $X\in \SymTen^\bullet(V)^\cpl$ be given, then
    $\widehat{X} \in \analytic(V'_h)$ by
    Corollary~\ref{cor:analytic}. The Gel'fand transformation is a
    unital $^*$-homomorphism onto its image by construction and
    injective because $\widehat{X} = 0$ implies
    $\skal{X}{X}^\bullet_\alpha =
    \ptwskal{\widehat{X}}{\widehat{X}}_\alpha^\bullet(0) = 0$
    for all $\skal{\argument}{\argument}_\alpha \in \mathcal{I}_{V,h}$
    by Proposition~\ref{prop:ptwskalinnerprod}, hence $X = 0$.
    It only remains to show that
    $\widehat{\argument}$ is surjective, so let $f\in \analytic(V'_h)$
    be given. Construct $X_{f} \in \SymTen^\bullet(V)^\cpl$
    like in the previous Lemma~\ref{lemma:nocheins},
    then
    \begin{align*}
      \ptwskal{f-\widehat{X}_f}{f-\widehat{X}_f}_\alpha^\bullet(0)
      &=
      \ptwskal{f}{f}_\alpha^\bullet(0)
      -
      \ptwskal{f}{\widehat{X}_f}_\alpha^\bullet(0)
      -
      \ptwskal{\widehat{X}_f}{f}_\alpha^\bullet(0)
      +
      \ptwskal{\widehat{X}_f}{\widehat{X}_f}_\alpha^\bullet(0) \\
      &=
      \ptwskal{f}{f}_\alpha^\bullet(0)
      -
      \ptwskal{\widehat{X}_f}{\widehat{X}_f}_\alpha^\bullet(0)
      -
      \ptwskal{\widehat{X}_f}{\widehat{X}_f}_\alpha^\bullet(0)
      +
      \ptwskal{\widehat{X}_f}{\widehat{X}_f}_\alpha^\bullet(0) \\
      &=
      0
    \end{align*}
    holds for all
    $\skal{\argument}{\argument}_\alpha \in \mathcal{I}_{V,h}$,
    hence $f=\widehat{X}_f$ due to Proposition~\ref{prop:analytic}.
\end{proof}
\begin{remark}
    Let $\cc{\argument}$ be a continuous antilinear involution on
    $V$. For a continuous bilinear form $\Lambda$ on $V$ the identity
    \begin{equation}
        \label{eq:PLambdaXtensY}
        \P_{\Lambda}(X \otimes_\pi Y)
        =
        \sum_{i, i' \in I} \Lambda(e_i, e_{i'})
        \big(
        D_{e_i^{\flat_\alpha}} X
        \otimes_\pi
        D_{e_{i'}^{\flat_\alpha}}Y
        \big)
    \end{equation}
    holds for all $X, Y \in \SymTen^\bullet(V)$ and every
    $\skal{\argument}{\argument}_\alpha \in \mathcal{I}_{V,h}$ for
    which $\seminorm{\alpha}{\argument} \in \mathcal{P}_{V,\Lambda}$
    and for every Hilbert base $e \in (V^\cpl_{h,\alpha})^I$ indexed
    by a set $I$. Thus
    \begin{align*}
        \widehat{X}
        \widehat{\stpro[\Lambda]}
        \widehat{Y}
        :=
        \widehat{X \stpro[\Lambda] Y}
        =
        \mu\bigg(
        \sum_{t=0}^\infty
        \frac{1}{t!}
        \bigg(
        \sum_{i, i' \in I}
        \Lambda(e_i, e_{i'})
        \big(
        \widehat{D}_{e_i^{\flat_\alpha}}
        \otimes
        \widehat{D}_{e_{i'}^{\flat_\alpha}}
        \big)
        \bigg)^t
        \big(\widehat{X} \otimes \widehat{Y}\big)
        \bigg)
    \end{align*}
    with
    $\mu\colon \smooth(V'_h)\otimes \smooth(V'_h)\rightarrow
    \smooth(V'_h)$
    the pointwise product is the usual exponential star product on
    $\analytic(V'_h)$. Moreover, if
    $\algebra{A} \subseteq \smooth(V'_h)$ is any unital
    $^*$-subalgebra on which all such products
    $\widehat{\stpro[\Lambda]}$ for all continuous Hermitian bilinear
    forms $\Lambda$ on $V$ converge, then
    $\algebra{A} \subseteq \analytic(V'_h)$, because analogous to
    Proposition~\ref{prop:ptwskalinnerprod}, every $f \in \algebra{A}$
    fulfills
    $\ptwskal{f}{f}^\bullet_\alpha = f^*
    \widehat{\stpro[\Lambda_\alpha]} f \in \algebra{A} \subseteq
    \smooth(V'_h)$
    for all $\skal{\argument}{\argument}_\alpha \in \mathcal{I}_{Vh}$
    with corresponding continuous Hermitian bilinear form
    $V^2 \ni (v, w) \mapsto \Lambda_\alpha(v,w) :=
    \skal{\cc{v}}{w}_\alpha \in \CC$.
    This is of course just our Theorem~\ref{theo:char} again.
\end{remark}

%
%

\subsection{Existence of continuous positive linear functionals}

Recall that a linear functional $\omega\colon\SymTen^\bullet(V)
\rightarrow \CC$ is said to be positive for $\stpro[\Lambda]$ if
$\omega(X^*\stpro[\Lambda]X) \ge 0$ holds for all
$X\in \SymTen^\bullet(V)$. Such positive linear functionals
yield important information about the representation theory
of a $^*$-algebra, e.g. there exists a faithful $^*$-representation
as adjointable operators on a pre-Hilbert space if and only
if the positive linear functionals are point-separating, see
\cite[Chap.~8.6]{schmuedgen:1990a}.
In this section we will determine the obstructions for the
existence of continuous positive linear functionals. First,
we need the following lemma which allows us to apply an
argument similar to the one used in \cite{bursztyn.waldmann:2000a}
in the formal case:

\begin{lemma}
   \label{lemma:squarecontractions}
   Let $\cc{\argument}$ be a continuous antilinear involution of $V$
   and $\Lambda$ a continuous Hermitian bilinear form on $V$ such that
   $\Lambda(\cc{v},v)\ge 0$ holds for all $v\in V$. Then for all
   $X\in\SymTen^\bullet(V)$ and all $t\in\NN_0$ there exist
   $n\in \NN$ and $X_1,\dots,X_n\in\SymTen^\bullet(V)$ such
   that
   \begin{equation}
     \big(\P_\Lambda\big)^t(X^*\otimes_\pi X)
     =
     \sum_{i=1}^n X_i^* \otimes_\pi X_i.
   \end{equation}
\end{lemma}
\begin{proof}
  This is trivial for scalar $X$ as well as for $t=0$ and for the
  remaining cases it is sufficient to
  consider $t=1$, the others then follow by induction.
  So let $k\in \NN$ and $X\in\SymTen^k(V)$ be given. Expand $X$ as
  $X=\sum_{j=1}^m x_{j,1}\vee \dots \vee x_{j,k}$ with $m\in \NN$ and
  vectors $x_{1,1},\dots,x_{m,k} \in V$. Then
  \begin{equation*}
    \P_\Lambda(X^*\otimes_\pi X)
    =
    \sum_{j',j=1}^m
    \sum_{\ell',\ell=1}^k
    \Lambda(\cc{x_{j',\ell'}},x_{j,\ell})
    (x_{j',1}\vee \cdots \widehat{x_{j',\ell'}} \cdots \vee x_{j',k})^*
    \otimes_\pi
    (x_{j,1}\vee \cdots \widehat{x_{j,\ell}} \cdots \vee x_{j,k}),
  \end{equation*}
  where $\widehat{\,\cdot\,}$ denotes omission of a vector in the
  product. The complex $mk\times mk$ -matrix with entries
  $\Lambda(\cc{x_{j',\ell'}},x_{j,\ell})$ is positive due to the
  positivity condition on $\Lambda$, which implies that it has
  a Hermitian square root $R\in \CC^{mk\times mk}$ that fulfills
  $\Lambda(\cc{x_{j',\ell'}},x_{j,\ell}) =
  \sum_{p=1}^m\sum_{q=1}^k \cc{R_{(p,q),(j',\ell')}} R_{(p,q),(j,\ell)}$
  for all $j,j'\in\{1,\dots,m\}$ and $\ell,\ell'\in\{1,\dots,k\}$.
  Consequently,
  \begin{align*}
    &\P_\Lambda(X^*\otimes_\pi X)=\\
    &=
    \sum_{p,q=1}^{m,k}
    \bigg(
      \sum_{j',\ell'=1}^{m,k}
      \cc{R_{(p,q),(j',\ell')}}
      (x_{j',1}\vee \cdots \widehat{x_{j',\ell'}} \cdots \vee x_{j',k})^*
    \bigg)
    \otimes_\pi
    \bigg(
      \sum_{j,\ell=1}^{m,k}
      R_{(p,q),(j,\ell)}
      (x_{j,1}\vee \cdots \widehat{x_{j,\ell}} \cdots \vee x_{j,k})
    \bigg)
  \end{align*}
  holds which proves the lemma.
\end{proof}

%

\begin{proposition}
    \label{prop:posex}%
    Let $\cc{\argument}$ be a continuous antilinear involution of $V$
    and $\Lambda$, $\Lambda'$ as well as $b$ three continuous Hermitian
    bilinear forms on $V$ such that $b$ is symmetric and of
    Hilbert-Schmidt type and such that
    $\Lambda'(\cc{v},v) + b(\cc{v},v)\ge 0$ holds for all
    $v\in V$. Given a continuous linear functional $\omega$ on
    $\SymTen^\bullet(V)$ that is positive for $\stpro[\Lambda]$,
    define $\omega_{zb}\colon \SymTen^\bullet(V) \rightarrow \CC$ as
    \begin{equation}
     X\;\mapsto\;\omega_{zb}(X) := \omega\big(\E^{z\Delta_b} X\big)
    \end{equation}
    for all $z\in \RR$. Then $\omega_{zb}$ is a continuous linear
    functional and positive for $\stpro[\Lambda+z\Lambda']$.
\end{proposition}
\begin{proof}
  It follows from
  Theorem~\ref{theo:equivalence} that $\omega_{zb}$ is
  continuous, and given $X\in\SymTen^\bullet(V)$, then
  \begin{align*}
    \omega
    \big(
      \E^{z\Delta_b}
      (
        X^*\stpro[\Lambda+z\Lambda']X
      )
    \big)
    &=
    \omega
    \big(
      (\E^{z\Delta_b}X)^*
      \stpro[\Lambda+z(\Lambda'+b)]
      (\E^{z\Delta_b}X)
    \big)\\
    &=
    \sum_{s,t=0}^\infty
    \frac{1}{s!t!}
    \omega
    \Big(
      \mu_{\vee}
      \Big(
        \big(\P_{\Lambda}\big)^s
        \big(\P_{z(\Lambda'+b)}\big)^t
        \big(
          (\E^{z\Delta_b}X)^*
          \otimes_\pi
          (\E^{z\Delta_b}X)
        \big)
      \Big)
    \Big)  \\
    &=
    \sum_{t=0}^\infty
    \frac{1}{t!}
    \omega
    \Big(
      \mu_{\stpro[\Lambda]}
      \Big(
        \big(\P_{z(\Lambda'+b)}\big)^t
        \big(
          (\E^{z\Delta_b}X)^*
          \otimes_\pi
          (\E^{z\Delta_b}X)
        \big)
      \Big)
    \Big)  \\
    &\ge
    0
  \end{align*}
  holds because $\P_\Lambda$ and $\P_{z(\Lambda'+b')}$ commute on
  symmetric tensors and because of
  Lemma~\ref{lemma:squarecontractions}.
\end{proof}

Note that Theorem~\ref{theo:equivalence} also shows that
$\omega_{zb}$ depends holomorphically on $z\in \CC$ in so far as
$\CC\ni z\mapsto \omega_{zb}(X)\in \CC$ is holomorphic for all
$X\in \SymTen^\bullet(V)$.
This is the analog of statements in
\cite{kaschek.neumaier.waldmann:2009a, lechner.waldmann:2016a} in
the Rieffel setting.

\begin{proposition}
    \label{prop:posnes}%
    Let $\cc{\argument}$ be a continuous antilinear involution of
    $V$ and $\Lambda$ a continuous Hermitian bilinear forms on $V$.
    If there exists a continuous linear functional $\omega$ on
    $\SymTen^\bullet(V)$ that is positive for $\stpro[\Lambda]$ and
    fulfills $\omega(1) = 1$, then the bilinear form
    $V^2 \ni (v,w) \mapsto b_\omega(v,w) := \omega(v\vee w) \in \CC$
    is symmetric, Hermitian, of Hilbert-Schmidt type and fulfills
    $\Lambda(\cc{v},v) + b_\omega(\cc{v},v)\ge 0$ for all $v\in V$.
\end{proposition}
\begin{proof}
  It follows immediately from the construction of $b_\omega$ that
  this bilinear form is symmetric and it is Hermitian because
  $\cc{b_\omega(v,w)} = \cc{\omega(v\vee w)} =
  \omega(\cc{w}\vee \cc{v}) = b_\omega(\cc{w},\cc{v})$
  holds for all $v,w\in V$. Continuity of $\omega$
  especially implies that there exists a
  $\skal{\argument}{\argument}_\alpha\in\mathcal{I}_V$
  such that $|\omega(X)|\le2^{-1/2}\seminorm{\alpha}{X}^\bullet$
  holds for all $X\in\SymTen^2(V)$, hence $b_\omega$ is of
  Hilbert-Schmidt type by Proposition~\ref{prop:hschar} and because
  $\Delta_{b_\omega} X = \omega(X)$ for $X\in\SymTen^2(V)$.
  Finally, $0\le\omega(v^*\stpro[\Lambda] v) = \Lambda(\cc{v},v)
  +b_\omega(\cc{v}, v)$ holds due to the positivity of $\omega$.
\end{proof}

\begin{theorem}
    \label{theorem:posex}%
    Let $\cc{\argument}$ be a continuous antilinear involution of
    $V$ and $\Lambda$ a continuous Hermitian bilinear forms on $V$.
    Assume $V\neq\{0\}$.
    There exists a non-zero continuous positive linear functional
    on $\big(\SymTen^\bullet(V),\stpro[\Lambda],^*\big)$ if and only
    if there exists a symmetric and hermitian bilinear form of
    Hilbert-Schmidt type $b$ on $V$ such that
    $\Lambda(\cc{v},v) + b(\cc{v},v)\ge 0$ holds for all $v\in V$.
    In this case, the continuous positive linear functionals
    on $\big(\SymTen^\bullet(V),\stpro[\Lambda],^*\big)$ are
    point-separating, i.e. their common kernel is $\{0\}$.
\end{theorem}
\begin{proof}
  If there exists a non-zero continuous positive linear functional
  $\omega$ on $\big(\SymTen^\bullet(V),\stpro[\Lambda],^*\big)$, then
  $\omega(1) \neq 0$ due to the Cauchy-Schwarz identity and we can
  rescale $\omega$ such that $\omega(1)=1$. Then the previous
  Proposition~\ref{prop:posnes} shows the existence of such a
  bilinear form $b$. Conversely, if such a bilinear form $b$
  exists, then Proposition~\ref{prop:posex} shows that all
  continuous linear functionals on $\SymTen^\bullet(V)$ that are
  positive for $\vee$ can be deformed to continuous linear functionals
  that are positive for $\stpro[\Lambda]$ by taking the pull-back
  with $\E^{\Delta_b}$. As $\E^{\Delta_b}$ is invertible, it only
  remains to show that that the continuous positive linear functionals
  on $\big(\SymTen^\bullet(V),\vee,^*\big)$ are point-separating.
  This is an immediate consequence of Theorem~\ref{theo:gelfand},
  which especially shows that the evaluation functionals
  $\delta_\rho$ with $\rho\in V'_h$ are point-separating.
\end{proof}

\subsection{Exponentials and Essential Self-Adjointness in GNS Representations}

Having a topology on the symmetric tensor algebra allows us to ask
whether or not some exponentials (with respect to the undeformed or
deformed products) exist in the completion, i.e. we want to discuss
for which tensors $X \in \SymTen^\bullet(V)^\cpl$ the series
$\exp_{\star_\Lambda}(X) := \sum_{n=0}^\infty \frac{1}{n!}
X^{\stpro[\Lambda] n}$
converges, where $X^{\stpro[\Lambda] n}$ denotes the $n$-th power of
$X$ with respect to the product $\stpro[\Lambda]$ for a continuous
bilinear form $\Lambda$ on $V$.  Note that since the algebra is
(necessarily) \emph{not} locally multiplicatively convex, this is a
non-trivial question.  This also allows to give a sufficient criterium
for a GNS representation of a Hermitian algebra element to be
essentially self-adjoint.
\begin{definition}
    For $k\in \NN_0$ we define
    \begin{equation}
        \label{eq:SymTensorsUpTok}
        \SymTen^{(k)}(V)
        :=
        \bigoplus_{\ell=0}^k \SymTen^\ell(V),
    \end{equation}
    and write $\SymTen^{(k)}(V)^\cpl$ for the closure of
    $\SymTen^{(k)}(V)$ in $\SymTen^{\bullet}(V)^\cpl$.
\end{definition}
\begin{lemma}
    \label{lemma:binomis}%
    One has
    \begin{equation}
        \label{eq:IneqForBinomis}
        \binom{m}{\ell}
        \binom{m-\ell+t}{t}
        \le
        \binom{\ell+t}{t} \binom{k(n+1)}{k}
    \end{equation}
    for all $k, n \in \NN_0$, $m \in \{0, \ldots, kn\}$,
    $t\in \{0, \ldots, k\}$, and all
    $\ell \in \big\{0, \ldots, \min\{m,k-t\}\big\}$.
\end{lemma}

\begin{lemma}
    \label{lemma:powersVerbessert}%
    Let $\Lambda$ be a continuous bilinear form on
    $V$.  Let $k, n \in \NN_0$ and
    $X_1, \ldots, X_n \in \SymTen^{(k)}(V)^\cpl$ be given.  Then the
    estimates
    \begin{gather}
        \label{eq:EstimateMultipleProductsComponents}
        \seminorm[\big]{\alpha}{
          \big\langle X_1
          \stpro[\Lambda] \cdots \stpro[\Lambda]
          X_n\big\rangle_m
        }^\bullet
        \le
        \bigg(\frac{(kn)!}{(k!)^n}\bigg)^{\frac{1}{2}}
        \big(2\E^2\big)^{kn}
        \seminorm{\alpha}{X_1}^{\bullet}
        \cdots
        \seminorm{\alpha}{X_n}^{\bullet} \\
        \shortintertext{and}
        \label{eq:EstimateMultipleProducts}
        \seminorm[\big]{\alpha}{
          X_1 \stpro[\Lambda] \cdots \stpro[\Lambda] X_n
        }^\bullet
        \le
        \bigg(\frac{(kn)!}{(k!)^n}\bigg)^{\frac{1}{2}}
        \big(2\E^3\big)^{kn}
        \seminorm{\alpha}{X_1}^{\bullet}
        \cdots
        \seminorm{\alpha}{X_n}^{\bullet}
    \end{gather}
    hold for all $m \in \{0, \ldots, kn\}$ and all
    $\seminorm{\alpha}{\argument} \in \mathcal{P}_{V,\Lambda}$.
\end{lemma}
\begin{proof}
    The first estimate implies the second, because
    $\seminorm{\alpha}{X_1 \stpro[\Lambda] \cdots \stpro[\Lambda]
      X_n}^\bullet$ has at most $(1+kn)$ non-vanishing homogeneous
    components, namely those of degree $m \in \{0, \ldots, kn\}$, and
    $(1+kn) \le \E^{kn}$. We will prove the first estimate by
    induction over $n$: If $n=0$ or $n=1$, then the estimate is
    clearly fulfilled for all possible $k$ and $m$, and if it holds
    for one $n\in \NN$, then
    \begin{align*}
        &\seminorm[\big]{\alpha}{
          \big\langle
          X_1\stpro[\Lambda] \cdots \stpro[\Lambda] X_{n+1}
          \big\rangle_m
        }^\bullet \\
        &\quad\le
        \sum_{t=0}^k
        \frac{1}{t!}
        \seminorm[\Big]{\alpha}{
          \Big\langle
          \mu_\vee\Big(
          \big(\P_{\Lambda}\big)^t
          \big(
          (
          X_1\stpro[\Lambda] \cdots \stpro[\Lambda] X_{n}
          )
          \otimes_\pi
          X_{n+1}
          \big)
          \Big)
          \Big\rangle_m
        }^\bullet\\
        &\quad\le
        \sum_{t=0}^k
        \sum_{\ell=0}^{\min\{m,k-t\}}
        \frac{1}{t!}
        \seminorm[\Big]{\alpha}{
          \mu_\vee\Big(
          \big(\P_{\Lambda}\big)^t
          \big(
          \langle
          X_1\stpro[\Lambda] \cdots \stpro[\Lambda] X_{n}
          \rangle_{m-\ell+t}
          \otimes_\pi
          \langle X_{n+1}\rangle_{\ell + t}
          \big)
          \Big)
        }^\bullet\\
        &\quad\le
        \sum_{t=0}^k
        \sum_{\ell=0}^{\min\{m,k-t\}}
        \frac{1}{t!}
        \binom{m}{\ell}^{\frac{1}{2}}
        \seminorm[\Big]{\alpha \otimes_\pi \alpha}{
          \big(\P_{\Lambda}\big)^t
          \big(
          \langle
          X_1\stpro[\Lambda] \cdots \stpro[\Lambda] X_{n}
          \rangle_{m-\ell+t}
          \otimes_\pi
          \langle X_{n+1}\rangle_{\ell + t}
          \big)
        }^\bullet\\
        &\quad\le
        \sum_{t=0}^k
        \sum_{\ell=0}^{\min\{m,k-t\}}
        \binom{m}{\ell}^{\frac{1}{2}}
        \binom{m-\ell+t}{t}^{\frac{1}{2}}
        \binom{\ell+t}{t}^{\frac{1}{2}}
        \seminorm{\alpha}{
          \langle
          X_1\stpro[\Lambda] \cdots \stpro[\Lambda] X_{n}
          \rangle_{m-\ell+t}
        }^\bullet
        \seminorm{\alpha}{
          \langle X_{n+1}\rangle_{\ell + t}
        }^\bullet\\
        &\quad\le
        \sum_{t=0}^k
        \sum_{\ell=0}^{\min\{m,k-t\}}
        \binom{\ell+t}{t}
        \binom{k(n+1)}{k}^{\frac{1}{2}}
        \seminorm{\alpha}{
          \langle
          X_1\stpro[\Lambda] \cdots \stpro[\Lambda] X_{n}
          \rangle_{m-\ell+t}
        }^\bullet
        \seminorm{\alpha}{
          \langle X_{n+1}\rangle_{\ell + t}
        }^\bullet\\
        &\quad\le
        \sum_{t=0}^k
        \sum_{\ell=0}^{\min\{m,k-t\}}
        \binom{\ell+t}{t}
        \binom{k(n+1)}{k}^{\frac{1}{2}}
        \bigg(\frac{(kn)!}{(k!)^n}\bigg)^{\frac{1}{2}}
        \big(2\E^2\big)^{kn}
        \seminorm{\alpha}{X_1}^{\bullet}
        \cdots
        \seminorm{\alpha}{X_{n}}^{\bullet}
        \seminorm{\alpha}{X_{n+1}}^{\bullet} \\
        &\quad=
        \sum_{t=0}^k
        \sum_{\ell=0}^{\min\{m,k-t\}}
        \binom{\ell+t}{t}
        \bigg(\frac{(k(n+1))!}{(k!)^{n+1}}\bigg)^{\frac{1}{2}}
        \big(2\E^2\big)^{kn}
        \seminorm{\alpha}{X_1}^{\bullet}
        \cdots
        \seminorm{\alpha}{X_{n+1}}^{\bullet} \\
        &\quad\le
        \bigg(\frac{(k(n+1))!}{(k!)^{n+1}}\bigg)^{\frac{1}{2}}
        \big(2\E^2\big)^{k(n+1)}
        \seminorm{\alpha}{X_1}^{\bullet}
        \cdots
        \seminorm{\alpha}{X_{n+1}}^{\bullet}
    \end{align*}
    holds due to the grading of $\mu_\vee$ and $\P_{\Lambda}$, the
    estimates from Propositions~\ref{prop:contotimes} as well as
    \ref{prop:symcont} and Lemma~\ref{lemma:plambdaklij} for
    $\mu_\vee$ and $\P_{\Lambda}$, and the previous
    Lemma~\ref{lemma:binomis}.
\end{proof}

\begin{proposition}
    \label{proposition:exp}%
    Let $\Lambda$ be a continuous bilinear form on
    $V$, then $\exp_{\stpro[\Lambda]}(v)$ is absolutely convergent and
    \begin{equation}
        \label{eq:exp}
        \exp_{\stpro[\Lambda]}(v)
        =
        \sum_{n=0}^\infty
        \frac{v^{\stpro[\Lambda] n}}{n!}
        =
        \E^{\frac{1}{2}\Lambda(v,v)} \exp_{\vee}(v)
    \end{equation}
    holds for all $v \in V$. Moreover,
    \begin{gather}
        \label{eq:expmult}
        \exp_{\vee}(v) \stpro[\Lambda] \exp_{\vee}(w)
        =
        \E^{\Lambda(v,w)} \exp_{\vee}(v+w) \\
        \shortintertext{and}
        \skal[\big]{\exp_\vee(v)}{\exp_\vee(w)}_\alpha^\bullet
        =
        \E^{\skal{v}{w}_\alpha}
    \end{gather}
    hold for all $v, w \in V$ and all
    $\skal{\argument}{\argument}_\alpha \in \mathcal{I}_V$. Finally,
    $\exp_\vee(v)^* = \exp_\vee(\cc{v})$ for all $v\in V$ if $V$ is
    equipped with a continuous antilinear involution $\cc{\argument}$.
\end{proposition}
\begin{proof}
    The existence and absolute convergence of
    $\stpro[\Lambda]$-exponentials of vectors follows directly from
    the previous Lemma~\ref{lemma:powersVerbessert}
    with $k=1$ and $X_1 = \cdots = X_n = v$:
    \begin{equation*}
        \sum_{n=0}^\infty
        \frac{\seminorm{\alpha}{v^{\stpro[\Lambda] n}}^\bullet}{n!}
        \le
        \sum_{n=0}^\infty
        \frac{\left(4 \E^3\seminorm{\alpha}{v}\right)^n}
        {\sqrt{n!}}
        \frac{1}{2^n}
        \stackrel{\cs}{\le}
        \bigg(
        \sum_{n=0}^\infty
        \frac{\left(4 \E^3\seminorm{\alpha}{v}\right)^{2n}}{n!}
        \bigg)^{\frac{1}{2}}
        \bigg(
        \sum_{n=0}^\infty
        \frac{1}{4^n}
        \bigg)^{\frac{1}{2}}
        =
        \frac{2\E^{16\E^6\seminorm{\alpha}{v}^2}}{\sqrt{3}}
    \end{equation*}
    The explicit formula can then be derived like in
    \cite[Lem.~5.5]{waldmann:2014a}. For \eqref{eq:expmult} we just
    note that
    \begin{align*}
        \P_{\Lambda}
        \big(
        \exp_{\vee}(v) \otimes_\pi \exp_{\vee}(w)
        \big)
        &=
        \sum_{k, \ell=0}^\infty
        \P_{\Lambda}
        \bigg(
        \frac{v^{\vee k}}{k!} \otimes_\pi \frac{w^{\vee \ell}}{\ell!}
        \bigg) \\
        &=
        \Lambda(v, w)
        \sum_{k, \ell=1}^\infty
        \frac{k v^{\vee (k-1)}}{k!}
        \otimes_\pi
        \frac{\ell w^{\vee (\ell-1)}}{\ell!} \\
        &=
        \Lambda(v, w)
        \exp_{\vee}(v) \otimes_\pi \exp_{\vee}(w),
    \end{align*}
    and so
    \begin{equation*}
        \exp_{\vee}(v) \stpro[\Lambda] \exp_{\vee}(w)
        =
        \sum_{t=0}^\infty \frac{1}{t!}
        \mu_\vee\Big(
        (\P_{\Lambda})^t
        \big(\exp_{\vee}(v) \otimes_\pi \exp_{\vee}(w)\big)
        \Big)
        =
        \E^{\Lambda(v, w)} \exp_{\vee}(v) \vee \exp_{\vee}(w).
    \end{equation*}
    The remaining two identities are the results of straightforward
    calculations.
\end{proof}

As an application we show that there exists a dense $^*$-subalgebra
consisting of uniformly bounded elements:
\begin{definition}
    Let $\cc{\argument}$ be a continuous antilinear involution on
    $V$. We define the linear subspace
    \begin{equation}
        \label{eq:SymTenPerSpace}
        \SymTen^\bullet_\per(V)
        :=
        \spann\big\{
        \exp_{\vee}(\I v) \in \SymTen^\bullet(V)^\cpl
        \; \big| \;
        v \in V
        \textrm{ and }
        \cc{v} = v
        \big\}
    \end{equation}
    of $\SymTen^\bullet(V)^\cpl$.
\end{definition}
\begin{proposition}
    \label{prop:periodic}%
    Let $\cc{\argument}$ be a continuous antilinear involution on
    $V$. Then $\SymTen^\bullet_\per(V)$ is a dense $^*$-subalgebra of
    $\big(\SymTen^\bullet(V)^\cpl, \stpro[\Lambda], { }^*\big)$
    with respect to all products $\stpro[\Lambda]$ for all continuous
    bilinear Hermitian forms $\Lambda$ on $V$ and
    \begin{equation}
      \seminorm{\infty,\Lambda}{X}
      :=
      \sup \sqrt{\omega(X^*\stpro[\Lambda]X)} < \infty
    \end{equation}
    holds for all $X\in \SymTen^\bullet_\per(V)$, where the supremum
    runs over all continuous positive linear functionals $\omega$
    on $\big(\SymTen^\bullet(V),\stpro[\Lambda],^*\big)$ that are
    normalized to $\omega(1)=1$.
\end{proposition}
\begin{proof}
    Proposition~\ref{proposition:exp} shows that
    $\SymTen^\bullet_\per(V)$ is a $^*$-subalgebra of
    $\SymTen^\bullet(V)^\cpl$ with respect to all products
    $\stpro[\Lambda]$ for all continuous bilinear Hermitian forms
    $\Lambda$ on $V$. As
    $-\I\frac{\D}{\D z}\big|_{z=0} \exp_\vee (\I zv) = v$ for all
    $v \in V$ with $v = \cc{v}$ we see that the closure of the
    subalgebra $\SymTen^\bullet_\per(V)$ contains $V$, hence
    $\SymTen^\bullet(V)$ which is (as a unital algebra) generated by
    $V$, and so the closure of $\SymTen^\bullet_\per(V)$ coincides
    with $\SymTen^\bullet(V)^\cpl$.

    As $\SymTen^\bullet_\per(V)$ is spanned by exponentials and
    $\omega\big(
     \exp_{\vee}(\I v)^*\stpro[\Lambda]\exp_{\vee}(\I v)
    \big)
    =
    \E^{\Lambda(v,v)} \omega\big(\exp_{\vee}(0)\big)
    =
    \E^{\Lambda(v,v)}$
    holds for all positive linear functionals $\omega$
    on $\big(\SymTen^\bullet(V),\stpro[\Lambda],^*\big)$ that are
    normalized to $\omega(1)=1$ by Proposition~\ref{proposition:exp},
    it follows that $\seminorm{\infty,\Lambda}{X}<\infty$ for all
    $X\in \SymTen^\bullet_\per(V)$.
\end{proof}

Note that one can show that $\seminorm{\infty,\Lambda}{\argument}$
is a $C^*$-norm on $\big(\SymTen^\bullet(V),\stpro[\Lambda],^*\big)$
if the continuous positive linear functionals are point-separating.
In contrast to the existence of exponential of vectors, we get strict
constraints on the existence of exponentials of quadratic elements:
\begin{proposition}
    Let $\cc{\argument}$ be a continuous antilinear
    involution on $V$. Then there is no locally convex topology $\tau$
    on $\SymTen^\bullet_\alg(V)$ with the property that any
    (undeformed) exponential
    $\exp_\vee(X) = \sum_{n=0}^\infty \frac{X^{\vee n}}{n!}$ of any
    $X\in \SymTen^2(V)\backslash\{0\}$ exists in the completion of
    $\SymTen^\bullet_\alg(V)$ under $\tau$ and such that all the
    products $\stpro[\Lambda]$ for all continuous Hermitian bilinear
    forms $\Lambda$ on $V$ as well as the $^*$-involution and the
    projection $\langle\argument\rangle_0$ on the scalars are
    continuous.
\end{proposition}
\begin{proof}
    Analogously to the proof of Theorem~\ref{theo:char} we see that,
    if all the products $\stpro[\Lambda]$ for all continuous Hermitian
    bilinear forms $\Lambda$ on $V$ as well as the $^*$-involution and
    the projection $\langle\argument\rangle_0$ on the scalars are
    continuous, then all the extended positive Hermitian forms
    $\skal{\argument}{\argument}_\alpha^\bullet$ for all
    $\skal{\argument}{\argument}_\alpha\in\mathcal{I}_V$ would have to
    be continuous and thus extend to the completion of
    $\SymTen^\bullet_\alg(V)$.

    Now let $X \in \SymTen^2(V)\setminus\{0\}$ be given. There exist
    $k \in \NN$ and $x \in V^k$ such that $x_1, \ldots, x_k$ are
    linearly independent and
    $X = \sum_{i=1}^k\sum_{j=i}^k \tilde{X}^{ij} x_i \vee x_j$ with
    complex coefficients $\tilde{X}^{ij}$. If there exists an
    $i \in \{1, \ldots, k\}$ such that $\tilde{X}^{ii} \neq 0$, then
    we can assume without loss of generality that $i=1$ and
    $\tilde{X}^{11} = 1$ and define a continuous positive Hermitian
    form on $V$ by $\skal{v}{w}_\omega := \cc{\omega(v)} \omega(w)$,
    where $\omega\colon V \rightarrow \CC$ is a continuous linear form
    on $V$ that satisfies $\omega(x_1) = 1$ and $\omega(x_i) = 0$ for
    $i \in \{2, \ldots, k\}$. Otherwise we can assume without loss of
    generality that $\tilde{X}^{11} = \tilde{X}^{22} = 0$ and
    $\tilde{X}^{12} = 1$ and define a continuous positive Hermitian
    form on $V$ by $\skal{v}{w}_\omega := \cc{\omega(v)}^T \omega(w)$,
    where $\omega\colon V \rightarrow \CC^2$ is a continuous linear
    map that satisfies $\omega(x_1) = \binom{1}{0}$,
    $\omega(x_2) = \binom{0}{1}$ and $\omega(x_i) = 0$ for
    $i \in \{3, \ldots, k\}$.

    In the first case, this results in
    $\skal{X^{\vee n}}{X^{\vee n}}^\bullet_\omega = (2n)!$ and in the
    second, $\skal{X^{\vee n}}{X^{\vee n}}^\bullet_\omega = (n!)^2$. So
    $\sum_{n=0}^\infty \frac{X^{\vee n}}{n!}$ cannot converge in the
    completion of $\SymTen^\bullet_\alg(V)$ because
    \begin{equation*}
        \skal[\bigg]{\sum_{n=0}^N \frac{X^{\vee n}}{n!}}
        {\sum_{n=0}^N \frac{X^{\vee n}}{n!}}^\bullet_\omega
        \ge
        \sum_{n=0}^N 1
        \xrightarrow{N\rightarrow \infty}
        \infty.
    \end{equation*}
\end{proof}

A similar result has already been obtained by Omori, Maeda, Miyazaki
and Yoshioka in the $2$-dimensional case in
\cite{omori.maeda.miyazaki.yoshioka:2000a}, where they show that
associativity of the Moyal-product breaks down on exponentials of
quadratic functions. Note that the above proposition does not exclude
the possibility that exponentials of \emph{some} quadratic functions
exist if one only demands that \emph{some} special deformations are
continuous.

Even though exponentials of non-trivial
tensors of degree $2$ are not
contained in $\SymTen^\bullet(V)^\cpl$, the continuous positive linear
functionals are in some sense ``analytic'' for such tensors:
\begin{proposition}
    \label{proposition:analyticfunctionals}%
    Let $\cc{\argument}$ be a continuous antilinear
    involution on $V$ and $\Lambda$ a continuous Hermitian bilinear
    form on $V$. Let
    $\omega\colon \SymTen^\bullet(V)^\cpl \rightarrow \CC$ be a
    continuous linear functional on $\SymTen^\bullet(V)^\cpl$ that is
    positive with respect to $\stpro[\Lambda]$. Then for all
    $X \in \SymTen^{(2)}(V)^\cpl$ there exists an $\epsilon > 0$ such
    that
    \begin{equation}
        \label{eq:ExpOmegaEstimate}
        \sum_{n=0}^\infty
        \frac{
          \epsilon^n
          \omega\big(
          (X^{\stpro[\Lambda] n})^*
          \stpro[\Lambda]
          X^{\stpro[\Lambda] n}
          \big)^{\frac{1}{2}}
        }
        {n!}
        <
        \infty
    \end{equation}
    holds.
\end{proposition}
\begin{proof}
    The seminorm
    $\SymTen^\bullet(V)^\cpl \ni Y \mapsto \omega(Y^*\stpro[\Lambda]
    Y)^{1/2} \in [0,\infty[$
    is continuous by construction, so there exist $C > 0$ and
    $\seminorm{\alpha}{\argument}\in \mathcal{P}_V$ such that
    $\omega(Y^* \stpro[\Lambda] Y)^{1/2} \le C
    \seminorm{\alpha}{Y}^\bullet$
    holds for all $Y\in \SymTen^\bullet(V)^\cpl$. We can even assume
    without loss of generality that
    $\seminorm{\alpha}{\argument} \in \mathcal{P}_{V,\Lambda}$. Now
    choose $\epsilon > 0$ with
    $\epsilon \big(8 \E^6 \seminorm{\alpha}{X}^{\bullet,
      2}\big) \le 1$,
    then Lemma~\ref{lemma:powersVerbessert} in the case $k=2$ and
    $X_1 = \cdots = X_n = X$ shows that
    \begin{align*}
        \sum_{n=0}^\infty
        \frac{
          \epsilon^n \omega
          \big(
          (X^{\stpro[\Lambda]n})^*
          \stpro[\Lambda]
          X^{\stpro[\Lambda] n}
          \big)^{\frac{1}{2}}
        }
        {n!}
        &\le
        C
        \sum_{n=0}^\infty
        \frac{
          \epsilon^n
          \seminorm{\alpha}{X^{\stpro[\Lambda] n}}^\bullet}
        {n!} \\
        &\le
        C
        \sum_{n=0}^\infty
        \frac{\sqrt{(2n)!}}
        {\sqrt{2}^{3n}n!} \\
        &\le
        C
        \sum_{n=0}^\infty \frac{1}{\sqrt{2}^n}
        \\
        &=
        \frac{C\sqrt{2}}{\sqrt{2} - 1}.
    \end{align*}
\end{proof}

It is an immediate consequence of this proposition that Hermitian
tensors of grade at most $2$ are represented by essentially
self-adjoint operators in every GNS-representation corresponding to a
continuous positive linear functional $\omega$. Recall that for a
$^*$-algebra $\algebra{A}$ with a positive linear functional
$\omega\colon \algebra{A} \rightarrow \CC$, the GNS representation of
$\algebra{A}$ associated to $\omega$ is the unital $^*$-homomorphism
$\pi_\omega\colon \algebra{A} \rightarrow \textup{Adj}(\algebra{A} /
\mathcal{I}_\omega)$ into the adjointable endomorphisms on the
pre-Hilbert space $\mathcal{H}_\omega = \algebra{A} /
\mathcal{I}_\omega$ with inner product
$\skal{\argument}{\argument}_\omega$, where $\mathcal{I}_\omega =
\big\{a \in \algebra{A} \; \big| \; \omega(a^*a) = 0\big\}$ and
$\skal{[a]}{[b]}_\omega = \omega(a^*b)$ for all $[a], [b] \in
\mathcal{H}_\omega$ with representatives $a, b \in\algebra{A}$.
\begin{theorem}
    \label{theo:essselfadj}%
    Let $\cc{\argument}$ be a continuous antilinear involution on $V$
    and $\Lambda$ a continuous Hermitian bilinear form on $V$. Let
    $\omega\colon \SymTen^\bullet(V)^\cpl \rightarrow \CC$ be a
    continuous linear functional on $\SymTen^\bullet(V)^\cpl$ that is
    positive with respect to $\stpro[\Lambda]$.  Then for $X^* = X \in
    \SymTen^{(2)}(V)^\cpl$ all vectors in the GNS pre-Hilbert space
    $\mathcal{H}_\omega$ are analytic for $\pi_\omega(X)$ which is
    therefore essentially self-adjoint.
\end{theorem}
\begin{proof}
    It is clear from the construction of the GNS representation that
    $\pi_\omega(X)$ is a symmetric operator on $\mathcal{H}_\omega =
    \SymTen^\bullet(V)^\cpl / \mathcal{I}_\omega$ and by Nelson's
    theorem, see e.g. \cite[Thm.~7.16]{schmuedgen:2012a}, it is
    sufficient to show that all vectors $[Y] \in \mathcal{H}_\omega$
    are analytic for $\pi_\omega(X)$: From
    \begin{equation*}
        \skal[\big]{\pi_\omega(X)^n[Y]}{\pi_\omega(X)^n[Y]}_\omega
        =
        \omega\big(
        (X^{\stpro[\Lambda]n} \stpro[\Lambda] Y)^*
        \stpro[\Lambda]
        (X^{\stpro[\Lambda]n} \stpro[\Lambda] Y)
        \big)
        =
        \omega\big(
        Y^* \stpro[\Lambda]
        (X^{\stpro[\Lambda]n})^*
        \stpro[\Lambda]
        X^{\stpro[\Lambda]n}
        \stpro[\Lambda] Y
        \big)
    \end{equation*}
    it follows that analyticity of the vector $[Y]$ is equivalent to
    the analyticity of the continuous positive linear functional
    $\SymTen^\bullet(V)^\cpl \ni Z \mapsto \omega_{Y}(Z) := \omega(Y^*
    \stpro[\Lambda] Z \stpro[\Lambda] Y) \in \CC$ in the sense of the
    previous Proposition~\ref{proposition:analyticfunctionals}.
\end{proof}

%
%

\section{Special Cases and Examples}
\label{sec:SpecialCasesExamples}

Finally we want to discuss two special cases that have appeared in the
literature before, namely that $V$ is a Hilbert space and that $V$ is
a nuclear space.

%
%

\subsection{Deformation Quantization of Hilbert Spaces}

Assume that $V$ is a (complex) Hilbert space with inner product
$\skal{\argument}{\argument}_1$. We note that in this case
$\SymTen^\bullet(V)$ is not a pre-Hilbert space but only a countable
projective limit of pre-Hilbert spaces, because the extensions
$\skal{\argument}{\argument}_\alpha^\bullet$ of the (equivalent) inner
products $\skal{\argument}{\argument}_\alpha := \alpha
\skal{\argument}{\argument}_1$ for $\alpha \in {]0, \infty[}$ are not
equivalent. If $V$ is a Hilbert space, then its topological dual and,
more generally, all spaces of bounded multilinear functionals on $V$
are Banach spaces. This allows a more detailed analysis of the
continuity of functions in $\analytic(V'_h)$ and of the dependence of
the product $\stpro[\Lambda]$ on $\Lambda \in \Bil(V)$.

\begin{theorem}
    \label{theorem:ContinuityOnLambda}%
    Let $V$ be a (complex) Hilbert space with inner product
    $\skal{\argument}{\argument}_1$ and unit ball $U \subseteq V$ and
    let $\Bil(V)$ be the Banach space of all continuous bilinear forms
    on $V$ with norm
    $\norm{\Lambda} := \sup_{v, w \in U} \abs{\Lambda(v, w)}$. Then
    the map
    $\Bil(V)\times \SymTen^\bullet(V)^\cpl \times
    \SymTen^\bullet(V)^\cpl \rightarrow \SymTen^\bullet(V)^\cpl$
    \begin{equation}
        \label{eq:TripleMapContinuous}
        (\Lambda, X, Y)
        \; \mapsto \;
        X \stpro[\Lambda] Y
    \end{equation}
    is continuous.
\end{theorem}
\begin{proof}
    Note that for a Hilbert space $V$, the continuous inner products
    $\skal{\argument}{\argument}_\lambda$ with $\lambda > 0$ are
    cofinal in $\mathcal{I}_V$. Now let $\Lambda \in \Bil(V)$,
    $X, Y \in \SymTen^\bullet(V)^\cpl$ and $\epsilon > 0$ be given,
    then
    \begin{equation*}
        \seminorm{\lambda}
        {X' \stpro[\Lambda'] Y' - X \stpro[\Lambda] Y}
        \le
        \seminorm{\lambda}
        {X' \stpro[\Lambda'] Y' - X \stpro[\Lambda'] Y}
        +
        \seminorm{\lambda}{X \stpro[\Lambda'] Y - X \stpro[\Lambda] Y}
    \end{equation*}
    holds for all $\lambda > 0$ and all $\Lambda' \in \Bil(V)$ as well
    as all $X', Y' \in \SymTen^\bullet(V)^\cpl$. Moreover,
    \begin{align*}
        \seminorm{\lambda}
        {X' \stpro[\Lambda'] Y' - X \stpro[\Lambda'] Y}
        &\le
        \seminorm{\lambda}{(X' - X) \stpro[\Lambda'] Y'}
        +
        \seminorm{\lambda}{X \stpro[\Lambda'] (Y' - Y)} \\
        &\le
        4 \seminorm{8\lambda}{X' - X}^\bullet
        \seminorm{8\lambda}{Y'}^\bullet
        +
        4 \seminorm{8\lambda}{X}^\bullet
        \seminorm{8\lambda}{Y' - Y}^\bullet
    \end{align*}
    holds for all $X', Y' \in \SymTen^\bullet(V)^\cpl$ as well as all
    $\lambda > 0$ and all $\Lambda' \in \Bil(V)$ such that
    $\seminorm{\lambda}{\argument} \in \mathcal{P}_{V,\Lambda'}$ by
    Lemma~\ref{lemma:stprocont}. One can check on factorizing
    symmetric tensors that $\P_\Lambda$ and $\P_{\Lambda'-\Lambda}$
    commute and by using that
    \begin{align*}
        X\stpro[\Lambda']Y
        &=
        \sum_{t'=0}^\infty \frac{1}{t'!}
        \mu_\vee \Big(
        \big(\P_{\Lambda + (\Lambda' - \Lambda)}\big)^{t'}
        (X \otimes_\pi Y)
        \Big) \\
        &=
        \sum_{t, s=0}^\infty \frac{1}{t!s!}
        \mu_\vee\Big(
        \big(\P_{\Lambda}\big)^t \big(\P_{\Lambda' - \Lambda}\big)^s
        (X\otimes_\pi Y)
        \Big) \\
        &=
        \sum_{s=0}^\infty \frac{1}{s!}
        \mu_{\stpro[\Lambda]}
        \big((\P_{\Lambda' - \Lambda})^s (X \otimes_\pi Y)\big),
    \end{align*}
    it follows that
    \begin{align*}
        \seminorm{\lambda}{X \stpro[\Lambda'] Y - X \stpro[\Lambda] Y}
        &\le
        \sum_{s=1}^\infty
        \frac{1}{\rho^s s!}
        \seminorm[\Big]{\lambda}{\multstpro[\Lambda]
          \Big(
          \big(\P_{\rho(\Lambda' - \Lambda)}\big)^s
          (X \otimes_\pi Y)
          \Big)
        }^\bullet \\
        &\le
        4
        \sum_{s=1}^\infty
        \frac{1}{\rho^s s!}
        \seminorm[\Big]{8\lambda \otimes_\pi 8\lambda}
        {
          \big(\P_{\rho(\Lambda' - \Lambda)}\big)^s
          (X \otimes_\pi Y)
        }^\bullet
        \\
        &\le
        8
        \sum_{s=1}^\infty
        \frac{1}{(2\rho)^s}
        \seminorm{32\lambda}{X}^\bullet
        \seminorm{32\lambda}{Y}^\bullet \\
        &=
        \frac{8}{2\rho - 1}
        \seminorm{32\lambda}{X}^\bullet
        \seminorm{32\lambda}{Y}^\bullet
    \end{align*}
    holds for all $\rho> \frac{1}{2}$, $\lambda>0$, and all
    $\Lambda' \in \Bil(V)$ if
    $\seminorm{\lambda}{\argument} \in \mathcal{P}_{V,\Lambda} \cap
    \mathcal{P}_{V, \rho(\Lambda' - \Lambda)}$
    by Lemma~\ref{lemma:stprocont} and
    Proposition~\ref{prop:contplambda} with $c=2$.

    Assume that $\lambda \ge 1 + \norm{\Lambda}$ and choose
    $\rho > \frac{1}{2}$ such that
    $\frac{8}{2\rho - 1} \seminorm{32\lambda}{X}^\bullet
    \seminorm{32\lambda}{Y}^\bullet \le \frac{\epsilon}{3}$.
    Then
    $\seminorm{\lambda}{\argument} \in \mathcal{P}_{V,\Lambda} \cap
    \mathcal{P}_{V, \rho(\Lambda' - \Lambda)}$
    for all $\Lambda'\in \Bil(V)$ with
    $\seminorm{}{\Lambda' - \Lambda} \le \frac{1}{\rho}$ and
    $\seminorm{\lambda}{X' \stpro[\Lambda'] Y' - X \stpro[\Lambda] Y}
    \le \epsilon$
    holds for all these $\Lambda'$ and all
    $X', Y' \in \SymTen^\bullet(V)^\cpl$ with
    $\seminorm{8\lambda}{X' - X}^\bullet \le \epsilon / \big(12 +
    12\seminorm{8\lambda}{Y}^\bullet\big)$
    and
    $\seminorm{8\lambda}{Y' - Y} \le \min\big\{1, \epsilon / \big(12 +
    12\seminorm{8\lambda}{X}^\bullet\big)\big\}$.
    This proves continuity of $\star$ at $(\Lambda, X, Y)$.
\end{proof}
\begin{theorem}
    Let $V$ be a (complex) Hilbert space with inner product
    $\skal{\argument}{\argument}_1$ and a continuous antilinear
    involution $\cc{\argument}$ that fulfills
    $\cc{\skal{v}{w}}_1 = \skal{\cc{v}}{\cc{w}}_1$ for all
    $v, w \in V$, then $\widehat{X}\colon V'_h \rightarrow \CC$ is
    smooth in the Fréchet sense for all
    $X \in \SymTen^\bullet(V)^\cpl$.
\end{theorem}
\begin{proof}
    By the Fréchet-Riesz theorem we can identify $V'_h$ with $V_h$ by
    means of the antilinear map
    $\argument^\flat\colon V_h \rightarrow V'_h$.  As the
    translations $\tau^*$ are automorphisms of
    $\SymTen^\bullet(V)^\cpl$, it is sufficient to show that
    $\widehat{X}$ is smooth at $0 \in V'_h$.  So let $K \in \NN_0$ and
    $r \in V_h$ be given with $r \neq 0$ and $\seminorm{1}{r} \le 1$.
    We have already seen in Proposition~\ref{prop:deriv} that all
    directional derivatives of $\widehat{X}$ exist and form bounded
    symmetric multilinear maps
    $(V'_h)^k \ni \rho \mapsto \big(\widehat{D}^{(k)}_{\rho}
    \widehat{X}\big)(0) \in \CC$.
    These maps are indeed the derivatives of $\widehat{X}$ in the
    Fréchet sense due to the analyticity of $\widehat{X}$: Define
    $\hat{r} := r / \seminorm{1}{r}$, then due to
    Proposition~\ref{prop:deriv} and Lemma~\ref{lemma:contcharacter}
    the estimate
    \begin{align*}
        \frac{1}{\norm{r}^{K+1}}
        \abs[\bigg]{
          \widehat{X}(r^\flat)
          -
          \sum_{k=0}^K
          \frac{1}{k!}
          \big(
          \widehat{D}^{(k)}_{(r^\flat, \ldots, r^\flat)} \widehat{X}
          \big)(0)
        }
        &=
        \frac{1}{\norm{r}^{K+1}}
        \abs[\bigg]{
          \bigg\langle
          \tau^*_{r^\flat}(X)
          -
          \sum_{k=0}^K \
          \frac{1}{k!}
          \big(D_{r^\flat}\big)^k X
          \bigg\rangle_0
        } \\
        &=
        \frac{1}{\norm{r}^{K+1}}
        \abs[\bigg]{
          \bigg\langle
          \sum_{k=K+1}^\infty
          \frac{1}{k!}
          \big(D_{r^\flat}\big)^k X
          \bigg\rangle_0
        } \\
        &\le
        \abs[\bigg]{
          \bigg\langle
          \sum_{k=K+1}^\infty
          \frac{1}{k!}
          \big(D_{\hat{r}^\flat}\big)^k X
          \bigg\rangle_0
        } \\
        &\le
        \sum_{k=K+1}^\infty
        \frac{1}{k!}
        \seminorm[\big]{1}{\big(D_{\hat{r}^\flat}\big)^k X}^\bullet \\
        &\le
        \sum_{k=K+1}^\infty
        \frac{1}{\sqrt{k!}}
        \seminorm{2}{X}^\bullet \\
        &\le
        C \seminorm{2}{X}^\bullet
    \end{align*}
    with $C = \sum_{k=K+1} \frac{1}{\sqrt{k!}} < \infty$ holds
    uniformly for all $r \neq 0$ with $\seminorm{1}{r} \le 1$.
\end{proof}

The formal deformation quantization of a Hilbert space in a very
similar setting has already been examined in \cite{dito:2005a} by
Dito. There the formal deformations of exponential type of a certain
algebra $\mathcal{F}_{HS}$ of smooth functions on a Hilbert space
$\hilb$ was constructed. More precisely, $\mathcal{F}_{HS}$ consists
of all smooth (in the Fréchet sense) functions $f$ whose derivatives
fulfill the additional condition that for all $\sigma \in \hilb$
\begin{align}
    k! \ptwskal[\big]{f}{f}^k(\sigma)
    :=
    \sum_{i \in I^k}
    \abs[\big]{
      \big(\widehat{D}^{(k)}_{(e_{i_1}, \ldots, e_{i_k})} f\big)
      (\sigma)
    }^2
    <
    \infty
\end{align}
holds and depends continuously on $\sigma$ for one (hence all) Hilbert
base $e \in \hilb^I$ of $\hilb$ indexed by a set $I$. In this case
$\ptwskal{f}{f}^k \in \mathcal{F}_{HS}$ holds.

The convergent deformations discussed in this article and the formal
deformations discussed by Dito in \cite{dito:2005a} are very much
analogous: In both cases it is necessary to restrict the construction
to a subalgebra of all smooth functions, $\mathcal{F}_{HS}$ or
$\analytic(V'_h)$, where the additional requirement is that all the
derivatives of fixed order (in the formal case) or of all orders (in
the convergent case) at every point $\sigma$ obey a Hilbert-Schmidt
condition and that the square of the corresponding Hilbert-Schmidt
norms, $\ptwskal{f}{f}^k(\sigma)$ or $\ptwskal{f}{f}^\bullet(\sigma)$,
respectively, depend in a sufficiently nice way on $\sigma$ such that
one can prove that $\ptwskal{f}{f}^k$ and $\ptwskal{f}{f}^\bullet$ are
again elements of $\mathcal{F}_{HS}$ or $\analytic(V'_h)$ (see the
proof of Proposition~3.4 in \cite{dito:2005a} and our
Proposition~\ref{prop:ptwskalinnerprod}). Moreover, the results
concerning equivalence of the deformations are similar: In
\cite[Thm.~2]{dito:2005a} it is shown that two (formal) deformations
are equivalent if and only if they differ by bilinear forms of
Hilbert-Schmidt type, while our Theorem~\ref{theo:equivalence} shows
that the corresponding equivalence transformations are continuous if
and only if they are generated by bilinear forms of Hilbert-Schmidt
type.

%
%

\subsection{Deformation Quantization of Nuclear Spaces}

We conclude this article with a short discussion of the case that $V$
is nuclear. It is well known that the topology of a nuclear space can
be described by continuous Hilbert seminorms. Moreover, the topology
of the Hilbert tensor product on $\SymTen^k(V)$ coincides with the
topology of the projective tensor product which was examined in
\cite{waldmann:2014a}. However, for the comparison of the topologies
on $\SymTen^\bullet(V)$ we have to be more careful: Let
$\seminorm{\alpha}{\argument}\in \mathcal{P}_V$ be given. Define the
seminorm $\seminorm{\alpha,\textup{pr}}{\argument}^\bullet$ as
\begin{equation}
    \label{eq:SeminormAlphaPr}
    \seminorm{\alpha,\textup{pr}}{X}^\bullet
    :=
    \abs[\big]{\langle X \rangle_0}
    +
    \sum_{k=1}^\infty \sqrt{k!}
    \inf \sum_{i \in I} \prod_{m=1}^k
    \seminorm{\alpha}{x_{i,m}}
\end{equation}
for all $X\in\Ten^\bullet_\alg(V)$, where the infimum runs over all
possibilities to express $\langle X\rangle_k$ as a finite sum of
factorizing tensors, i.e. as
$\langle X\rangle_k = \sum_{i\in I} x_{i,1} \otimes \cdots \otimes
x_{i,k}$ with $x_i \in V^k$.
\begin{lemma}
    One has the estimate
    \begin{equation}
        \label{eq:EstimatePr}
        \seminorm{\alpha}{X}^\bullet
        \le
        \seminorm{\alpha,\textup{pr}}{X}^\bullet
    \end{equation}
    for all $X\in \Ten^\bullet_\alg(V)$. Moreover, if there is a
    $\seminorm{\beta}{\argument} \in \mathcal{P}_V$,
    $\seminorm{\beta}{\argument} \ge \seminorm{\alpha}{\argument}$,
    such that for every
    $\skal{\argument}{\argument}_\beta$-orthonormal $e \in V^d$ and
    all $d \in \NN$ the estimate
    $\sum_{i=1}^d \seminorm{\alpha}{e_i}^2 \le 1$ holds, then
    \begin{equation}
        \label{eq:NochEineAbschaetzung}
        \seminorm{\alpha,\textup{pr}}{X}^\bullet
        \le
        \seminorm{\beta}{X}^\bullet
    \end{equation}
    for all $X \in \Ten^\bullet_\alg(V)$.
\end{lemma}
\begin{proof}
    Let $X\in \Ten^\bullet_\alg(V)$ be given, then
    $\seminorm{\alpha}{X}^\bullet \le \sum_{k=0}^\infty
    \seminorm{\alpha}{\langle X \rangle_k}^\bullet$ and
    $\seminorm{\alpha,\textup{pr}}{X}^\bullet = \sum_{k=0}^\infty
    \seminorm{\alpha,\textup{pr}}{\langle X \rangle_k}^\bullet$.
    Thus it is sufficient for the first estimate to show that
    $\seminorm{\alpha}{\langle X\rangle_k}^\bullet \le
    \seminorm{\alpha,\textup{pr}}{\langle X\rangle_k}^\bullet$
    for all $k\in \NN_0$. Fix $k\in \NN_0$ and assume that
    $\langle X \rangle_k = \sum_{i\in I} x_{i,1} \otimes \cdots
    \otimes x_{i,k}$ with $x_i \in V^k$. Then
    \begin{equation*}
        \seminorm{\alpha}{\langle X\rangle_k}^\bullet
        \le
        \sum\nolimits_{i\in I}
        \seminorm{\alpha}
        {x_{i,1} \otimes \cdots \otimes x_{i,k}}^\bullet
        =
        \sqrt{k!}
        \sum\nolimits_{i\in I}
        \prod\nolimits_{m=1}^k
        \seminorm{\alpha}{x_{i,m}}
    \end{equation*}
    shows that
    $\seminorm{\alpha}{\langle X\rangle_k}^\bullet \le
    \seminorm{\alpha,\textup{pr}}{\langle X\rangle_k}^\bullet$,
    hence
    $\seminorm{\alpha}{X}^\bullet \le
    \seminorm{\alpha,\textup{pr}}{X}^\bullet$.
    For the second estimate, let $\seminorm{\beta}{\argument}$ with
    the stated properties and $X\in \Ten^k_\alg(V)$ be given. Use
    Lemma~\ref{lemma:helpfull} to construct
    $X_0 = \sum_{a\in A} x_{a,1} \otimes \cdots \otimes x_{a,k}$ and
    $\tilde{X} = \sum_{a'\in \{1, \ldots, d\}^k} X^{a'} e_{a'_1}
    \otimes \cdots \otimes e_{a'_k}$
    with $e\in V^k$ orthonormal with respect to
    $\skal{\argument}{\argument}_\beta$. Clearly
    $\seminorm{\alpha,\pr}{X_0}^\bullet = 0$ and so
    \begin{align*}
        \seminorm{\alpha,\textup{pr}}{X}^\bullet
        &\le
        \seminorm{\alpha,\textup{pr}}{\tilde{X}}^\bullet \\
        &\le
        \sqrt{k!}
        \sum_{a'\in \{1, \ldots, d\}^k}
        \abs[\big]{X^{a'}}
        \prod_{m=1}^k \seminorm{\alpha}{e_{a'_m}} \\
        &\stackrel{\textup{cs}}{\le}
        \Bigg(
        k! \bigg(
        \sum_{a'\in \{1, \ldots, d\}^k}
        \abs[\big]{X^{a'}}^2
        \bigg)
        \bigg(
        \sum_{a' \in \{1, \ldots, d\}^k}
        \prod_{m=1}^k \seminorm{\alpha}{e_{a'_m}}^{2}
        \bigg)
        \Bigg)^{\frac{1}{2}} \\
        &\le
        \Bigg(
        k!
        \bigg(
        \sum_{a'\in \{1, \ldots, d\}^k}
        \abs[\big]{X^{a'}}^2
        \bigg)
        \bigg(
        \sum_{i=1}^d
        \seminorm{\alpha}{e_i}^{2}
        \bigg)^{k}
        \Bigg)^{\frac{1}{2}} \\
        &\le
        \seminorm{\beta}{X}^\bullet.
    \end{align*}
\end{proof}
\begin{proposition}
    Let $V$ be a nuclear space, then the topology on
    $\SymTen^\bullet(V)$ coincides with the one constructed in
    \cite{waldmann:2014a} for $R = \frac{1}{2}$.
\end{proposition}
\begin{proof}
    This is a direct consequence of the preceeding lemma because the
    locally convex topology constructed in \cite{waldmann:2014a} for
    $R = \frac{1}{2}$ is the one defined by the seminorms
    $\seminorm{\alpha,\textup{pr}}{\argument}^\bullet$ for all
    $\seminorm{\alpha}{\argument}\in\mathcal{P}_V$ and because in a
    nuclear space, such seminorms $\seminorm{\beta}{\argument}$ as
    required in the lemma exist for all
    $\seminorm{\alpha}{\argument}\in\mathcal{P}_V$, see e.g.
    \cite[Satz~28.4]{meise.vogt:1992a} or also
    \cite[Chap.~21.2, Thm.~1]{jarchow:1981a}.
\end{proof}
From \cite[Thm.~4.10]{waldmann:2014a} we get:
\begin{corollary}
    Let $V$ be a nuclear space, then $\SymTen^\bullet(V)$ is nuclear.
\end{corollary}
And conversely, our Theorem~\ref{theo:char} implies:
\begin{corollary}
    Let $V$ be a nuclear space, then the $R = \frac{1}{2}$ topology
    constructed in \cite{waldmann:2014a} is the coarsest one possible
    under the conditions of Theorem~\ref{theo:char} in the truely (not
    graded) symmetric case.
\end{corollary}

As all continuous bilinear forms on a nuclear space $V$ are
automatically of Hilbert-Schmidt type (see
\cite[Chap.~21.3, Thm.~5]{jarchow:1981a} or use
\cite[Satz~28.4]{meise.vogt:1992a}), we also see that the
equivalence transformations $\E^{\Delta_b}$ are continuous for all
continuous symmetric bilinear forms $b$ on $V$, which corresponds to
\cite[Prop.~5.9]{waldmann:2014a}. Our discussion of translations and
evaluation functionals then shows the existence of point-separating many
positive linear functionals on the deformed algebras:

\begin{theorem}
  \label{theo:posexnuclear}
  Let $V$ be a Hausdorff nuclear space and $\cc{\argument}$ a continuous
  antilinear involution of $V$ as well as $\Lambda$ a continuous Hermitian
  bilinear form on $V$, then there exist point-separating many
  continuous positive linear functionals of
  $\big(\SymTen^\bullet(V),\stpro[\Lambda],^*\big)$.
\end{theorem}
\begin{proof}
  Choose some $\skal{\argument}{\argument}_\alpha\in\mathcal{I}_{V,h}$ and
  define a bilinear form $b$ on $V$ by $b(v,w):=\skal{\cc{v}}{w}_\alpha$ for all
  $v,w\in V$. Then $b$ is continuous and Hermitian by construction and
  symmetric due to the compatibility of $\skal{\argument}{\argument}_\alpha$
  with $\cc{\argument}$. Moreover, $\Lambda(\cc{v},v) \le
  \seminorm{\alpha}{\cc{v}}\seminorm{\alpha}{v} = \seminorm{\alpha}{v}^2
  = \skal{v}{v}_\alpha = b(\cc{v},v)$ holds for all $v\in V$ and $b$ is
  of Hilbert-Schmidt type because every continuous bilinear form on
  a nuclear space is of Hilbert-Schmidt type (again, see
  \cite[Chap.~21.3, Thm.~5]{jarchow:1981a} or use
  \cite[Satz~28.4]{meise.vogt:1992a}).
  Because of this, Theorem~\ref{theorem:posex} applies.
\end{proof}

\begin{remark}
    \label{remark:TheLastRemark}%
    As Theorem~\ref{theo:posexnuclear} shows the existence of many
    continuous positive linear functionals in the nuclear case, this
    might be the best candidate for applications, because it allows to
    combine most of our results: The space $\SymTen^\bullet(V)^\cpl$
    has a clear interpretation as a space of certain analytic
    functions (Theorem~\ref{theo:gelfand}) and its topology is
    essentially the coarsest possible one
    (Theorem~\ref{theo:char}). The usual equivalences of star products
    that are generated by continuous bilinear forms that differ only
    in the symmetric part still holds due to
    Theorem~\ref{theo:equivalence} and because all symmetric bilinear
    forms on a nuclear space are of Hilbert-Schmidt type. Finally, the
    existence of many continuous positive linear functionals assures
    that there exist non-trivial representations of the deformed
    algebras, in which all elements of up to degree $2$ -- which
    include the most important elements from the point of view of
    physics, e.g. the Hamiltonian of the harmonic oscillator -- are
    represented by essentially self-adjoint operators
    (Theorem~\ref{theo:essselfadj}).  Note that these results are very
    similar to the well-known properties that make $C^*$-algebras
    interesting for applications in physics, even though the topology
    on the algebra that we have considered here is far from $C^*$,
    indeed not even submultiplicative.
\end{remark}

{
  \footnotesize
  \renewcommand{\arraystretch}{0.5}

}

%
%


%
%

\end{document}
